\documentclass[11pt,a4paper]{amsart}
\pdfoutput=1

\usepackage{amsmath,amsfonts,amssymb,amsthm,amscd,stmaryrd}
\usepackage{dsfont,mathrsfs,enumerate}
\usepackage{curves,epic,eepic}

\usepackage[style=numeric-comp]{biblatex}
\usepackage[colorlinks=true, linkcolor=red, citecolor=blue]{hyperref}
\usepackage{url}
\usepackage{cancel}
\usepackage{pdflscape}
\usepackage{float}
\usepackage{makecell}
\usepackage{nicefrac}
\usepackage[section]{placeins}

\usepackage[pdftex]{graphicx}
\usepackage{psfrag}
\usepackage{epic}
\usepackage{eepic}
\usepackage[matrix,arrow,curve]{xy}
\input{epsf}
\usepackage{tikz}
\usetikzlibrary{calc}
\usetikzlibrary{intersections}
\usetikzlibrary{through}
\usetikzlibrary{backgrounds}

\usepackage{algpseudocode}
\usepackage{algorithm}

\newtheorem{thm}{Theorem}[section]
\newtheorem{lemma}[thm]{Lemma}
\newtheorem{prop}[thm]{Proposition}
\newtheorem{coro}[thm]{Corollary}
\newtheorem{conj}[thm]{Conjecture}
\newtheorem{class}[thm]{Classification}
\theoremstyle{definition}
\newtheorem{defi}[thm]{Definition}
\newtheorem{rem}[thm]{Remark}

\newtheorem{con}[thm]{Construction}
\newtheorem{algo}[thm]{Algorithm}

\DeclareMathOperator{\Z}{\mathds{Z}}
\DeclareMathOperator{\Q}{\mathds{Q}}
\DeclareMathOperator{\R}{\mathds{R}}

\DeclareMathOperator{\Vol}{\mathrm{Vol}}
\DeclareMathOperator{\lw}{\mathrm{lw}}

\DeclareMathOperator{\ehr}{\mathrm{ehr}}
\DeclareMathOperator{\ldist}{\mathrm{ldist}}
\DeclareMathOperator{\length}{\mathrm{length}}
\DeclareMathOperator{\denom}{\mathrm{denom}}
\DeclareMathOperator{\width}{\mathrm{width}}
\DeclareMathOperator{\hnf}{\mathrm{hnf}}
\DeclareMathOperator{\unf}{\mathrm{unf}}
\DeclareMathOperator{\anf}{\mathrm{anf}}

\DeclareMathOperator{\sur}{\mathrm{sur}}
\DeclareMathOperator{\convhull}{\mathrm{convhull}}
\DeclareMathOperator{\GL}{\mathrm{GL}}
\DeclareMathOperator{\hilb}{\mathcal{H}}
\DeclareMathOperator{\ldp}{\mathrm{ldp}}
\DeclareMathOperator{\all}{\mathrm{all}}
\newcommand{\conv}[1]{\mathrm{convhull}\left(#1\right)}
\newcommand{\cone}{\mathrm{cone}}

\newcommand{\usim}{\sim^{\mathrm{u}}}
\newcommand{\asim}{\sim^{\mathrm{a}}}
\def\<{\langle}
\def\>{\rangle}

\begin{document}

\title[Rational polygons]{Classifying rational polygons with small denominator
	and few interior lattice points}

\author[Martin Bohnert]{Martin Bohnert}
\address{Mathematisches Institut, Universit\"at T\"ubingen,
	Auf der Morgenstelle 10, 72076 T\"ubingen, Germany}
\email{martin.bohnert@uni-tuebingen.de}

\author[Justus Springer]{Justus Springer}
\address{Mathematisches Institut, Universit\"at T\"ubingen,
	Auf der Morgenstelle 10, 72076 T\"ubingen, Germany}
\email{justus.springer@uni-tuebingen.de}

\begin{abstract}
	We present algorithms for classifying rational polygons with fixed
	denominator and number of interior lattice points. Our approach is
    to first describe maximal polygons and then compute all
    subpolygons, where we eliminate redundancy by a suitable normal
    form. Executing our classification, we obtained a dataset of
    several billions of polygons covering a wide variety of cases.
\end{abstract}

\maketitle

\thispagestyle{empty}

\section{Introduction}

A polytope is called \emph{\( k \)-rational}, if its \( k \)-fold
dilation is a lattice polytope, i.e. has integral vertices. Lagarias
and Ziegler \cite{LZ91} showed that for fixed number \( i \) of
interior lattice points, there exist up to lattice automorphism only
finitely many \( k \)-rational polytopes in each dimension, provided
\( i \geq 1 \). Even though this does not hold for \( i = 0 \),
Averkov, Wagner and Weismantel \cite{AWW11} showed that in this case
there are still only finitely \emph{maximal} \( k \)-rational
polytopes, i.e. which are not strictly included in another \( k
\)-rational polytope with equal number of interior lattice points.

In this article, we develop algorithms to classify polytopes in two
dimensions. We used Julia~\cite{Julia} to carry out the
classification and made the implementation available as part of a
software package on rational polygons~\cite{RationalPolygons_jl}.
Moreover, large parts of our data are available for
download~\cite{springer_2024_13838476,springer_2024_13839216,
	springer_2024_13960791, springer_2024_13838991,
	springer_2024_13928298, springer_2024_13959996}.

\begin{thm}
	\label{thm:main_classification}
	The numbers of \( k \)-rational polygons with \( i \) interior
	lattice points is given by the following table. Each cell contains the
	number of maximal polygons on top, the number of distinct Ehrhart
	quasipolynomials in the middle and the total number of polygons
	below. For \( i = 0 \), only the polygons which cannot be realized in
	\( \R \times [0,1] \) are counted.
    \begin{center}
		{\scriptsize
			\begin{tabular}{c|cccccccc}
				\( k\downarrow i\rightarrow \) & 0             & 1 & 2 & 3 & 4 & 5 & 6 &
				\dots                                                                     \\\hline
				1                               & \makecell{1                             \\ 1 \\ 1} & \makecell{3 \\ 7 \\ 16} &
				\makecell{4                                                               \\ 8 \\ 45}
				                                & \makecell{6                             \\ 10 \\120} &
				\makecell{9                                                               \\12\\211} & \makecell{11\\14\\403} &
				\makecell{13                                                              \\16\\714} &
				\dots                                                                     \\\hline
				2                               & \makecell{4                             \\ 34 \\79} & \makecell{10 \\ 270 \\5\,145} &
				\makecell{25                                                              \\ 586\\48\,639} &
				\makecell{33                                                              \\ 1\,060\\249\,540} & \makecell{63 \\ 1\,701\\893\,402} &
				\makecell{106                                                             \\ 2\,525\\2\,798\,638} & \makecell{178 \\ 3\,577\\7\,299\,589} &
				\dots                                                                     \\\hline
				3                               & \makecell{14                            \\ 803 \\6\,723} & \makecell{39 \\ 8\,124 \\ 924\,042} &
				\makecell{146                                                             \\ 25\,892\\17\,656\,886} &
				\makecell{303                                                             \\ 62\,347 \\156\,777\,687} & \makecell{687
				\\124\,021\\909\,056\,858} &
				\makecell{1\,452
                \\ 218\,169 \\4\,211\,988\,753} \\\hline
				4                               & \makecell{39                            \\ 18\,916 \\399\,294} & \makecell{145 \\ 320\,256
				\\101\,267\,212} & \makecell{670 \\ 1\,593\,475 \\ 3\,452\,922\,966}
				\\\hline
				5                               & \makecell{134                           \\ 253\,631 \\18\,935\,385} & \makecell{698
				\\ 5\,034\,566 \\8\,544\,548\,186}
				\\\hline
				6                               & \makecell{299                           \\ 8\,930\,335\\820\,697\,679}
				\\
			\end{tabular}
		}
    \end{center}
\end{thm}

The first row in the table above is the case of~\emph{lattice
	polygons}, which has already been studied by several authors.
Starting from the topleft, the unique maximal lattice polygon without
interior lattice points is the twofold standard lattice triangle, see
e.g. \cite[sec. 4.1]{Koe91}, \cite[§5]{Kho97} and \cite[Theorem 2]{Sch03}. The classification of the 16 lattice polygons with
exactly one interior lattice point has been done e.g. in
\cites{Bat85, Rab89} and \cite[sec. 4.2]{Koe91}. Lattice polygons
with collinear interior lattice points, which includes the case \( i
= 2 \), were classified by Koelman \cite[sec. 4.3]{Koe91} and a
general recursive algorithm was given by Castryck \cite{Cas12}, which
he ran up to \( i = 30 \). There are also algorithms classifying
lattice polygons by their number of lattice points \cite[sec.
	4.4]{Koe91} and by their area \cite{Bal21}. Pick's formula
\cite{Pic99} relates the area to the numbers of lattice and interior
lattice points, hence one can also obtain the classification by
interior lattice points from those algorithms.

To the authors knowledge, besides the first row in our table, only
the four maximal \( 2 \)-rational polygons without interior lattice
points have directly appeared in the literature before~\cite{AKW17}.
The classification of almost \( k \)-hollow LDP polygons, which was
done in~\cite{HHS23} for \( k \in \{2,3\} \) is also worth
mentioning, as these can be viewed as a subset of all \( k
\)-rational polygons with exactly one interior lattice point. Hence
they appear in our classification among the \( 5\,145 \) resp. \(
924\,042 \) polygons in the second column of the table above.

Let us also mention that in three dimensions, the lattice polytopes
with one and two interior lattice points have been
classified~\cite{Kas10,BK16} and the case of no interior lattice
points has been done in~\cite{AWW11,AKW17}. Moreover, there is the
classification of four-dimensional reflexive polytopes \cite{KS00},
which is an important subset of those having exactly one interior
lattice point.

This article is organised as follows. In
section~\ref{sec:prerequesites_for_classifying_rational_polygons}, we
go over some topics that are neccessary for our classifications,
namely normal forms, computing subpolygons and the notion of
maximality. Sections~\ref{sec:rational_polygons_in_m1p1},
\ref{sec:rational_polygons_without_interior_lattice_points}
and~\ref{sec:rational_polygons_with_exactly_one_interior_lattice_points}
describe our classification, first for the case of polygons contained
in \( \R \times [-1,1] \), then for polygons with no and exactly one
interior lattice point, filling out the first two columns of the
table in Theorem~\ref{thm:main_classification}. The methods we use
are mainly constraints on the lattice width data developed
in~\cite{Boh23}. Section~\ref{sec:half_integral_polygons} is about
the case \( k = 2 \). There, we were able to generalize the idea of
moving out the edges from Castryck's classification~\cite{Cas12}
using the classification of two-dimensional Fine interiors of
three-dimensional lattice polytopes~\cite{Boh24}. Finally, in
section~\ref{sec:general_case}, we discuss a generic (yet
inefficient) method of classifying rational polygons for any pair \(
(k,i) \in \Z_{\geq 1}^2 \) that we used to fill in the rest of the
table in Theorem~\ref{thm:main_classification}.

\section{Prerequesites for classifying rational polygons}
\label{sec:prerequesites_for_classifying_rational_polygons}

By a \emph{polygon}, we mean a convex bounded polyhedral set \( P
\subseteq \R^2 \). We call a polygon \emph{rational}
(\emph{integral}), if its vertices are contained in \( \Q^2 \) (\(
\Z^2 \)). Integral polygons will also be called \emph{lattice
	polygons}. For \( k \in \Z_{\geq 1} \), we say \( P \) is \emph{\( k
	\)-rational}, if \( kP \) is integral. The \emph{denominator} \(
\denom(P) \) of a rational polygon is the smallest positive integer
\( k \) such that \( P \) is \( k \)-rational. 

We call two polygons \( P \) and \( P' \) \emph{unimodular
	equivalent} (\( P \usim P' \)) if \( P' = UP \) for some \( U \in
\GL(2,\Z) \). We call them~\emph{affine unimodular equivalent} (\( P
\asim P'\)), if \( P' = UP+b \) for some \( U \in \GL(2,\Z) \) and \(
b \in \Z^2 \). We say that \( P \) can be~\emph{realized} in some
subset \( A \subseteq \R^2 \), if \( P \asim Q \) for some \( Q
\subseteq A \). We call \( P \) a \emph{subpolygon} of \( P' \), if
\( P \) can be realized in~\( P' \).

	We write \( \Vol(P) \)
	for the \emph{normalized volume} of a polygon, which is defined to be
	twice its euclidian volume. Moreover, we define the \emph{\( k
		\)-normalized} volume as \(\Vol_k(P) := k^2\Vol(P)\). Note that if \(
	P \) is \( k \)-rational, \( \Vol_k(P) \) is an integer. We define
	the number of interior, boundary and total number of lattice points
	as:
	\[
	i(P)\ :=\ |P^{\circ} \cap \Z^2|,\qquad b(P)\ :=\ |\partial P \cap \Z^2|,\qquad l(P)\ :=\ |P \cap \Z^2|.
	\]
	
All these are invariants under (affine) unimodular equivalence and
there exist several relations between them, e.g. Pick's formula
states that \( \Vol(P) = 2i(P) + b(P) - 2 \) holds for lattice
polygons~\cite{Pic99}. Another important invariant is the
\emph{Ehrhart quasipolynomial}, which counts the numbers of lattice
points in integral dilations of a rational polygon, i.e.
\[
	\ehr_P(t)\ :=\ l(tP)\ =\ At^2+c_1(t)t+c_2(t), \qquad t \in \Z_{\geq
		1},
\]
where \( A \) is the euclidian volume of \( P \) and \( c_1,c_2 : \Z
\to \Q \) are uniquely determined periodic functions whose period
divides \( \denom(P) \). For lattice polygons, \( c_1(t) =
\nicefrac{b(P)}{2} \) and \( c_2(t) = 1 \) holds for all \( t \). For
the existence of the Ehrhart quasipolynomial and more insights on
Ehrhart theory see for example \cite{BS15}.

\subsection{Normal forms}
\label{subsec:normal_forms}

Since we want to classify polygons up to lattice automorphism, we
need an way to decide whether two polygons are equivalent by an
affine unimodular transformation, which can be done by introducing a
normal form. PALP \cite{KS04}, see also \cite{GK13}, is such a normal
form that works for lattice polytopes of any dimension. However, we
argue that the two-dimensional case deserves to be treated
seperately. This is because the vertices of a polygon are ordered
naturally by walking along its boundary, up the choice of the
starting point and the direction. In this subsection, we present our
approach to the normal form of rational polygons, which is based on
this natural ordering of the vertices. An implementation of our
normal form is available in~\texttt{RationalPolygons.jl}
\cite{RationalPolygons_jl}.

We begin by describing how we encode \( k \)-rational polygons by
integral matrices.

\begin{defi}
	\label{def:vertex_matrix}
	Fix \( n \in \Z_{\geq 1} \). By a \emph{vertex matrix}, we
	mean an integral matrix \( V = [v_1\ \dots\ v_n] \in \Z^{2\times
		n} \) such that \( v_1, \dots, v_n \) are the vertices of \(
	\convhull(v_1,\dots,v_n) \) and \( v_i \) is connected to \(
	v_{i+1} \) by an edge for all \( i = 1, \dots, n - 1 \).
	For \( k \in \Z_{\geq 1} \), we
	obtain a \( k \)-rational polygon
	\[
		P_{k,V}\ :=\ \convhull(v_1 / k, \dots, v_n / k).
	\]
\end{defi}

\begin{rem}
	\label{rem:convex_hull_algorithms}
	Using a planar convex hull algorithm, we can reduce any finite set
	of lattice points in the plane to a set of vertices ordered by
	angle. For example, a Graham scan \cite{Gra72} achieves this
	with time complexity \( O(n\log n) \).
\end{rem}

\begin{defi}
	\label{def:vertex_matrix_reverse_and_shift}
	For a vertex matrix \( V = [v_1\ \dots\ v_n] \in \Z^{2\times n} \),
	we obtain new vertex matrices by reversing or shifting the order
	of its
	columns:
	\[
		V^{(-1)}\ :=\ [v_n\ \dots\ v_1], \qquad V_i := [v_i\ \dots
				v_n\ v_1\ \dots v_{i-1}].
	\]
	Moreover, we write \( V + w := [v_1+w\ \dots\ v_n+w] \). Two vertex
	matrices \( V, V' \in \Z^{2\times n} \) are called

	\begin{enumerate}[(i)]
		\item \emph{strictly equivalent} (\( V \sim V' \)), if \( V' = V^{(s)}_i \) for some \( s
		      \in \{ \pm 1\} \) and \( i = 1, \dots, n \),
		\item \emph{unimodular equivalent} (\( V \usim V' \)), if \(
		      V' \sim UV \) for some \( U \in \GL(2,\Z) \),
		\item \emph{\( k \)-affine unimodular equivalent} (\( V \asim_k V' \)),
		      if \( V' \sim UV + b \) for some \( U \in \GL(2,\Z) \) and \(
		      b \in k\Z^2 \).
	\end{enumerate}
\end{defi}

\begin{rem}
	\label{rem:unimodular_equivalence_by_vertex_matrices}
	Let \( V,V' \in \Z^{2\times n} \) be vertex matrices and \( k
	\in \Z_{\geq 1} \). Then we have
	\begin{enumerate}[(i)]
		\item \( P_{k,V} = P_{k,V'} \) if and only if \( V \sim V' \),
		\item \( P_{k,V} \usim P_{k,V'} \) if and only if \( V \usim
		      V' \),
		\item \( P_{k,V} \asim P_{k,V'} \) if and only if \( V \asim_k
		      V' \).
	\end{enumerate}
\end{rem}

\begin{defi}
	\label{def:unimodular_and_affine_normal_form}
	Let \( V = [v_1\ \dots\ v_n] \) be a vertex matrix.
	\begin{enumerate}[(i)]
		\item The \emph{unimodular normal form} of \( V \) is
		      \[
			      \unf(V)\ :=\ \max\{\hnf(V_i^{(s)})\ \mid\ s \in \{\pm 1\},\
			      i = 1, \dots, n\ \},
		      \]
		      where the maximum is taken by lexicographical ordering
              and we write \( \hnf(M) \) for the hermite normal form of an
		      integral matrix \( M \).
		\item The \emph{1-affine normal form} of \( V \) is
		      \[
			      \anf_1(V)\ :=\ \max\{\hnf(V_i^{(s)} - v_i)\ \mid\ s \in \{\pm 1\},\ i
			      = 1, \dots, n \}.
		      \]
		\item Let \( k \in \Z_{>1} \) and consider the lexicographically smallest
		      vector \( (x,y) \in \{0, \dots, k-1\}^2 \) such that \( \anf_1(V) +
		      (x,y) \asim_k V \). Then the \emph{k-affine normal form} of \( V \)
		      is
		      \[
			      \anf_k(V)\ :=\ \anf_1(V) + (x,y).
		      \]
	\end{enumerate}

\end{defi}

\begin{prop}
	\label{prp:normal_form_correct}
	Let \( k \in \Z_{\geq 1} \) and consider vertex matrices \( V, V'
	\in \Z^{2\times n} \). Then
	\begin{enumerate}[(i)]
		\item \( V \usim V' \) if and only if \( \unf(V) = \unf(V')
		      \),
		\item \( V \asim_1 V' \) if and only if \( \anf_1(V) =
		      \anf_1(V') \),
		\item \( V \asim_k V' \) if and only if \( \anf_k(V) =
		      \anf_k(V') \).
	\end{enumerate}
\end{prop}
\begin{proof}

	Note that we have \( V \usim \unf(V) \) as well as \( \unf(UV) =
	\unf(V) \) and \( \unf(V_i^{(s)}) = \unf(V) \) for all \( U \in
	\GL(2,\Z) \) and \(i \in \{1, \dots, n\}, s \in \{\pm 1\} \). Thus
	(i) follows and (ii) is handled analogously. For (iii), we again have
	\( V \asim_k \anf_k(V) \), hence \( \anf_k(V) = \anf_k(V') \) implies
	\( V \asim_k V' \). For the converse, note that \( V \asim_k V' \)
	implies \( V \asim_1 V' \), hence by (ii), we have \( \anf_1(V) =
	\anf_1(V') \). But this means \( \anf_1(V) + (x,y) \asim_k V \) if
	and only if \( \anf_1(V') + (x,y) \asim_k V' \) for all \( (x,y) \in
	\Z^2 \), hence \( \anf_k(V) = \anf_k(V') \).
\end{proof}

\begin{rem}
	\label{rem:improve_normal_form_computation}
	Using Definition~\ref{def:unimodular_and_affine_normal_form},
	computing the unimodular or affine normal form requires computing \(
	2n \) hermite normal forms, where \( n \) is the number of vertices.
	This can be improved by first distinguishing a subset of vertices
	that maximize a suitable invariant. For example, we can call a vertex
	\( v_i \) \emph{special}, if it maximizes the area \(\det(v_{i+1} -
	v_i, v_i - v_{i-1}).\) Then we obtain an alternative unimodular
	normal form by
	\[
		\unf'(V)\ :=\ \max\{\hnf(V^{(s)}_i)\ \mid\ s \in \{\pm 1\},\ v_i\
		\text{special} \}.
	\]
	For polygons with few symmetries, there tend to be few special
	vertices, hence for those polygons \( \unf'(V) \) is faster to
	compute than \( \unf(V) \). More sophisticated invariants lead to
	fewer special vertices on average, however then the computation of
	the invariant itself becomes more expensive.
\end{rem}

\begin{rem}
	When computing the unimodular or affine normal form of a rational
	polygon, we get its respective automorphism group for free: Consider
	the set \( I \) of index pairs \( (i,s) \) with \( i = 1, \dots, n \)
	and \( s \in \{\pm 1\} \) such that \( \hnf(V^{(s)}_i) = \unf(V) \).
	Then the number of rotational symmetries of the polygon is equal to
	the cardinality of \( I \). Moreover, there exists a reflectional
	symmetry if and only if \( I \) contains two index pairs with
	opposite second coordinate. In this case, the unimodular automorphism
	group is the dihedral group \( D_{|I|/2} \). Otherwise, it is
	the cyclic group of order \( |I| \).
\end{rem}

\subsection{Subpolygons}
\label{subsec:subpolygons}

We address the problem of finding all subpolygons of a given rational
polygon up to equivalence. We use the same approach of
\emph{succesive shaving} that the authors of \cite{BK13} used to
classify lattice polygons contained in a square of given size. In our
implementation, we noticed a significant speed-up by using hilbert
bases in the shaving process, hence we describe this in detail. As an
application, we reproduce and extend the classification
of~\cite{BK13}, see~\ref{class:polygons_in_box}.

Fix \( k \in \Z_{\geq 1} \) and let \( P \) be a \( k \)-rational
polygon. Recall that a \emph{subpolygon} of \( P \) is a \( k
\)-rational polygon \( Q \) such that \( Q \asim Q' \) for some \( Q'
\subseteq P \). Up to affine unimodular equivalence, there exist only
finitely many subpolygons to any given \( k \)-rational polygon. For
a vertex \( v \), we obtain a subpolygon by taking the convex hull of
all \( k \)-rational points of \( P \) except \( v \):
\[
	P \backslash v\ :=\ \convhull\left(\left(P \cap
		\frac{1}{k}\Z^2\right) \backslash \{v\}\right).
\]

\begin{prop}
	\label{prp:subpolygons_exists_chain}
	Fix \( k \in \Z_{\geq 1} \) and let \( P \) and \( Q \) be \( k \)-rational
	polygons with \( Q \) a subpolygon of \( P \). Then there exists a chain of
	\( k \)-rational polygons
	\[
		P\ =\ P_0\ \supsetneq\ P_1\ \supsetneq\ \dots\ \supsetneq\ P_n\ \asim\ Q
	\]
	such that for all \( i = 0, \dots, n-1 \), we have \( P_{i+1} = P_i
	\backslash v_i \) for some vertex \( v_i \) of \( P_i \).
\end{prop}
\begin{proof}
	We may assume that \( Q \subsetneq P \) holds. Note that the
	number of \( k \)-rational points of \( P \) must be strictly
	greater than the number of \( k \)-rational points of \( Q \).
	Let's consider the case that they differ by one, i.e. there is a
	unique \( k \)-rational point \( v
	\in \frac{1}{k}\Z \) with \( v \in P \backslash Q \). By
	convexity, \( v \) must be a vertex, hence we have \( Q = P
	\backslash v \). The general case now follows by induction.
\end{proof}

\begin{algorithm}[H]
	\caption{Computation of subpolygons}
	\label{algo:subpolygons}
	\begin{algorithmic}[1]
		\Require A set of \( k \)-rational polygons \( M \).
		\Ensure A set of \( k \)-rational polygons \( S \) such that for
		all \( P \in M \) and all subpolygons \( Q \) of \( P \), there exists
		precisely one \( Q' \in S \) such that \( Q \asim Q' \).
		\Procedure{subpolygons}{M}
		\State {\( A\ :=\ \max\{\Vol_k(P)\ \mid\ P \in M\} \)}
		\For {\( a \) from \( 1 \) to \( A \)}
		\State {\( S_a := \{P \in M\ \mid\ \Vol_k(P) = a\} \)}
		\EndFor
		\For {\( a \) from \( A \) to \( 1 \)}
		\For {\( P \in S_a \)}
		\For {\( v \) vertex of \( P \)}
		\State {\( Q\ :=\ \anf_k(P \backslash v)\)}
		\label{line:normal_form_and_remove_vertex}
		\State {\( a'\ :=\ \Vol_k(Q) \)}
		\If {\( \dim(Q) = 2 \) and \( Q \notin
			S_{a'} \)}
		\State {add \( Q \) to \( S_{a'} \)}
		\EndIf
		\EndFor
		\EndFor
		\EndFor
		\State \Return {\( S_3 \cup \dots \cup S_A \)}
		\EndProcedure
	\end{algorithmic}
\end{algorithm}
\begin{proof}
	Proposition~\ref{prp:subpolygons_exists_chain} ensures that the
	algorithm encounters every subpolygon at least once. By
	Propisition \ref{prp:normal_form_correct}, the resulting set
	contains each polygon precisely once up to affine unimodular equivalence.
\end{proof}

In line~\ref{line:normal_form_and_remove_vertex} of
Algorithm~\ref{algo:subpolygons}, we need to compute the polygon \( P
\backslash v \). This can be done by first determining all \( k
\)-rational points of \( P \), removing \( v \) and computing the
convex hull of the remaining points (for example, by performing a
Graham scan). In the rest of this section, we show how to
significantly speed this up using an algorithm involving hilbert
bases. First, let us quickly summarize the neccessary facts around
hilbert bases and their computation in two dimensions using continued
fractions. We refer to~\cite[Chapter 10]{CLS11} for details.

Let \( \sigma \subseteq \R^2 \) be a rational pointed cone. After
applying a unimodular transformation, we may assume \( \sigma =
\cone(e_2, de_1-ke_2) \) where \( 0 \leq k < d \) are coprime
integers. The \emph{hilbert basis} \( \hilb_{\sigma} \) of \( \sigma
\) is the unique minimal generating set of the monoid \( \sigma \cap
\Z^2 \). It can be computed as follows: First, using the modified
euclidian algorithm, we express \( d/k \) as a Hirzebruch-Jung
continued fraction \( \llbracket b_1, \dots, b_r \rrbracket \), which
can be defined succinctly as
\[
	\llbracket b \rrbracket\ :=\ b, \qquad\qquad \llbracket b_1, \dots, b_r
	\rrbracket\ :=\ b_1 - \frac{1}{\llbracket b_2, \dots, b_r \rrbracket}.
\]
Given \( d/k = \llbracket b_1, \dots, b_r \rrbracket \), the Hilbert
basis is given by \( \hilb_{\sigma} = \{u_0, \dots, u_{r+1}\} \),
where
\[
	u_0\ :=\ e_2, \qquad u_1\ :=\ e_1, \qquad u_{i+1}\ :=\
	b_iu_i-u_{i-1}.
\]

\begin{prop}
	\label{prp:hilbert_basis_eq_vertices_of_convhull}
	Let \( \sigma \subseteq \R^2 \) be a rational pointed cone with
	hilbert basis \( \hilb_{\sigma} = \{u_0, \dots, u_{r+1}\} \) as
	constructed above. Consider the unbounded polyhedron
    \[
        \Theta_{\sigma}\ :=\ 
        \convhull((\sigma \cap \Z^2) \backslash \{0\}).
    \]
    Then the following holds.
	\begin{enumerate}[(i)]
		\item \( \hilb_{\sigma} \) is the set of lattice points on the
		      bounded edges of \( \Theta_{\sigma} \).
		\item \( \hilb'_{\sigma} := \{u_i \in \hilb_{\sigma}\ \mid\
		      b_i \neq 2\} \) is the set of vertices of \(
		      \Theta_{\sigma} \).
	\end{enumerate}
\end{prop}
\begin{proof}
	For (i), we refer to \cite[Theorem 10.2.8 (b)]{CLS11}. For part
	(ii), note that we have \( u_{i+1}-u_i = (b_i-1)u_i-u_{i-1} \) for
	all \( i=1, \dots, r \).
	Hence \( u_i \) lies on the line segment between \( u_{i-1} \) and
	\( u_{i+1} \) if and only if \( b_i=2 \).
\end{proof}

\begin{coro}
	\label{cor:vertices_of_removed_edge_polygon}
	Let \( P \) be a lattice polygon with vertices \( v_1, \dots, v_n
	\in \Z^2 \) and assume \( v_j=0 \). With \(
	\sigma=\cone(v_{j-1},v_{j+1}) \) and \( \hilb_{\sigma}' \) as in
	the proposition, the vertices of \( P \backslash v_j \) are given
	by
	\[
		(\{v_1, \dots, v_n\} \backslash \{v_j\})\ \cup\ \hilb_{\sigma}'.
	\]
\end{coro}

\begin{rem}
	\label{rem:vertex_ordering_hilbert_basis}
	The vertices \( \hilb_{\sigma}' = \{u_1', \dots, u_s'\} \) are
	already ordered naturally, i.e. \( u_i' \) and \( u_{i+1}' \) are
	connected by an edge. This means that if we start with ordered
	vertices \( v_1, \dots, v_n \) of \( P \),
	Corollary~\ref{cor:vertices_of_removed_edge_polygon} allows us to
	easily write down the ordered vertices of \( P \backslash v_j \),
	which ties in nicely with our normal form.
\end{rem}

Note that the condition \( v_j = 0 \) in
Corollary~\ref{cor:vertices_of_removed_edge_polygon} is not a
restriction, since we can translate any lattice polygon to move \(
v_j \) to the origin. Moreover, for any \( k \)-rational polygon \( P
\), we can apply the Corollary to \( kP \) to compute the vertices of
\( P \backslash v \) for any vertex \( v \) of \( P \).

Using this improved algorithm for determining subpolygons, we were
able to reproduce and extend a classification by Brown and Kasprzyk,
see~\cite[Table 1]{BK13}.

\begin{class}[Data available at~\cite{springer_2024_13838476}]
	\label{class:polygons_in_box}
	Below are the numbers \( \#(m) \) of lattice subpolygons of \( [0,m] \times
	[0,m] \) that are not subpolygons of \( [0,m-1] \times [0,m-1] \),
	up to affine unimodular equivalence. We also include the maximally
	attained number of vertices \( N(m) \) and the number of
	vertex maximizers \( M(m) \).
	\begin{center}
		{\small
			\begin{tabular}{c|cccccc}
				\( m \)       & 1                & 2           & 3            & 4      & 5       & 6        \\\hline
				\( \#(m) \)   & 2                & 15          & 131          & 1\,369 & 13\,842 & 129\,185 \\
				\( N(m) \)    & 4                & 6           & 8            & 9      & 10      & 12       \\
				\( M(m) \)    & 1                & 1           & 1            & 1      & 15      & 2        \\\hline\hline
				\( m \)       & 7                & 8           & 9            & 10     & 11      &          \\\hline
				\( \#(m) \)   & 1\,104\,895      & 8\,750\,964 & 64\,714\,465 &
				450\,686\,225 & 2\,976\,189\,422 &                                                          \\
				\( N(m) \)    & 13               & 14          & 16           & 16     & 17                 \\
				\( M(m) \)    & 3                & 13          & 1            & 102    & 153                \\
			\end{tabular}
		}
	\end{center}
\end{class}

\begin{rem}
	\label{rem:subpolygons_with_same_number_of_interior_points}
	In our application, we are rarely interested in all subpolygons of
	a \( k \)-rational polygon \( P \), but only in those sharing the
	same number of interior lattice points with \( P \). This can be
	integrated rather easily into our Algorithm~\ref{algo:subpolygons}:
	Whenever we compute the vertices of \( P \backslash v \) using
	Corollary~\ref{cor:vertices_of_removed_edge_polygon}, we check
	whether \( \hilb_{\sigma}' \cap k\Z^2 \neq \emptyset \). If
	so, that means an interior lattice point of \( P \) has moved to
	the boundary of \( P \backslash v \). Hence \( P \backslash v \)
	has stricly fewer interior lattice points than \( P \) and we can
	discard it immediately.
\end{rem}

\subsection{Maximal polygons}
\label{subsec:maximal_polygons}

By a \emph{\( k \)-maximal} polygon, we mean a \( k \)-rational
polygon \( P \) such that for any \( k \)-rational polygon \( Q
\supseteq P\) with \( i(P) = i(Q) \), we have \( P = Q \). Clearly,
every \( k \)-rational polygon is a subpolygon of a \( k \)-maximal
one. This means if we classify all \( k \)-maximal polygons with \( i
\in \Z_{\geq 0} \) interior lattice points, we can use
Algorithm~\ref{algo:subpolygons} to obtain all \( k \)-rational
polygons with \( i \) interior lattice points. In
Proposition~\ref{prp:k_maximal_criterion}, we give a criterion for \(
k \)-maximality that can be checked efficiently. Let us first fix
some notation.

A \( k \)-rational polygon \( P \) can be written as the intersection
of the affine halfplanes \( H_e \) associated with its edges, where
\[
	H_e = \{x \in \R^2\ \mid\ \langle
	n_e, x\rangle \geq c_e \}, \qquad n_e \in \Z^2\ \text{primitive},
	\qquad c_e \in \Q.
\]
For any \( x \in \R^2 \), we define the \emph{lattice distance} of
the edge \( e \) with \( x \) as \( \ldist_e(x) := \langle n_e,x
\rangle - c_e \). Hence \( H_e \) consists of those points having
nonnegative lattice distance with \( e \). For any rational number \(
q \in \Q \), we define the \emph{translated halfplane} by
\[
	H_e + q\ :=\ \{x \in \R^2\ \mid\ \ldist_e(x) \geq q \}.
\]
By moving out the affine halfplanes, we obtain a new rational polygon
\(P^{(-1)}:=\bigcap_e \left(H_e - \frac{1}{k}\right)\), where the
intersection runs through all edges of \( P \).

\begin{lemma}
	\label{lem:empty_triangles_lattice_distance}
	Let \( \Delta = \convhull(x,y,z) \subseteq \R^2 \) be a lattice
	triangle containing no lattice points other than its vertices.
	Then \( \ldist_e(z) = 1 \), where \( e \) is the edge connecting
	\( x \) and \( y \).
\end{lemma}
\begin{proof}
	By Pick's formula, \( \Delta \) has area \( i(\Delta) +
	\frac{b(\Delta)}{2} - 1 = \frac{1}{2} \). This forces \(
	\ldist_e(z) = 1 \).
\end{proof}

\begin{prop}
	\label{prp:k_maximal_criterion}
	For a \( k \)-rational polygon \( P \), the following are
	equivalent:
	\begin{enumerate}[(i)]
		\item \( P \) is \( k \)-maximal,
		\item for all \( v \in \partial P^{(-1)} \cap \frac{1}{k}\Z^2 \), we have \(
		      i(\convhull(P\cup\{v\})) > i(P) \),
		\item for all \( v \in \partial P^{(-1)} \cap \frac{1}{k}\Z^2 \), one of the following holds:
		      \begin{enumerate}[(1)]
			      \item There exists an edge \( e \) of \( P \) and a point \( w \in
			            e^{\circ} \cap \Z^2 \) such that \( v \notin H_e \),
			      \item There exists a vertex \( w \in \Z^2 \) of \( P \) such that \( v
			            \notin H_e \) and \( v \notin H_{e'} \), where \( e \) and \( e' \)
			            are the edges adjacent to \( w \).
		      \end{enumerate}
	\end{enumerate}
\end{prop}
\begin{proof}
	``(i) \( \Rightarrow \) (ii)'' is clear. To show ``(ii) \(
	\Rightarrow \) (i)'', assume \( P \) is not \( k \)-maximal. Then we
	have \( P \subsetneq Q \) for some \( k \)-rational \( Q \) with
	\( i(P) = i(Q) \). Let \( z \in Q \) be a vertex and \( e \) an
	edge of \( P \) such that \( z \notin H_e \). Pick two \( k
	\)-rational points \( x,y \in e \) containing no \( k \)-rational
	points on the line segment between them and consider the triangle
	\( \Delta := \convhull(x,y,z) \). Let \( z' \in \Delta \) be a \(
	k \)-rational point with minimal lattice distance to \( e \). Then
	\( \Delta' := \convhull(x,y,z') \) contains no \( k \)-rational
	points other than its vertices, hence Lemma
	\ref{lem:empty_triangles_lattice_distance} implies \( \ldist_e(z')
	= -\frac{1}{k} \). We obtain
	\[
		i(P)\ \leq\ i(\convhull(P \cup \{z'\}))\ \leq\ i(Q).
	\]
	Since \( i(P) = i(Q) \), we get equality, which contradicts (ii).

	To see that (ii) implies (iii), note that \( Q := \convhull(P \cup
	\{v\}) \) must contain a lattice point \( w \in \partial P \cap \Z^2 \) in its interior. If \( w \) lies in the relative
	interior of an edge of \( P \), we are in case~(1). If \( w \) is a
	vertex of \( P \), we are in case~(2). The converse (iii) \(
	\Rightarrow \) (ii) is immediate.

\end{proof}

\begin{rem}
	\label{rem:maximal_convex_set}

	A \( k \)-maximal polygon is not necessarily maximal as a convex set,
	i.e. there could be a convex set with the same number of interior
	lattice points which strictly contains the \( k \)-maximal polygon. A
	\( k \)-rational polygon is maximal as a convex set if and only if
	every edge contains an interior lattice point. For example, the
	quadrilateral
	\begin{displaymath}
		\convhull\left(\left(0,1+\frac{1}{k}\right), (0,0),
		(k(i+1)+i,0), \left(\frac{1}{k},1+\frac{1}{k}\right)\right), \quad
		i \in \Z_{\geq 0}, k \in \Z_{\geq 1}
	\end{displaymath}
	is \( k \)-maximal for \( (i,k)\notin \{(0,1), (0,2), (1,1)\} \) by
	Proposition~\ref{prp:k_maximal_criterion}, but
	it is not a maximal convex set since there is no interior
	lattice point on the edge between \(\left(0,1+\frac{1}{k}\right)\)
	and \(\left(\frac{1}{k},1+\frac{1}{k}\right)\).
\end{rem}

We end this section by showing that for lattice polygons, our notion
of maximality coincides with the one used by~\cite{Koe91}
and~\cite{Cas12}. Recall that the \emph{interior hull} of a rational
polygon is \( P^{(1)} := \convhull(P^{\circ} \cap \Z^2) \).

\begin{prop}
	\label{prp:lattice_polygon_maximal_criterion}
	Let \( P \) be a lattice polygon with two-dimensional interior
	hull \( P^{(1)} \). Then \( P \) is \( 1 \)-maximal if and only if
	\( P = P^{(1)(-1)} \).
\end{prop}
\begin{proof}
	Assume \( P \) is \( 1 \)-maximal. Then by~\cite[Theorem
		5]{Cas12}, \( P^{(1)(-1)} \) is a lattice polygon containing \( P
	\). Clearly, we have \( i(P) = i(P^{(1)(-1)}) \), hence we obtain
	\( P = P^{(1)(-1)} \). Conversely, if \( P = P^{(1)(-1)} \), pick
	a \( 1 \)-maximal polygon \( Q \) containing \( P \) with \(
	i(P)=i(Q) \). Then we must have \( Q^{(1)} = P^{(1)} \), hence \(
	Q = Q^{(1)(-1)} = P^{(1)(-1)} = P \).
\end{proof}

\section{Rational polygons in \texorpdfstring{\( \R\times [-1,1]
    \)}{R x [-1,1]} or with collinear
  interior points}
\label{sec:rational_polygons_in_m1p1}

In this short section, we will classify \( k \)-maximal polygons that
can be realized in \( \R \times [-1,1] \) with given number of
interior lattice points. This case occurs in our classifications of
\( k \)-maximal polygons with no interior lattice points in
section~\ref{sec:rational_polygons_without_interior_lattice_points}
and exactly one interior lattice point in
section~\ref{sec:rational_polygons_with_exactly_one_interior_lattice_points},
hence we treat it here seperately. As an application, we classify \(
k \)-rational polygons with collinear interior lattice points.

\begin{prop}
	\label{polygons_in_m11}
	Every rational polygon $P \subseteq \R \times [-1,1]$ with $i \in
		\Z_{\geq 0}$ interior lattice points can be realized in
	\(\R\times[-1,0]\ \cup\ T\), with the half-open trapezoid
	\[
		T := \convhull\left((0,0), (i+1,0), (i+2,1), (0,1)\right)
		\setminus (\R \times \{0\}).
	\]
\end{prop}

\begin{proof}
	After a suitable translation, we can assume the intersection of the
	interior of \( P \) with \( \R \times \{0\} \) is contained in \(
	(0,i+1) \times \{0\} \). Furthermore, we can assume this intersection
	to be non-empty, otherwise we can realize \( P \) in \( \R
	\times [-1,0] \). Consider the points \(a_1 := (0,0) \) and \( a_2 :=
	(i+1,0).\) Then we find unique supporting lines \( L_1 \) and \( L_2
	\) of \( P \) with \( a_j \in L_j \) such that \( L_1 \) and \( L_2
	\) both contain a point of \( P \) with positive second coordinate.
	Since we can reflect on \( \R \times \{0\} \), we can assume \( L_1
	\) and \( L_2 \) to be either parallel or intersect in a point of \(
	\R \times (0,\infty) \). Therefore, after a suitable horizontal
	shearing, we can assume that \( L_1 \) intersects \( \R\times \{1\}
	\) in a point of \( [0,1] \times \{1\} \) and \( L_2 \) intersects \(
	\R \times \{1\} \) in a point of \( (-\infty, i+2] \times \{1\} \).
	But this implies \( P \subseteq \R \times [-1,0] \cup T \).
\end{proof}

\begin{figure}[H]
	\caption{
		All rational polygons in \( \R \times [-1,1] \) with \( i
		\) interior lattice points can be realized in \( \R \times [-1,0] \cup T \)}
	\label{fig:classify_maximal_polygons_m1p1}
	\begin{tikzpicture}[x=1cm,y=1cm]

		\draw[step=1,black,line width=1pt,xshift=0cm,yshift=0cm,
			opacity=0.3] (-0.9,0.3) grid (6.9,-2.32);

		\fill[line width=2pt,color=black,fill=black,fill
			opacity=0.2] (1,0) -- (6,0) -- (5,-1) -- (7,-1) --
		(7,-2) -- (-1,-2) -- (-1,-1) -- (1,-1) -- cycle;

		\draw [line width=2pt,color=black] (-1,-1) -- (1,-1) -- (1,0) -- (6,0) -- (5,-1) -- (7,-1);
		\draw [line width=2pt,color=black] (-1,-2) -- (7,-2);
		\draw [line width=2pt,color=black,dotted] (1,-1) -- (5,-1);

		\draw [fill=black] (1,0) circle (3pt);
		\draw [fill=black] (1,-1) circle (3pt) node[below]{$a_1$};
		\draw [fill=black] (6,0) circle (3pt);
		\draw [fill=black] (5,-1) circle (3pt) node[below]{$a_2$};

		\node at (2.5,-0.5) {$T$};

		\node at (7.3,0) {1};
		\node at (7.3,-1) {0};
		\node at (7.3,-2) {-1};

		\node at (0,-2.6) {-1};
		\node at (1,-2.6) {0};
		\node at (2,-2.6) {1};
		\node at (3,-2.6) {\dots};
		\node at (4,-2.6) {$i$};
		\node at (5,-2.6) {$i+1$};
		\node at (6,-2.6) {$i+2$};

	\end{tikzpicture}

\end{figure}
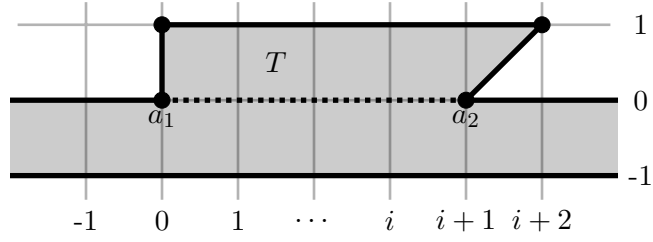

\begin{con}
	\label{cns:surrounding_polygon_m1p1}
	Consider a \( k \)-rational polygon \( P \subseteq \R \times
	[-1,0] \cup T \) with \( P \cap T \neq \emptyset \).
	Then there are supporting lines \( L_1 \) and \( L_2 \) as
	in the proof of the Proposition above. Consider the associated affine
	halfplanes \( H_1 \) and \( H_2 \) with \( P
	\subseteq H_j \). We define the
	\emph{surrounding polygon} of \( P \) to be the \( k \)-rational
	polygon
	\[
		P_{\sur}\ :=\ \convhull\left(Q \cap
		\frac{1}{k}\Z^2\right),\quad \text{where}\ Q\ :=\
		\R\times [-1,1]\ \cap\ H_1 \cap H_2.
	\]
	Note that we have \( P \subseteq P_{\sur} \subseteq \R \times [-1,0]
	\cup T \) and \( i(P) = i(P_{\sur}) = i \). In particular, if \( P \)
	is \( k \)-maximal, we have \( P = P_{\sur} \).
\end{con}

\begin{defi}
	\label{def:halfplane_by_points}
	For two distinct points \( v,w \in \R^2 \), we write \( H(v,w) \) for
	the closed affine halfplane given by all points to the right of the line
	through \( v\) and \( w \) with respect to the direction vector \( w
	- v \)
\end{defi}

\begin{algo}
	\label{algo:classify_maximal_polygons_m1p1}
	Given integers \( k \in \Z_{\geq 1} \) and \( i \in \Z_{\geq 0}
	\), we can classify all \( k \)-maximal polygons with \( i \)
	interior lattice points in \( \R \times [-1,1] \) as follows: Consider
	\( a_1 := (0,0) \) and \( a_2 := (i+1,0) \). To any pair
	of points \( x_1,x_2 \in T \), we associate the polygon
	\[
		P := \convhull\left(Q \cap \frac{1}{k}\Z^2\right), \quad
		\text{where}\ Q := \R \times [-1,1] \cap H(a_1,x_1) \cap
		H(x_2,a_2)
	\]
	This defines a possible surrounding polygon. Going through all
	polygons that arise this way and filtering out those that are \( k
	\)-maximal (for instance, using
	Proposition~\ref{prp:k_maximal_criterion}) gives us all \( k
	\)-maximal in \( \R \times [-1,1] \) having \( i \) interior lattice
	points, excluding those with \( P \cap T = \emptyset \), i.e. those
	that can be realized in \( \R \times [0,1] \). However, polygons
	contained in \( \R \times [0,1] \) are not \( k \)-maximal, hence we
	already have all \( k \)-maximal polygons.
\end{algo}

We turn to the case of collinear interior lattice points, which is
tightly related to being realizable in \( \R \times [-1,1] \).
Clearly, a polygon in \( \R \times [-1,1] \) has collinear interior
lattice points. The converse holds for lattice polygons with at least
two collinear interior lattice points, which Koelman~\cite[Section
	4.3]{Koe91} used to completely describe them. The following
proposition generalizes this fact to \( k \)-rational polygons.

\begin{prop}
	Let $P\subseteq \R^2$ be $k$-rational polygon with collinear interior
	lattice points and $i(P)>k$. Then $P$ can be realized in $\R \times
		[-1,1]$.
\end{prop}

\begin{proof}
	After an affine unimodular transformation, we can assume the interior
	lattice points of \( P \) to be of the form \( (j,0) \) for \( 0 \leq
	j \leq i(P)-1 \). Suppose \( P \) is not contained in \( \R \times
	[-1,1] \). After a reflexion, we can assume that there exists a
	vertex \( v=(x,y) \) of \( P \) with \( y > 1 \). We intend to show
	that there exists a \( k \)-rational point \( q \in P \cap (\R \times
	\{1+\frac{1}{k}\}) \). If \( y = 1 + \frac{1}{k} \), we can take \( q
	= v \). Otherwise, we have \( y \geq 1 + \frac{2}{k} \) and we
	consider the line segment \( L := P \cap (\R \times \{1+\frac{1}{k}\}
	\). We can estimate its length using the intercept theorem to obtain
	\( \length(L) \geq \frac{k}{k+2} \). If \( k \geq 2 \), this implies
	\( \length(L) \geq \frac{1}{k} \) and therefore there exists a \( k
	\)-rational point \( q = (\frac{a}{k},1+\frac{1}{k}) \in L \). After
	a horizontal sheraring, we can assume \( 0 \leq a \leq k \). Consider
	the triangle \(Q\ :=\ \convhull(q, (0,0),(i(P)-1,0))\) and note that
	\( Q \setminus \{q\} \subseteq P^{\circ} \). If \( a = 0 \), we have
	\( (0,1) \in Q \). If \( a > 0 \), then \( i(P) > k \) implies \(
	(1,1) \in Q \). Thus in both cases, we have an additional interior
	lattice point of \( P \), a contradiction.

	It remains to consider the case \( k = 1 \). Take any vertex \( v =
	(x,y) \in P \) with \( y \geq 1 \). Then the triangle \( \convhull(v,
	(0,0),(1,0)) \) does not contain any lattice points other than its
	vertices. Lemma~\ref{lem:empty_triangles_lattice_distance} then
	implies \( y = 1 \).
\end{proof}

Using Algorithm~\ref{algo:subpolygons}, we can determine all
subpolygons of the \( k \)-maximal polygons in \( \R \times [-1,1]
\). For \( i > k \), the preceding Proposition implies that this
gives us all \( k \)-rational polygons with \( i \) collinear
interior lattice points. We arrive at the following classification.

\begin{class}[Data available at~\cite{springer_2024_13928298}]
	\label{class:collinear_interior_lattice_points}
	Below are the numbers of \( 2 \)-rational polygons with \( i \)
	collinear interior lattice points.
	\begin{center}
		{\scriptsize
			\begin{tabular}{c|cccccc}
				\( i \)       & 3                & 4                & 5             & 6             & 7           & 8  \\
				\#Polygons    & 168\,640         & 504\,530         & 1\,279\,695   & 2\,881\,106   & 5\,924\,808
				              & 11\,343\,912                                                                           \\\hline
				\( i \)       & 9                & 10               & 11            & 12            & 13          & 14 \\
				\#Polygons    & 20\,496\,555     & 35\,295\,876     & 58\,364\,056  & 93\,212\,470  &
				144\,449\,999 & 218\,021\,550                                                                          \\\hline
				\( i \)       & 15               & 16               & 17            & 18            & 19          & 20 \\
				\#Polygons    & 321\,478\,832    & 464\,285\,436    & 658\,158\,267 & 917\,447\,376
				              & 1\,259\,556\,240 & 1\,705\,404\,538
			\end{tabular}
		}
	\end{center}
	Below are the numbers of \( 3 \)-rational polygons with \( i \)
	collinear interior lattice points.
	\begin{center}
		{\small
			\begin{tabular}{c|cccc}
				\( i \)    & 4             & 5                & 6 \\\hline
				\#Polygons & 546\,989\,533 & 2\,150\,324\,427
				           & 7\,124\,538\,405
			\end{tabular}
		}
	\end{center}
\end{class}

Recall that lattice polygons with collinear interior lattice points
were classified by Koelman~\cite[Section 4.3]{Koe91}. In particular,
he showed that there are exactly \( \frac{1}{6}(i+3)(2i^2+15i+16) \)
lattice polygons with \( i \geq 2 \) collinear interior lattice
points. Using our values for \( k = 2 \), we found that they too can
be described by a polynomial. Hence we raise the following conjecture,
which we know to be true up to \( i = 20 \).

\begin{conj}
	\label{cnj:number_of_half_integral_polygons_collinear_interior}
	There are exactly
	\[
		\frac{1}{1260}(i + 1) (512 i^6 + 12\,928 i^5 + 137\,740 i^4 +
		685\,145 i^3 + 1\,582\,743 i^2 + 1\,665\,222 i + 710\,640)
	\]
	half-integral polygons with \( i \geq 3 \) collinear interior lattice
	points.
\end{conj}

\section{Rational polygons without interior lattice points}
\label{sec:rational_polygons_without_interior_lattice_points}

We give a classification algorithm for \( k \)-maximal polygons with
no interior lattice points. Note that for \( k = 2 \), the four \( k
\)-maximal polygons with no interior lattice points are
known~\cite[Section 2]{AKW17}. With our approach, we obtain results
up to \( k = 20 \), see
Classification~\ref{class:maximal_polygons_with_no_interior_lattice_points}.

The main tool that allows us to get good bounds for maximal polygons
with no interior lattice points is the lattice width. Recall
that the \emph{lattice width} of a rational polygon \( P \) in
direction \( w \in \Z^2 \setminus \{0\} \) is defined as
\[
	\width_w(P)\ :=\ \max_{v \in P} \langle v,w \rangle - \min_{v
		\in P} \langle v,w\rangle\ \in\ \Z^2.
\]
The \emph{lattice width} \( \lw(P) \) of \( P \) is defined to be the
minimum over all \( \width_w(P) \) with \( w \in \Z^2 \backslash
\{0\} \). We call \( w \in \Z^2 \setminus \{0\} \) a \emph{lattice
	width direction} of \( P \), if \( \lw(P) = \width_w(P) \).

\begin{lemma}
	\label{lem:convex_body_without_interior_lattice_points_width}
	Every rational polygon \( P \) without interior lattice points
	can be realized in \( \R \times [-1,2] \).
\end{lemma}

\begin{proof}
	After an affine unimodular transformation, we may assume that \(
	(1,0) \) is a lattice width direction vector of \( P \) and \( P
	\subseteq [x_l,x_r] \times \R \) for \( x_l,x_r \in \R \). Then
	\cite[Corollary 2.13]{Boh23} gives lower bounds for the numbers \(
	|P^{\circ} \cap (\{h\} \times \Z)| \), where \( x_l \leq h
	\leq x_r \) and \( h \in \Z \). Since \( P \) does not have
	interior lattice points, these imply that there are at most two
	interior integral vertical lines of \( P \), hence \( x_r - x_l
	\leq 3 \) and the claim follows.
\end{proof}

\begin{prop}
	\label{prp:convex_body_without_interior_lattice_points}
	Every rational polygon \( P \) without interior lattice points
	can be realized in \( A \cup (\R\times [0,1]) \cup B \), with the two half-open trapezoids
	\begin{align*}
		A\ := & \ \convhull\left((0,1), (1,1), (2,2), (-1,2)\right)
		\setminus (\R \times \{1\}),                                \\
		B\ := & \ \convhull\left((0,0),(-1,-1),(2,-1),(1,0)\right)
		\setminus (\R \times \{0\}).
	\end{align*}
\end{prop}

\begin{proof}

	By \ref{lem:convex_body_without_interior_lattice_points_width}, we
	can assume \( P \subseteq \R \times [-1,2] \). Consider the line
	segments
	\[
		P_0\ :=\ P^{\circ} \cap (\R \times \{0\}), \qquad P_1\ :=\
		P^{\circ} \cap (\R
		\times \{1\}).
	\]
	We may assume \( P_0 \) and \( P_1 \) to be non-empty, otherwise we
	are in the situation of Proposition~\ref{polygons_in_m11}. Since \(
	P_0 \) and \( P_1 \) do not contain lattice points, we can apply a
	translation and shearing to achieve \( P_0 \subseteq [0,1] \times
	\{0\} \) and \( P_1 \subseteq [0,1] \times \{1\} \). By symmetry, it
	is enough to show that \( P \) does not contain a point \( v=(x,y) \)
	with \( x < 0 \) and \( y > x \). Suppose there is such a point. Then
	\( \convhull(P_0 \cup P_1 \cup \{v\}) \subseteq P \) has \( (0,0) \)
	as an interior lattice point, a contradiction.
\end{proof}

\begin{figure}[H]

	\caption{
		All rational polygons without interior lattice points can be
		realized in $A \cup \R\times [0,1]\cup B$.}
	\label{fig:classify_maximal_polygons_no_interior_lattice_points}
	\begin{tikzpicture}[x=1cm,y=1cm]

		\draw[step=1,black,line width=1pt,xshift=0cm,yshift=0cm,
			opacity=0.3] (-0.9,0.3) grid (5.9,-3.32);


		\fill[line width=2pt,color=black,fill=black,fill opacity=0.2] (1,0) -- (4,0) -- (3,-1) -- (6,-1) -- (6,-2) -- (3,-2) --(4,-3) --(1,-3) --(2,-2) --(-1,-2) --(-1,-1) --(2,-1) --cycle;
		\draw [line width=2pt,color=black] (1,0)-- (4,0);
		\draw [line width=2pt,color=black] (4,0)-- (3,-1);
		\draw [line width=2pt,color=black] (3,-1)-- (6,-1);
		\draw [line width=2pt,color=black] (6,-2)-- (3,-2);
		\draw [line width=2pt,color=black] (3,-2)-- (4,-3);
		\draw [line width=2pt,color=black] (4,-3)-- (1,-3);
		\draw [line width=2pt,color=black] (1,-3)-- (2,-2);
		\draw [line width=2pt,color=black] (2,-2)-- (-1,-2);
		\draw [line width=2pt,color=black] (-1,-1)-- (2,-1);
		\draw [line width=2pt,color=black] (2,-1)-- (1,0);
		\draw [line width=2pt,color=black,dotted] (2,-1)-- (3,-1);
		\draw [line width=2pt,color=black,dotted] (2,-2)-- (3,-2);

		\draw [fill=black] (1,0) circle (3pt);
		\draw [fill=black] (2,-1) circle (3pt) node[below]{\small $a_1$};
		\draw [fill=black] (3,-1) circle (3pt) node[below]{\small $a_2$};
		\draw [fill=black] (4,0) circle (3pt);
		\draw [fill=black] (1,-3) circle (3pt);
		\draw [fill=black] (2,-2) circle (3pt) node[below]{\small $b_1$};
		\draw [fill=black] (3,-2) circle (3pt) node[below]{\small $b_2$};
		\draw [fill=black] (4,-3) circle (3pt);

		\node at (2.5,-0.5) {$A$};
		\node at (2.5,-2.5) {$B$};

		\node at (6.3,0) {2};
		\node at (6.3,-1) {1};
		\node at (6.3,-2) {0};
		\node at (6.3,-3) {-1};

		\node at (0,-3.6) {-2};
		\node at (1,-3.6) {-1};
		\node at (2,-3.6) {0};
		\node at (3,-3.6) {1};
		\node at (4,-3.6) {2};
		\node at (5,-3.6) {3};

	\end{tikzpicture}

\end{figure}
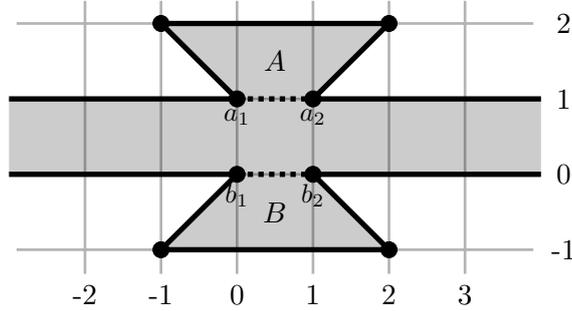

\begin{con}
	\label{cns:surrounding_polygon_no_interior_lattice_points}
	Consider a \( k \)-rational polygon \( P \subseteq A \cup \R \times
	[0,1] \cup B \) with \( P \cap A \neq \emptyset \)  and \( P \cap B
	\neq \emptyset \) where \( A \) and \( B \) are as in the last
	Proposition. Consider the points
	\[
		a_1\ :=\ (0,1), \quad a_2\ :=\ (1,1), \quad b_1\ :=\
		(0,0), \quad b_2\ :=\ (1,0).
	\]
	Then we have unique supporting lines \( L_{a_1}, L_{a_2}, L_{b_1},
	L_{b_2} \) of \( P \) with \( a_j \in L_{a_j}, b_j \in L_{b_j} \)
	such that \( L_{a_j} \) contains a point of \( P \cap A \) and \(
	L_{b_j} \) contains a point of \( P \cap B \). These define affine
	halfplanes \( H_{a_1},H_{a_2},H_{b_1},H_{b_2} \) with \( P \subseteq
	H_{a_j} \) and \( P \subseteq H_{b_j} \). We define the
	\emph{surrounding polygon} of \( P \) to be the \( k \)-rational
	polygon
	\[
		P_{\sur}\ :=\ \convhull\left(Q \cap \frac{1}{k}\Z^2\right), \qquad
		Q\ :=\ \R \times [-1,2] \cap H_{a_1} \cap H_{a_2} \cap H_{b_1}
		\cap H_{b_2}.
	\]
	Note that we have \( P \subseteq P_{\sur} \subseteq A \cup \R \times
	[0,1] \cup B \) and \( i(P) = i(P_{\sur}) = 0 \). In particular, if
	\( P \) is \( k \)-maximal, we have \( P = P_{\sur} \).
\end{con}

\begin{algo}
	\label{algo:classify_maximal_polygons_no_interior_lattice_points}
	Given \( k \in \Z_{\geq 1} \), we can classify all \( k \)-maximal
	polygons without interior lattice points as follows: To any choice
	of points \( v_1, v_2 \in A \) and \( w_1,w_2 \in B \), we
	associate the polygon \( P := \convhull\left(Q \cap
	\frac{1}{k}\Z^2\right) \), where
	\[
		Q := \R \times [-1,2] \cap H(a_1,v_1) \cap H(v_2,a_2) \cap H(b_1,w_1) \cap H(w_2,b_2).
	\]
	This defines a possible surrounding polygon. Going through all
	polygons that arise this way and filtering out those that are \( k
	\)-maximal (for instance, using
	Proposition~\ref{prp:k_maximal_criterion}) gives us all \( k
	\)-maximal polygons without interior lattice points, excluding those
	with \( P \cap A = \emptyset \) or \( P \cap B = \emptyset \). The
	latter ones can be realized in \( \R \times [-1,1] \) and hence can
	be obtained with Algorithm~\ref{algo:classify_maximal_polygons_m1p1}
	applied to \( i = 0 \).
\end{algo}

An implementation of
Algorithm~\ref{algo:classify_maximal_polygons_no_interior_lattice_points}
is available in \texttt{RationalPolygons.jl}
\cite{RationalPolygons_jl}. We ran the algorithm up
to \( k = 20 \) and obtained the following classification, see also
figures~\ref{fig:two_maximal_polygons_no_interior_lattice_points},
\ref{fig:three_maximal_polygons_no_interior_lattice_points} and~
\ref{fig:four_maximal_polygons_no_interior_lattice_points}.

\begin{class}[Data available at~\cite{springer_2024_13838991}]
	\label{class:maximal_polygons_with_no_interior_lattice_points}
	Below are the numbers of \( k \)-maximal polygons without interior
	lattice points, split into the two cases of
	Algorithm~\ref{algo:classify_maximal_polygons_no_interior_lattice_points}.
	\begin{center}
		{\scriptsize
			\begin{tabular}{c|ccccccccccc}
				\( k \)                & 1        & 2        & 3        & 4           & 5           & 6   & 7   & 8      & 9 & 10 \\
				\( \R \times [-1,1] \) & 1        & 4        & 12       & 24          & 54          & 85  & 164 & 244    &
				380                    & 517                                                                                      \\
				\( \R \times [-1,2] \) & 0        & 0        & 2        & 15          & 80          & 214 & 791 & 1\,652 &
				4\,101                 & 8\,368                                                                                   \\
				Total                  & 1        & 4        & 14       & 39          & 134         & 299 & 955
				                       & 1\,896   & 4\,481   & 8\,885                                                             \\\hline\hline
				\( k \)                & 11       & 12       & 13       & 14          & 15          & 16  & 17  & 18
				                       & 19       & 20                                                                            \\
				\( \R \times [-1,1] \) & 809      & 1\,021   & 1\,506   & 1\,878      & 2\,398      &
				2\,987                 & 4\,039   & 4\,743   & 6\,239   & 7\,263                                                  \\
				\( \R \times [-1,2] \) & 17\,757  & 29\,637  & 57\,937  & 91\,009     &
				161\,418               & 258\,728 & 409\,990 & 595\,051 & 930\,544    & 1\,304\,553                               \\
				Total                  & 18\,566  & 30\,658  & 59\,443  & 92\,887     & 163\,816    &
				261\,715               & 414\,029 & 599\,794 & 936\,783 & 1\,311\,816
			\end{tabular}
		}
	\end{center}
\end{class}

Using Algorithm~\ref{algo:subpolygons}, we computed all subpolygons
of the \( k \)-maximal polygons from
Classification~\ref{class:maximal_polygons_with_no_interior_lattice_points}
up to \( k = 6 \). This gives us all \( k \)-rational polygons
without interior lattice points that cannot be realized in \( \R
\times [0,1] \) (there are infinitely many in \( \R \times [0,1] \)),
see Theorem~\ref{thm:main_classification}. The data is available
at~\cite{springer_2024_13838991}.

\begin{figure}
	\caption{All four $2$-maximal rational polygons without interior
        integral points, see also~\cite[Section 2]{AKW17}. All of them
		can be realized in \( \R \times [-1,1] \).}
	\label{fig:two_maximal_polygons_no_interior_lattice_points}
	\begin{tikzpicture}[x=0.5cm,y=0.5cm]

		\draw[step=2,black,line width=1.0pt] (-0.5,-0.5) grid (15.5,5.5);
		\draw[help lines,step=1,black,line width=0.2pt] (-0.5,-0.5) grid (15.5,5.5);

		\fill[opacity=0.2]
		(5,1)--(2,4)--(8,4)--cycle;
		\draw[line width=2.0pt, color=black]
		(5,1)--(2,4)--(8,4)--cycle;
		\draw [fill=black] (5,1) circle (2.0pt);
		\draw [fill=black] (2,4) circle (2.0pt);
		\draw [fill=black] (8,4) circle (2.0pt);

		\fill[opacity=0.2]
		(9,4)--(11,0)--(7,0)--cycle;
		\draw[line width=2.0pt, color=black]
		(9,4)--(11,0)--(7,0)--cycle;
		\draw [fill=black] (9,4) circle (2.0pt);
		\draw [fill=black] (11,0) circle (2.0pt);
		\draw [fill=black] (7,0) circle (2.0pt);

		\fill[opacity=0.2]
		(4,0)--(0,4)--(0,0)--cycle;
		\draw[line width=2.0pt, color=black]
		(4,0)--(0,4)--(0,0)--cycle;
		\draw [fill=black] (4,0) circle (2.0pt);
		\draw [fill=black] (0,4) circle (2.0pt);
		\draw [fill=black] (0,0) circle (2.0pt);

		\fill[opacity=0.2]
		(15,3)--(13,5)--(11,3)--(13,1)--cycle;
		\draw[line width=2.0pt, color=black]
		(15,3)--(13,5)--(11,3)--(13,1)--cycle;
		\draw [fill=black] (15,3) circle (2.0pt);
		\draw [fill=black] (13,5) circle (2.0pt);
		\draw [fill=black] (11,3) circle (2.0pt);
		\draw [fill=black] (13,1) circle (2.0pt);

	\end{tikzpicture}
\end{figure}
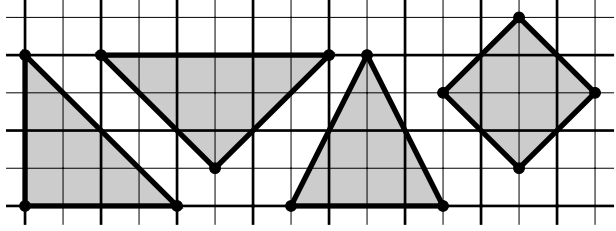

\begin{figure}
	\caption{All fourteen $3$-maximal rational polygons without
		interior integral points. The
		two rightmost polygons in the lower row are the only ones that cannot be
		realized in \( \R \times [-1,1] \).}
	\label{fig:three_maximal_polygons_no_interior_lattice_points}
	\begin{tikzpicture}[x=0.25cm,y=0.25cm]

		\draw[step=3,black,line width=1.0pt] (-1.5,-8.5) grid (41.5,6.5);
		\draw[help lines,step=1,black,line width=0.2pt] (-1.5,-8.5) grid (41.5,6.5);

		\fill[opacity=0.2]
		(10,5)--(8,1)--(3,6)--(10,6)--cycle;
		\draw[line width=2.0pt, color=black]
		(10,5)--(8,1)--(3,6)--(10,6)--cycle;
		\draw [fill=black] (10,5) circle (2.0pt);
		\draw [fill=black] (8,1) circle (2.0pt);
		\draw [fill=black] (3,6) circle (2.0pt);
		\draw [fill=black] (10,6) circle (2.0pt);

		\fill[opacity=0.2]
		(40,1)--(38,5)--(37,6)--(35,0)--(40,0)--cycle;
		\draw[line width=2.0pt, color=black]
		(40,1)--(38,5)--(37,6)--(35,0)--(40,0)--cycle;
		\draw [fill=black] (40,1) circle (2.0pt);
		\draw [fill=black] (38,5) circle (2.0pt);
		\draw [fill=black] (37,6) circle (2.0pt);
		\draw [fill=black] (35,0) circle (2.0pt);
		\draw [fill=black] (40,0) circle (2.0pt);

		\fill[opacity=0.2]
		(6,0)--(0,6)--(0,0)--cycle;
		\draw[line width=2.0pt, color=black]
		(6,0)--(0,6)--(0,0)--cycle;
		\draw [fill=black] (6,0) circle (2.0pt);
		\draw [fill=black] (0,6) circle (2.0pt);
		\draw [fill=black] (0,0) circle (2.0pt);

		\fill[opacity=0.2]
		(13,6)--(17,0)--(11,0)--cycle;
		\draw[line width=2.0pt, color=black]
		(13,6)--(17,0)--(11,0)--cycle;
		\draw [fill=black] (13,6) circle (2.0pt);
		\draw [fill=black] (17,0) circle (2.0pt);
		\draw [fill=black] (11,0) circle (2.0pt);

		\fill[opacity=0.2]
		(18,4)--(30,0)--(18,0)--cycle;
		\draw[line width=2.0pt, color=black]
		(18,4)--(30,0)--(18,0)--cycle;
		\draw [fill=black] (18,4) circle (2.0pt);
		\draw [fill=black] (30,0) circle (2.0pt);
		\draw [fill=black] (18,0) circle (2.0pt);

		\fill[opacity=0.2]
		(11,-5)--(13,-1)--(17,-5)--(13,-7)--cycle;
		\draw[line width=2.0pt, color=black]
		(11,-5)--(13,-1)--(17,-5)--(13,-7)--cycle;
		\draw [fill=black] (11,-5) circle (2.0pt);
		\draw [fill=black] (13,-1) circle (2.0pt);
		\draw [fill=black] (17,-5) circle (2.0pt);
		\draw [fill=black] (13,-7) circle (2.0pt);

		\fill[opacity=0.2]
		(-1,-5)--(1,-1)--(2,-1)--(4,-5)--(4,-6)--(-1,-6)--cycle;
		\draw[line width=2.0pt, color=black]
		(-1,-5)--(1,-1)--(2,-1)--(4,-5)--(4,-6)--(-1,-6)--cycle;
		\draw [fill=black] (-1,-5) circle (2.0pt);
		\draw [fill=black] (1,-1) circle (2.0pt);
		\draw [fill=black] (2,-1) circle (2.0pt);
		\draw [fill=black] (4,-5) circle (2.0pt);
		\draw [fill=black] (4,-6) circle (2.0pt);
		\draw [fill=black] (-1,-6) circle (2.0pt);

		\fill[opacity=0.2]
		(5,-5)--(7,-7)--(8,-7)--(10,-5)--(10,-4)--(8,-2)--(7,-2)--(5,-4)--cycle;
		\draw[line width=2.0pt, color=black]
		(5,-5)--(7,-7)--(8,-7)--(10,-5)--(10,-4)--(8,-2)--(7,-2)--(5,-4)--cycle;
		\draw [fill=black] (5,-5) circle (2.0pt);
		\draw [fill=black] (7,-7) circle (2.0pt);
		\draw [fill=black] (8,-7) circle (2.0pt);
		\draw [fill=black] (10,-5) circle (2.0pt);
		\draw [fill=black] (10,-4) circle (2.0pt);
		\draw [fill=black] (8,-2) circle (2.0pt);
		\draw [fill=black] (7,-2) circle (2.0pt);
		\draw [fill=black] (5,-4) circle (2.0pt);

		\fill[opacity=0.2]
		(27,2)--(27,6)--(18,6)--(26,2)--cycle;
		\draw[line width=2.0pt, color=black]
		(27,2)--(27,6)--(18,6)--(26,2)--cycle;
		\draw [fill=black] (27,2) circle (2.0pt);
		\draw [fill=black] (27,6) circle (2.0pt);
		\draw [fill=black] (18,6) circle (2.0pt);
		\draw [fill=black] (26,2) circle (2.0pt);

		\fill[opacity=0.2]
		(26,-7)--(28,-5)--(26,-1)--(23,-4)--(23,-5)--(25,-7)--cycle;
		\draw[line width=2.0pt, color=black]
		(26,-7)--(28,-5)--(26,-1)--(23,-4)--(23,-5)--(25,-7)--cycle;
		\draw [fill=black] (26,-7) circle (2.0pt);
		\draw [fill=black] (28,-5) circle (2.0pt);
		\draw [fill=black] (26,-1) circle (2.0pt);
		\draw [fill=black] (23,-4) circle (2.0pt);
		\draw [fill=black] (23,-5) circle (2.0pt);
		\draw [fill=black] (25,-7) circle (2.0pt);

		\fill[opacity=0.2]
		(19,-8)--(22,-5)--(20,-1)--(17,-4)--cycle;
		\draw[line width=2.0pt, color=black]
		(19,-8)--(22,-5)--(20,-1)--(17,-4)--cycle;
		\draw [fill=black] (19,-8) circle (2.0pt);
		\draw [fill=black] (22,-5) circle (2.0pt);
		\draw [fill=black] (20,-1) circle (2.0pt);
		\draw [fill=black] (17,-4) circle (2.0pt);

		\fill[opacity=0.2]
		(32,0)--(34,6)--(29,6)--(31,0)--cycle;
		\draw[line width=2.0pt, color=black]
		(32,0)--(34,6)--(29,6)--(31,0)--cycle;
		\draw [fill=black] (32,0) circle (2.0pt);
		\draw [fill=black] (34,6) circle (2.0pt);
		\draw [fill=black] (29,6) circle (2.0pt);
		\draw [fill=black] (31,0) circle (2.0pt);

		\fill[opacity=0.2]
		(31,-7)--(29,-5)--(31,-1)--(32,-1)--(34,-5)--(32,-7)--cycle;
		\draw[line width=2.0pt, color=black]
		(31,-7)--(29,-5)--(31,-1)--(32,-1)--(34,-5)--(32,-7)--cycle;
		\draw [fill=black] (31,-7) circle (2.0pt);
		\draw [fill=black] (29,-5) circle (2.0pt);
		\draw [fill=black] (31,-1) circle (2.0pt);
		\draw [fill=black] (32,-1) circle (2.0pt);
		\draw [fill=black] (34,-5) circle (2.0pt);
		\draw [fill=black] (32,-7) circle (2.0pt);

		\fill[opacity=0.2]
		(38,-7)--(35,-1)--(41,-4)--cycle;
		\draw[line width=2.0pt, color=black]
		(38,-7)--(35,-1)--(41,-4)--cycle;
		\draw [fill=black] (38,-7) circle (2.0pt);
		\draw [fill=black] (35,-1) circle (2.0pt);
		\draw [fill=black] (41,-4) circle (2.0pt);

	\end{tikzpicture}
\end{figure}
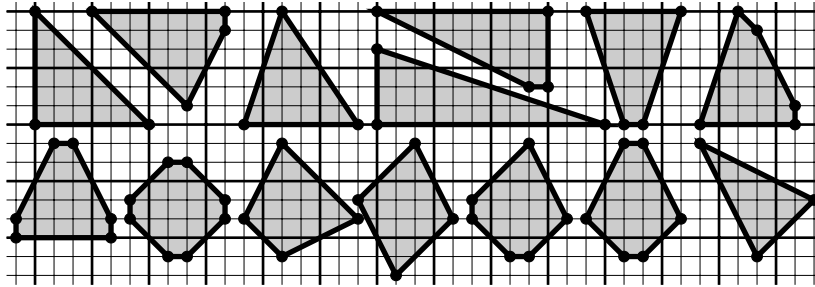

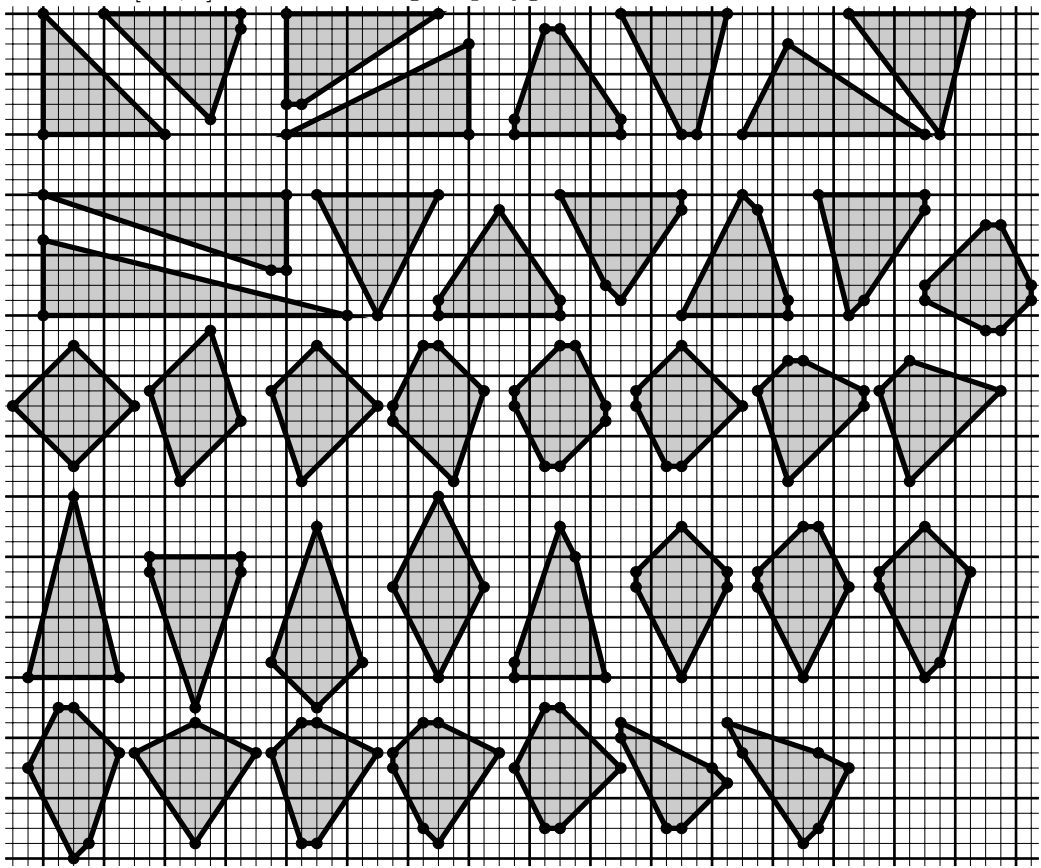
\begin{figure}
	\caption{All thirtynine $4$-maximal rational polygons without
		interior lattice points. The 24 polygons in the upper three rows
		can be realized in \( \R \times [-1,1] \). The remaining 15
		polygons cannot.}
	\label{fig:four_maximal_polygons_no_interior_lattice_points}
	\begin{tikzpicture}[x=0.2cm,y=0.2cm]

		\draw[step=4,black,line width=1.0pt] (-18.5,-64.5) grid (49.5,-7.5);
		\draw[help lines,step=1,black,line width=0.2pt] (-18.5,-64.5) grid (49.5,-7.5);

		\fill[opacity=0.2]
		(27,-16)--(29,-8)--(22,-8)--(26,-16)--cycle;
		\draw[line width=2.0pt, color=black]
		(27,-16)--(29,-8)--(22,-8)--(26,-16)--cycle;
		\draw [fill=black] (27,-16) circle (2.0pt);
		\draw [fill=black] (29,-8) circle (2.0pt);
		\draw [fill=black] (22,-8) circle (2.0pt);
		\draw [fill=black] (26,-16) circle (2.0pt);

		\fill[opacity=0.2]
		(-10,-34)--(-14,-30)--(-18,-34)--(-14,-38)--cycle;
		\draw[line width=2.0pt, color=black]
		(-10,-34)--(-14,-30)--(-18,-34)--(-14,-38)--cycle;
		\draw [fill=black] (-10,-34) circle (2.0pt);
		\draw [fill=black] (-14,-30) circle (2.0pt);
		\draw [fill=black] (-18,-34) circle (2.0pt);
		\draw [fill=black] (-14,-38) circle (2.0pt);

		\fill[opacity=0.2]
		(15,-15)--(17,-9)--(18,-9)--(22,-15)--(22,-16)--(15,-16)--cycle;
		\draw[line width=2.0pt, color=black]
		(15,-15)--(17,-9)--(18,-9)--(22,-15)--(22,-16)--(15,-16)--cycle;
		\draw [fill=black] (15,-15) circle (2.0pt);
		\draw [fill=black] (17,-9) circle (2.0pt);
		\draw [fill=black] (18,-9) circle (2.0pt);
		\draw [fill=black] (22,-15) circle (2.0pt);
		\draw [fill=black] (22,-16) circle (2.0pt);
		\draw [fill=black] (15,-16) circle (2.0pt);

		\fill[opacity=0.2]
		(-7,-39)--(-3,-35)--(-5,-29)--(-9,-33)--cycle;
		\draw[line width=2.0pt, color=black]
		(-7,-39)--(-3,-35)--(-5,-29)--(-9,-33)--cycle;
		\draw [fill=black] (-7,-39) circle (2.0pt);
		\draw [fill=black] (-3,-35) circle (2.0pt);
		\draw [fill=black] (-5,-29) circle (2.0pt);
		\draw [fill=black] (-9,-33) circle (2.0pt);

		\fill[opacity=0.2]
		(12,-10)--(0,-16)--(12,-16)--cycle;
		\draw[line width=2.0pt, color=black]
		(12,-10)--(0,-16)--(12,-16)--cycle;
		\draw [fill=black] (12,-10) circle (2.0pt);
		\draw [fill=black] (0,-16) circle (2.0pt);
		\draw [fill=black] (12,-16) circle (2.0pt);

		\fill[opacity=0.2]
		(38,-27)--(42,-21)--(42,-20)--(35,-20)--(37,-28)--cycle;
		\draw[line width=2.0pt, color=black]
		(38,-27)--(42,-21)--(42,-20)--(35,-20)--(37,-28)--cycle;
		\draw [fill=black] (38,-27) circle (2.0pt);
		\draw [fill=black] (42,-21) circle (2.0pt);
		\draw [fill=black] (42,-20) circle (2.0pt);
		\draw [fill=black] (35,-20) circle (2.0pt);
		\draw [fill=black] (37,-28) circle (2.0pt);

		\fill[opacity=0.2]
		(33,-27)--(31,-21)--(30,-20)--(26,-28)--(33,-28)--cycle;
		\draw[line width=2.0pt, color=black]
		(33,-27)--(31,-21)--(30,-20)--(26,-28)--(33,-28)--cycle;
		\draw [fill=black] (33,-27) circle (2.0pt);
		\draw [fill=black] (31,-21) circle (2.0pt);
		\draw [fill=black] (30,-20) circle (2.0pt);
		\draw [fill=black] (26,-28) circle (2.0pt);
		\draw [fill=black] (33,-28) circle (2.0pt);

		\fill[opacity=0.2]
		(39,-33)--(41,-39)--(47,-33)--(41,-31)--cycle;
		\draw[line width=2.0pt, color=black]
		(39,-33)--(41,-39)--(47,-33)--(41,-31)--cycle;
		\draw [fill=black] (39,-33) circle (2.0pt);
		\draw [fill=black] (41,-39) circle (2.0pt);
		\draw [fill=black] (47,-33) circle (2.0pt);
		\draw [fill=black] (41,-31) circle (2.0pt);

		\fill[opacity=0.2]
		(6,-28)--(10,-20)--(2,-20)--cycle;
		\draw[line width=2.0pt, color=black]
		(6,-28)--(10,-20)--(2,-20)--cycle;
		\draw [fill=black] (6,-28) circle (2.0pt);
		\draw [fill=black] (10,-20) circle (2.0pt);
		\draw [fill=black] (2,-20) circle (2.0pt);

		\fill[opacity=0.2]
		(49,-26)--(47,-22)--(46,-22)--(42,-26)--(42,-27)--(46,-29)--(47,-29)--(49,-27)--cycle;
		\draw[line width=2.0pt, color=black]
		(49,-26)--(47,-22)--(46,-22)--(42,-26)--(42,-27)--(46,-29)--(47,-29)--(49,-27)--cycle;
		\draw [fill=black] (49,-26) circle (2.0pt);
		\draw [fill=black] (47,-22) circle (2.0pt);
		\draw [fill=black] (46,-22) circle (2.0pt);
		\draw [fill=black] (42,-26) circle (2.0pt);
		\draw [fill=black] (42,-27) circle (2.0pt);
		\draw [fill=black] (46,-29) circle (2.0pt);
		\draw [fill=black] (47,-29) circle (2.0pt);
		\draw [fill=black] (49,-27) circle (2.0pt);

		\fill[opacity=0.2]
		(0,-25)--(0,-20)--(-16,-20)--(-1,-25)--cycle;
		\draw[line width=2.0pt, color=black]
		(0,-25)--(0,-20)--(-16,-20)--(-1,-25)--cycle;
		\draw [fill=black] (0,-25) circle (2.0pt);
		\draw [fill=black] (0,-20) circle (2.0pt);
		\draw [fill=black] (-16,-20) circle (2.0pt);
		\draw [fill=black] (-1,-25) circle (2.0pt);

		\fill[opacity=0.2]
		(-8,-16)--(-16,-8)--(-16,-16)--cycle;
		\draw[line width=2.0pt, color=black]
		(-8,-16)--(-16,-8)--(-16,-16)--cycle;
		\draw [fill=black] (-8,-16) circle (2.0pt);
		\draw [fill=black] (-16,-8) circle (2.0pt);
		\draw [fill=black] (-16,-16) circle (2.0pt);

		\fill[opacity=0.2]
		(15,-34)--(17,-38)--(18,-38)--(21,-35)--(21,-34)--(19,-30)--(18,-30)--(15,-33)--cycle;
		\draw[line width=2.0pt, color=black]
		(15,-34)--(17,-38)--(18,-38)--(21,-35)--(21,-34)--(19,-30)--(18,-30)--(15,-33)--cycle;
		\draw [fill=black] (15,-34) circle (2.0pt);
		\draw [fill=black] (17,-38) circle (2.0pt);
		\draw [fill=black] (18,-38) circle (2.0pt);
		\draw [fill=black] (21,-35) circle (2.0pt);
		\draw [fill=black] (21,-34) circle (2.0pt);
		\draw [fill=black] (19,-30) circle (2.0pt);
		\draw [fill=black] (18,-30) circle (2.0pt);
		\draw [fill=black] (15,-33) circle (2.0pt);

		\fill[opacity=0.2]
		(10,-28)--(18,-28)--(18,-27)--(14,-21)--(10,-27)--cycle;
		\draw[line width=2.0pt, color=black]
		(10,-28)--(18,-28)--(18,-27)--(14,-21)--(10,-27)--cycle;
		\draw [fill=black] (10,-28) circle (2.0pt);
		\draw [fill=black] (18,-28) circle (2.0pt);
		\draw [fill=black] (18,-27) circle (2.0pt);
		\draw [fill=black] (14,-21) circle (2.0pt);
		\draw [fill=black] (10,-27) circle (2.0pt);

		\fill[opacity=0.2]
		(1,-39)--(6,-34)--(2,-30)--(-1,-33)--cycle;
		\draw[line width=2.0pt, color=black]
		(1,-39)--(6,-34)--(2,-30)--(-1,-33)--cycle;
		\draw [fill=black] (1,-39) circle (2.0pt);
		\draw [fill=black] (6,-34) circle (2.0pt);
		\draw [fill=black] (2,-30) circle (2.0pt);
		\draw [fill=black] (-1,-33) circle (2.0pt);

		\fill[opacity=0.2]
		(-3,-9)--(-5,-15)--(-12,-8)--(-3,-8)--cycle;
		\draw[line width=2.0pt, color=black]
		(-3,-9)--(-5,-15)--(-12,-8)--(-3,-8)--cycle;
		\draw [fill=black] (-3,-9) circle (2.0pt);
		\draw [fill=black] (-5,-15) circle (2.0pt);
		\draw [fill=black] (-12,-8) circle (2.0pt);
		\draw [fill=black] (-3,-8) circle (2.0pt);

		\fill[opacity=0.2]
		(23,-34)--(25,-38)--(26,-38)--(30,-34)--(26,-30)--(23,-33)--cycle;
		\draw[line width=2.0pt, color=black]
		(23,-34)--(25,-38)--(26,-38)--(30,-34)--(26,-30)--(23,-33)--cycle;
		\draw [fill=black] (23,-34) circle (2.0pt);
		\draw [fill=black] (25,-38) circle (2.0pt);
		\draw [fill=black] (26,-38) circle (2.0pt);
		\draw [fill=black] (30,-34) circle (2.0pt);
		\draw [fill=black] (26,-30) circle (2.0pt);
		\draw [fill=black] (23,-33) circle (2.0pt);

		\fill[opacity=0.2]
		(33,-10)--(42,-16)--(30,-16)--cycle;
		\draw[line width=2.0pt, color=black]
		(33,-10)--(42,-16)--(30,-16)--cycle;
		\draw [fill=black] (33,-10) circle (2.0pt);
		\draw [fill=black] (42,-16) circle (2.0pt);
		\draw [fill=black] (30,-16) circle (2.0pt);

		\fill[opacity=0.2]
		(-16,-23)--(4,-28)--(-16,-28)--cycle;
		\draw[line width=2.0pt, color=black]
		(-16,-23)--(4,-28)--(-16,-28)--cycle;
		\draw [fill=black] (-16,-23) circle (2.0pt);
		\draw [fill=black] (4,-28) circle (2.0pt);
		\draw [fill=black] (-16,-28) circle (2.0pt);

		\fill[opacity=0.2]
		(43,-16)--(37,-8)--(45,-8)--cycle;
		\draw[line width=2.0pt, color=black]
		(43,-16)--(37,-8)--(45,-8)--cycle;
		\draw [fill=black] (43,-16) circle (2.0pt);
		\draw [fill=black] (37,-8) circle (2.0pt);
		\draw [fill=black] (45,-8) circle (2.0pt);

		\fill[opacity=0.2]
		(7,-34)--(9,-30)--(10,-30)--(13,-33)--(11,-39)--(7,-35)--cycle;
		\draw[line width=2.0pt, color=black]
		(7,-34)--(9,-30)--(10,-30)--(13,-33)--(11,-39)--(7,-35)--cycle;
		\draw [fill=black] (7,-34) circle (2.0pt);
		\draw [fill=black] (9,-30) circle (2.0pt);
		\draw [fill=black] (10,-30) circle (2.0pt);
		\draw [fill=black] (13,-33) circle (2.0pt);
		\draw [fill=black] (11,-39) circle (2.0pt);
		\draw [fill=black] (7,-35) circle (2.0pt);

		\fill[opacity=0.2]
		(33,-31)--(31,-33)--(33,-39)--(38,-34)--(38,-33)--(34,-31)--cycle;
		\draw[line width=2.0pt, color=black]
		(33,-31)--(31,-33)--(33,-39)--(38,-34)--(38,-33)--(34,-31)--cycle;
		\draw [fill=black] (33,-31) circle (2.0pt);
		\draw [fill=black] (31,-33) circle (2.0pt);
		\draw [fill=black] (33,-39) circle (2.0pt);
		\draw [fill=black] (38,-34) circle (2.0pt);
		\draw [fill=black] (38,-33) circle (2.0pt);
		\draw [fill=black] (34,-31) circle (2.0pt);

		\fill[opacity=0.2]
		(0,-14)--(0,-8)--(10,-8)--(1,-14)--cycle;
		\draw[line width=2.0pt, color=black]
		(0,-14)--(0,-8)--(10,-8)--(1,-14)--cycle;
		\draw [fill=black] (0,-14) circle (2.0pt);
		\draw [fill=black] (0,-8) circle (2.0pt);
		\draw [fill=black] (10,-8) circle (2.0pt);
		\draw [fill=black] (1,-14) circle (2.0pt);

		\fill[opacity=0.2]
		(21,-26)--(18,-20)--(26,-20)--(26,-21)--(22,-27)--cycle;
		\draw[line width=2.0pt, color=black]
		(21,-26)--(18,-20)--(26,-20)--(26,-21)--(22,-27)--cycle;
		\draw [fill=black] (21,-26) circle (2.0pt);
		\draw [fill=black] (18,-20) circle (2.0pt);
		\draw [fill=black] (26,-20) circle (2.0pt);
		\draw [fill=black] (26,-21) circle (2.0pt);
		\draw [fill=black] (22,-27) circle (2.0pt);

		\fill[opacity=0.2]
		(23,-46)--(26,-52)--(29,-46)--(29,-45)--(26,-42)--(23,-45)--cycle;
		\draw[line width=2.0pt, color=black]
		(23,-46)--(26,-52)--(29,-46)--(29,-45)--(26,-42)--(23,-45)--cycle;
		\draw [fill=black] (23,-46) circle (2.0pt);
		\draw [fill=black] (26,-52) circle (2.0pt);
		\draw [fill=black] (29,-46) circle (2.0pt);
		\draw [fill=black] (29,-45) circle (2.0pt);
		\draw [fill=black] (26,-42) circle (2.0pt);
		\draw [fill=black] (23,-45) circle (2.0pt);

		\fill[opacity=0.2]
		(-9,-45)--(-6,-54)--(-3,-45)--(-3,-44)--(-9,-44)--cycle;
		\draw[line width=2.0pt, color=black]
		(-9,-45)--(-6,-54)--(-3,-45)--(-3,-44)--(-9,-44)--cycle;
		\draw [fill=black] (-9,-45) circle (2.0pt);
		\draw [fill=black] (-6,-54) circle (2.0pt);
		\draw [fill=black] (-3,-45) circle (2.0pt);
		\draw [fill=black] (-3,-44) circle (2.0pt);
		\draw [fill=black] (-9,-44) circle (2.0pt);

		\fill[opacity=0.2]
		(18,-54)--(22,-58)--(18,-62)--(17,-62)--(15,-58)--(17,-54)--cycle;
		\draw[line width=2.0pt, color=black]
		(18,-54)--(22,-58)--(18,-62)--(17,-62)--(15,-58)--(17,-54)--cycle;
		\draw [fill=black] (18,-54) circle (2.0pt);
		\draw [fill=black] (22,-58) circle (2.0pt);
		\draw [fill=black] (18,-62) circle (2.0pt);
		\draw [fill=black] (17,-62) circle (2.0pt);
		\draw [fill=black] (15,-58) circle (2.0pt);
		\draw [fill=black] (17,-54) circle (2.0pt);

		\fill[opacity=0.2]
		(-14,-54)--(-11,-57)--(-13,-63)--(-14,-64)--(-17,-58)--(-15,-54)--cycle;
		\draw[line width=2.0pt, color=black]
		(-14,-54)--(-11,-57)--(-13,-63)--(-14,-64)--(-17,-58)--(-15,-54)--cycle;
		\draw [fill=black] (-14,-54) circle (2.0pt);
		\draw [fill=black] (-11,-57) circle (2.0pt);
		\draw [fill=black] (-13,-63) circle (2.0pt);
		\draw [fill=black] (-14,-64) circle (2.0pt);
		\draw [fill=black] (-17,-58) circle (2.0pt);
		\draw [fill=black] (-15,-54) circle (2.0pt);

		\fill[opacity=0.2]
		(13,-46)--(10,-40)--(7,-46)--(10,-52)--cycle;
		\draw[line width=2.0pt, color=black]
		(13,-46)--(10,-40)--(7,-46)--(10,-52)--cycle;
		\draw [fill=black] (13,-46) circle (2.0pt);
		\draw [fill=black] (10,-40) circle (2.0pt);
		\draw [fill=black] (7,-46) circle (2.0pt);
		\draw [fill=black] (10,-52) circle (2.0pt);

		\fill[opacity=0.2]
		(31,-46)--(34,-52)--(37,-46)--(35,-42)--(34,-42)--(31,-45)--cycle;
		\draw[line width=2.0pt, color=black]
		(31,-46)--(34,-52)--(37,-46)--(35,-42)--(34,-42)--(31,-45)--cycle;
		\draw [fill=black] (31,-46) circle (2.0pt);
		\draw [fill=black] (34,-52) circle (2.0pt);
		\draw [fill=black] (37,-46) circle (2.0pt);
		\draw [fill=black] (35,-42) circle (2.0pt);
		\draw [fill=black] (34,-42) circle (2.0pt);
		\draw [fill=black] (31,-45) circle (2.0pt);

		\fill[opacity=0.2]
		(15,-51)--(18,-42)--(19,-44)--(21,-52)--(15,-52)--cycle;
		\draw[line width=2.0pt, color=black]
		(15,-51)--(18,-42)--(19,-44)--(21,-52)--(15,-52)--cycle;
		\draw [fill=black] (15,-51) circle (2.0pt);
		\draw [fill=black] (18,-42) circle (2.0pt);
		\draw [fill=black] (19,-44) circle (2.0pt);
		\draw [fill=black] (21,-52) circle (2.0pt);
		\draw [fill=black] (15,-52) circle (2.0pt);

		\fill[opacity=0.2]
		(10,-55)--(14,-57)--(10,-63)--(9,-62)--(7,-58)--(7,-57)--(9,-55)--cycle;
		\draw[line width=2.0pt, color=black]
		(10,-55)--(14,-57)--(10,-63)--(9,-62)--(7,-58)--(7,-57)--(9,-55)--cycle;
		\draw [fill=black] (10,-55) circle (2.0pt);
		\draw [fill=black] (14,-57) circle (2.0pt);
		\draw [fill=black] (10,-63) circle (2.0pt);
		\draw [fill=black] (9,-62) circle (2.0pt);
		\draw [fill=black] (7,-58) circle (2.0pt);
		\draw [fill=black] (7,-57) circle (2.0pt);
		\draw [fill=black] (9,-55) circle (2.0pt);

		\fill[opacity=0.2]
		(39,-45)--(42,-42)--(45,-45)--(43,-51)--(42,-52)--(39,-46)--cycle;
		\draw[line width=2.0pt, color=black]
		(39,-45)--(42,-42)--(45,-45)--(43,-51)--(42,-52)--(39,-46)--cycle;
		\draw [fill=black] (39,-45) circle (2.0pt);
		\draw [fill=black] (42,-42) circle (2.0pt);
		\draw [fill=black] (45,-45) circle (2.0pt);
		\draw [fill=black] (43,-51) circle (2.0pt);
		\draw [fill=black] (42,-52) circle (2.0pt);
		\draw [fill=black] (39,-46) circle (2.0pt);

		\fill[opacity=0.2]
		(-14,-40)--(-11,-52)--(-17,-52)--cycle;
		\draw[line width=2.0pt, color=black]
		(-14,-40)--(-11,-52)--(-17,-52)--cycle;
		\draw [fill=black] (-14,-40) circle (2.0pt);
		\draw [fill=black] (-11,-52) circle (2.0pt);
		\draw [fill=black] (-17,-52) circle (2.0pt);

		\fill[opacity=0.2]
		(2,-55)--(6,-57)--(2,-63)--(1,-63)--(-1,-57)--(1,-55)--cycle;
		\draw[line width=2.0pt, color=black]
		(2,-55)--(6,-57)--(2,-63)--(1,-63)--(-1,-57)--(1,-55)--cycle;
		\draw [fill=black] (2,-55) circle (2.0pt);
		\draw [fill=black] (6,-57) circle (2.0pt);
		\draw [fill=black] (2,-63) circle (2.0pt);
		\draw [fill=black] (1,-63) circle (2.0pt);
		\draw [fill=black] (-1,-57) circle (2.0pt);
		\draw [fill=black] (1,-55) circle (2.0pt);

		\fill[opacity=0.2]
		(-6,-55)--(-10,-57)--(-6,-63)--(-2,-57)--cycle;
		\draw[line width=2.0pt, color=black]
		(-6,-55)--(-10,-57)--(-6,-63)--(-2,-57)--cycle;
		\draw [fill=black] (-6,-55) circle (2.0pt);
		\draw [fill=black] (-10,-57) circle (2.0pt);
		\draw [fill=black] (-6,-63) circle (2.0pt);
		\draw [fill=black] (-2,-57) circle (2.0pt);

		\fill[opacity=0.2]
		(2,-54)--(5,-51)--(2,-42)--(-1,-51)--cycle;
		\draw[line width=2.0pt, color=black]
		(2,-54)--(5,-51)--(2,-42)--(-1,-51)--cycle;
		\draw [fill=black] (2,-54) circle (2.0pt);
		\draw [fill=black] (5,-51) circle (2.0pt);
		\draw [fill=black] (2,-42) circle (2.0pt);
		\draw [fill=black] (-1,-51) circle (2.0pt);

		\fill[opacity=0.2]
		(22,-56)--(25,-62)--(26,-62)--(29,-59)--(28,-58)--(22,-55)--cycle;
		\draw[line width=2.0pt, color=black]
		(22,-56)--(25,-62)--(26,-62)--(29,-59)--(28,-58)--(22,-55)--cycle;
		\draw [fill=black] (22,-56) circle (2.0pt);
		\draw [fill=black] (25,-62) circle (2.0pt);
		\draw [fill=black] (26,-62) circle (2.0pt);
		\draw [fill=black] (29,-59) circle (2.0pt);
		\draw [fill=black] (28,-58) circle (2.0pt);
		\draw [fill=black] (22,-55) circle (2.0pt);

		\fill[opacity=0.2]
		(35,-62)--(37,-58)--(35,-57)--(29,-55)--(30,-57)--(34,-63)--cycle;
		\draw[line width=2.0pt, color=black]
		(35,-62)--(37,-58)--(35,-57)--(29,-55)--(30,-57)--(34,-63)--cycle;
		\draw [fill=black] (35,-62) circle (2.0pt);
		\draw [fill=black] (37,-58) circle (2.0pt);
		\draw [fill=black] (35,-57) circle (2.0pt);
		\draw [fill=black] (29,-55) circle (2.0pt);
		\draw [fill=black] (30,-57) circle (2.0pt);
		\draw [fill=black] (34,-63) circle (2.0pt);

	\end{tikzpicture}
\end{figure}

\section{Rational polygons with exactly one interior lattice point
  and LDP polygons}
\label{sec:rational_polygons_with_exactly_one_interior_lattice_points}

We give a classification algorithm for \( k \)-maximal polygons with no
interior lattice points and obtain explicit results up to \( k = 10
\), see
Classification~\ref{class:maximal_polygons_one_interior_lattice_points}.
Moreover, we apply this to LDP polygons and extend the classification
of \( k \)-hollow LDP polygons from~\cite{HHS23} to \( k = 6 \), see
Classification~\ref{class:ldp_polygons}.

\begin{lemma}\label{convex_body_one_interior_lattice_points_width}
	Every rational polygon \( P \) with exactly one interior lattice point
	can be realized in \( \R \times [-2,2] \).
\end{lemma}

\begin{proof}
	After an affine unimodular transformation, we may assume that \(
	(1,0) \) is a lattice width direction vector of \( P \) and \( P
	\subseteq [x_l,x_r] \times \R \) for \( x_l,x_r \in \R \). Then
	\cite[Corollary 2.13]{Boh23} gives lower bounds for the numbers \(
	|P^{\circ} \cap (\{h\} \times \Z)| \), where \( x_l \leq h
	\leq x_r \) and \( h \in \Z \). Since \( P \) only has one
	interior lattice point, these imply that there are at most
	three interior integral vertical lines of \( P \), hence \( x_r - x_l
	\leq 4 \) and the claim follows.
\end{proof}

\begin{prop}
	Every rational polygon with exactly one interior
	lattice point can be realized in \( \R\times [0,1] \cup A \cup B \cup
	C_q\) for some \( q \in \{1,2,3,4,5\} \), where
	\begin{align*}
		A\ :=\    & \convhull\left((-1,1), (0,1), (0,2),
		(-2,2)\right)\ \setminus\ (\R \times \{1\}),     \\
		B\ :=\    &
		\convhull\left((-1,0),(-2,-1),(3,-1),(1,0)\right)\
		\setminus\ (\R \times \{0\},                     \\
		C_1\ :=\  &
		\convhull\left((-3,-2),(-2,-2),(-1,-1),(-2,-1)\right)\
		\setminus\ (\R \times \{-1\}),                   \\
		C_2\ :=\  &
		\convhull\left((-2,-2),(0,-2),(0,-1),(-1,-1)\right)\
		\setminus\ (\R \times \{-1\}),                   \\
		C_3\ :=\  &
		\convhull\left((0,-2),(2,-2),(1,-1),(0,-1)\right)\
		\setminus\ (\R \times \{-1\}),                   \\
		C_4\ :=\  &
		\convhull\left((2,-2),(4,-2),(5,-2),(1,0)\right) \
		\setminus\ (\R \times \{-1\}),                   \\
		C_5\ :=\  &
		\convhull\left((-1,0),(-3,-2),(5,-2),(1,0)\right)\
		\setminus\ (\R \times \{-1\}).
	\end{align*}
\end{prop}

\begin{proof}
	By \ref{convex_body_one_interior_lattice_points_width}, we can
	assume \( P \subseteq \R \times [-2,2] \) and after translation,
	we can assume the interior lattice point of \( P \) to be the
	origin. Consider the line segments
	\[
		P_1\ :=\ P^{\circ} \cap (\R \times \{1\}), \quad P_0\ :=\
		P^{\circ} \cap (\R \times \{0\}), \quad P_{-1}\ :=\
		P^{\circ} \cap (\R \times \{-1\}). \quad
	\]
	Clearly, we have \( P_0 \neq \emptyset \) and convexity forces \( P_0
	\subseteq [-1,1] \times \{0\} \). If \( P_1 = P_{-1} = \emptyset \),
	we are in the situation of Proposition~\ref{polygons_in_m11}, hence
	we are done after suitable translation and reflexion. Thus, after a
	reflexion, we can assume \( P_1 \neq \emptyset \). If \( P \) were to
	contain any point outside of \( \R\times [0,1] \cup A \cup B \cup
	\bigcup_q C_q\), convexity would force \( P \) to contain an
	additional lattice point, which is impossible. Moreover, if there
	were points \( v \in P^{\circ} \cap C_q \) and \( w \in P^{\circ}
	\cap C_r \) for \( q \neq r \), we would have an interior lattice
	point in the triangle \( \convhull(v,w,(0,0)) \subseteq P \), which
	is impossible. Hence \( P^{\circ} \cap C_q \neq \emptyset \) for at
	most one \( q \in \{1,2,3,4,5\} \) and we arrive at the claim.
\end{proof}

\begin{figure}[H]
	\caption{
		All rational polygons with exactly one interior lattice point
		can be realized in \( \R \times
		[0,1] \cup A \cup B \cup C_q \) for some \( q \in \{1,2,3,4,5\} \).}

	\label{fig:maximal_polygons_one_interior_lattice_point}
	\begin{tikzpicture}[x=1cm,y=1cm]

		\draw[step=1,black,line width=1pt,xshift=0cm,yshift=0cm,
			opacity=0.3] (-3.9,2.3) grid (5.9,-2.3);

		\fill[line width=2pt,color=black,fill=black,fill
			opacity=0.2] (-2,2) -- (0,2) -- (0,1) -- (5.9,1) --
		(5.9,0) -- (1,0) -- (5,-2) -- (-3,-2) -- (-1,0) --
		(-3.9,0) -- (-3.9,1) -- (-1,1) -- cycle;

		\draw [line width=2pt,color=black] (-3.9,1) -- (-1,1) --
		(-2,2) -- (0,2) -- (0,1) -- (5.9,1);
		\draw [line width=2pt,color=black] (-3.9,0) -- (-1,0) --
		(-3,-2) -- (5,-2) -- (1,0) -- (5.9,0);
		\draw [line width=2pt,color=black] (-1,-1) -- (-2,-2);
		\draw [line width=2pt,color=black] (0,-1) -- (0,-2);
		\draw [line width=2pt,color=black] (1,-1) -- (2,-2);
		\draw [line width=2pt,color=black] (2,-1) -- (3,-2);
		\draw [line width=2pt,color=black,dotted] (-1,1) -- (0,1);
		\draw [line width=2pt,color=black,dotted] (-1,0) -- (1,0);
		\draw [line width=2pt,color=black,dotted] (-2,-1) --
		(3,-1);

		\draw [fill=black] (-1,1) circle (3pt) node[below]{$a_1$};
		\draw [fill=black] (-2,2) circle (3pt);
		\draw [fill=black] (0,2) circle (3pt);
		\draw [fill=black] (0,1) circle (3pt) node[below]{$a_2$};
		\draw [fill=black] (-1,0) circle (3pt) node[below]{$b_1$};
		\draw [fill=black] (0,0) circle (3pt);
		\draw [fill=black] (1,0) circle (3pt) node[below]{$b_2$};
		\draw [fill=black] (-2,-1) circle (3pt);
		\draw [fill=black] (-1,-1) circle (3pt);
		\draw [fill=black] (0,-1) circle (3pt);
		\draw [fill=black] (1,-1) circle (3pt);
		\draw [fill=black] (2,-1) circle (3pt);
		\draw [fill=black] (3,-1) circle (3pt);
		\draw [fill=black] (-3,-2) circle (3pt);
		\draw [fill=black] (5,-2) circle (3pt);

		\node at (6.2,2) {2};
		\node at (6.2,1) {1};
		\node at (6.2,0) {0};
		\node at (6.2,-1) {-1};
		\node at (6.2,-2) {-2};

		\node at (-3,-2.5) {-3};
		\node at (-2,-2.5) {-2};
		\node at (-1,-2.5) {-1};
		\node at (0,-2.5) {0};
		\node at (1,-2.5) {1};
		\node at (2,-2.5) {2};
		\node at (3,-2.5) {3};
		\node at (4,-2.5) {4};
		\node at (5,-2.5) {5};

		\node at (-0.5,1.5) {$A$};
		\node at (0,-0.5) {$B$};
		\node at (-2,-1.5) {$C_1$};
		\node at (-0.7,-1.5) {$C_2$};
		\node at (0.7,-1.5) {$C_3$};
		\node at (2,-1.5) {$C_4$};
		\node at (3.3,-1.5) {$C_5$};

	\end{tikzpicture}

\end{figure}
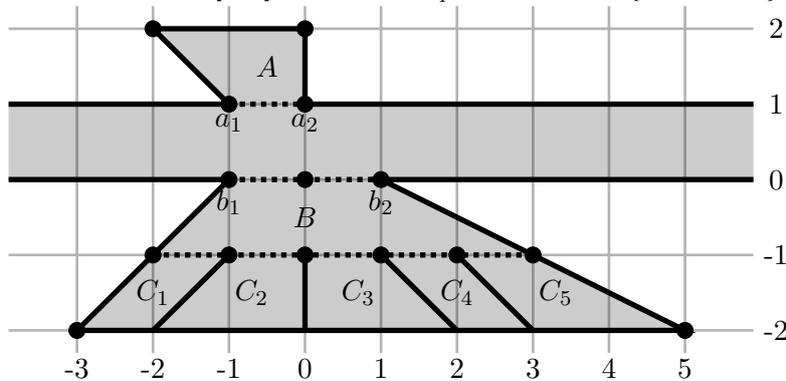

\begin{con}
	\label{cns:surrounding_polygon_one_interior_lattice_point}
	Consider a \( k \)-rational polygon \( P \subseteq \R \times
	[0,1] \cup A \cup B \cup C_q \) with \( P \cap A \neq
	\emptyset \), where \( q \in \{1,2,3,4,5\} \). Consider
	the points
	\[
		a_1 := (-1,1),\quad a_2 := (0,1), \quad b_1 := (-1,0),
		\quad b_2 := (1,0),
	\]
	\[
		c_1 := (q-3, -1), \quad c_2 := (q-2, -1).
	\]
	Then we have unique supporting lines \(
	L_{a_1},L_{a_2},L_{b_1},L_{b_2} \) of \( P \) with \( a_j \in
	L_{a_j}, b_j \in L_{b_j} \) such that \( L_{a_j} \) contains a point
	of \( P \cap A \) and \( L_{b_j} \) contains a point of \( P \cap B
	\). Moreover, if \( P \cap C_q \neq \emptyset \), we have unique
	supporting lines \( L_{c_1}, L_{c_2} \) defined analogously. All
	these supporting lines define affine halfplanes \( H_{a_j}, H_{b_j}
	\) and \( H_{c_j} \) that contain \( P \). We define the
	\emph{surrounding polygon} of \( P \) to be the \( k \)-rational
	polygon \(P_{\sur} := \convhull\left(Q \cap \frac{1}{k}\Z^2\right)\)
	where
	\[
		Q :=
		\begin{cases}
			\R \times [-1,2] \cap H_{a_1} \cap H_{a_2} \cap
			H_{b_1} \cap H_{b_2},                           & P \cap C_q = \emptyset, \\
			\R \times [-2,2] \cap H_{a_1} \cap H_{a_2} \cap
			H_{b_1} \cap H_{b_2} \cap H_{c_1} \cap H_{c_2}, & P
			\cap C_q \neq \emptyset.
		\end{cases}
	\]
	Note that we have \( P \subseteq P_{\sur} \subseteq \R \times [0,1]
	\cup A \cup B \cup C_q \) and \( i(P) = i(P_{\sur}) = 1 \). In
	paticular, if \( P \) is \( k \)-maximal, we have \( P = P_{\sur} \).

\end{con}

\begin{algo}
	\label{algo:classify_maximal_polygons_one_interior_lattice_points}
	Given \( k \in \Z_{\geq 1} \), we can classify all \( k \)-maximal
	polygons with exactly one interior lattice point as follows:
	Pick \( q \in \{1, \dots, 5\} \). To any choice of points
	\( v_1, v_2 \in A, w_1,w_2 \in B \) and \( u_1, u_2 \in C_q
	\), we associate the polygon \( P := \convhull\left(Q \cap
	\frac{1}{k}\Z^2\right) \), where
	\[
		Q := \R \times [-2,2] \cap H(a_1,v_1) \cap H(v_2,a_2) \cap
		H(b_1,w_1) \cap H(w_2,b_2) \cap H(c_1,u_1) \cap
		H(u_2,c_2).
	\]
	This defines a possible surrounding polygon. Going through all
	polygons that arise this way and filtering out those that are \( k
	\)-maximal (for instance, using
	Proposition~\ref{prp:k_maximal_criterion}) gives us all \( k
	\)-maximal polygons with one interior lattice point that can be
	realized in \( \R \times [-2,2] \), excluding those that can be
	realized in \( \R \times [-1,2] \). The latter ones are obtained in
	an analogous way but without the halfplanes given by \( c_1 \) and \(
	c_2 \). Finally, the polygons that can be realized in \( \R \times
	[-1,1] \) are obtained by
	Algorithm~\ref{algo:classify_maximal_polygons_m1p1} applied to \( i =
	1 \).
\end{algo}

An implementation of
Algorithm~\ref{algo:classify_maximal_polygons_one_interior_lattice_points}
is available in \texttt{RationalPolygons.jl}
\cite{RationalPolygons_jl}. We ran the algorithm up
to \( k = 10 \) and obtained the following classification, see also
figures~\ref{fig:two_maximal_polygons_one_interior_lattice_point}
and~\ref{fig:three_maximal_polygons_one_interior_lattice_point}.

\begin{class}[Data available at~\cite{springer_2024_13839216}]
	\label{class:maximal_polygons_one_interior_lattice_points}
	Below are the numbers of \( k \)-rational polygons with
	exactly one interior lattice point, split into the three
	cases according to
	Algorithm~\ref{algo:classify_maximal_polygons_one_interior_lattice_points}.
	\begin{center}
		\begin{tabular}{c|cccccccccc}
			\( k \)                & 1       & 2        & 3       & 4       & 5   & 6      & 7      & 8
			                       & 9       & 10                                                              \\\hline
			\(\R\times[-1,1]\)     & 2       & 9        & 26      & 57      & 132 & 199
			                       & 396     & 605      & 937     & 1\,260                                     \\
			\( \R \times [-1,2] \) & 1       & 1        & 12      & 83      & 470 &
			1\,390                 & 4\,964  & 11\,426  & 27\,801 & 56\,303                                    \\
			\( \R \times [-2,2] \) & 0       & 0        & 1       & 5       & 96  & 329    & 2\,874 & 8\,241 &
			32\,176                & 81\,950                                                                   \\
			Total                  & 3       & 10       & 39      & 145     & 698 & 1\,918 & 8\,234 &
			20\,272                & 60\,914 & 139\,513                                                        \\
		\end{tabular}
	\end{center}
\end{class}

Using Algorithm~\ref{algo:subpolygons}, we computed all subpolygons
of the \( k \)-maximal polygons from
Classification~\ref{class:maximal_polygons_one_interior_lattice_points}
up to \( k = 5 \). This gives us all \( k \)-rational polygons with
exactly one interior lattice point, see
Theorem~\ref{thm:main_classification}.

\begin{figure}
	\caption{All ten $2$-maximal rational polygons with exactly one
		interior integral point up to affine unimodular equivalence. The
		threefold standard triangle is the only one that cannot be realized in
		\( \R \times [-1,1] \).}
	\label{fig:two_maximal_polygons_one_interior_lattice_point}
	\begin{tikzpicture}[x=0.3cm,y=0.3cm]

		\draw[step=2,black,line width=1.0pt] (-18.5,-8.5) grid (14.5,2.5);
		\draw[help lines,step=1,black,line width=0.2pt] (-18.5,-8.5) grid (14.5,2.5);

		\fill[opacity=0.2]
		(-6,-2)--(-10,2)--(-14,-2)--cycle;
		\draw[line width=2.0pt, color=black]
		(-6,-2)--(-10,2)--(-14,-2)--cycle;
		\draw [fill=black] (-6,-2) circle (2.0pt);
		\draw [fill=black] (-10,2) circle (2.0pt);
		\draw [fill=black] (-14,-2) circle (2.0pt);

		\fill[opacity=0.2]
		(-18,-5)--(-14,-3)--(-14,-8)--(-18,-8)--cycle;
		\draw[line width=2.0pt, color=black]
		(-18,-5)--(-14,-3)--(-14,-8)--(-18,-8)--cycle;
		\draw [fill=black] (-18,-5) circle (2.0pt);
		\draw [fill=black] (-14,-3) circle (2.0pt);
		\draw [fill=black] (-14,-8) circle (2.0pt);
		\draw [fill=black] (-18,-8) circle (2.0pt);

		\fill[opacity=0.2]
		(-3,-4)--(-1,-8)--(-5,-8)--(-7,-4)--cycle;
		\draw[line width=2.0pt, color=black]
		(-3,-4)--(-1,-8)--(-5,-8)--(-7,-4)--cycle;
		\draw [fill=black] (-3,-4) circle (2.0pt);
		\draw [fill=black] (-1,-8) circle (2.0pt);
		\draw [fill=black] (-5,-8) circle (2.0pt);
		\draw [fill=black] (-7,-4) circle (2.0pt);

		\fill[opacity=0.2]
		(-1,2)--(-3,-2)--(-9,2)--cycle;
		\draw[line width=2.0pt, color=black]
		(-1,2)--(-3,-2)--(-9,2)--cycle;
		\draw [fill=black] (-1,2) circle (2.0pt);
		\draw [fill=black] (-3,-2) circle (2.0pt);
		\draw [fill=black] (-9,2) circle (2.0pt);

		\fill[opacity=0.2]
		(1,1)--(2,1)--(8,-2)--(-2,-2)--cycle;
		\draw[line width=2.0pt, color=black]
		(1,1)--(2,1)--(8,-2)--(-2,-2)--cycle;
		\draw [fill=black] (1,1) circle (2.0pt);
		\draw [fill=black] (2,1) circle (2.0pt);
		\draw [fill=black] (8,-2) circle (2.0pt);
		\draw [fill=black] (-2,-2) circle (2.0pt);

		\fill[opacity=0.2]
		(1,-8)--(3,-8)--(5,-4)--(-1,-4)--cycle;
		\draw[line width=2.0pt, color=black]
		(1,-8)--(3,-8)--(5,-4)--(-1,-4)--cycle;
		\draw [fill=black] (1,-8) circle (2.0pt);
		\draw [fill=black] (3,-8) circle (2.0pt);
		\draw [fill=black] (5,-4) circle (2.0pt);
		\draw [fill=black] (-1,-4) circle (2.0pt);

		\fill[opacity=0.2]
		(-8,-4)--(-12,-4)--(-12,-8)--(-8,-8)--cycle;
		\draw[line width=2.0pt, color=black]
		(-8,-4)--(-12,-4)--(-12,-8)--(-8,-8)--cycle;
		\draw [fill=black] (-8,-4) circle (2.0pt);
		\draw [fill=black] (-12,-4) circle (2.0pt);
		\draw [fill=black] (-12,-8) circle (2.0pt);
		\draw [fill=black] (-8,-8) circle (2.0pt);

		\fill[opacity=0.2]
		(8,-1)--(14,2)--(2,2)--cycle;
		\draw[line width=2.0pt, color=black]
		(8,-1)--(14,2)--(2,2)--cycle;
		\draw [fill=black] (8,-1) circle (2.0pt);
		\draw [fill=black] (14,2) circle (2.0pt);
		\draw [fill=black] (2,2) circle (2.0pt);

		\fill[opacity=0.2]
		(7,-4)--(8,-4)--(12,-8)--(5,-8)--cycle;
		\draw[line width=2.0pt, color=black]
		(7,-4)--(8,-4)--(12,-8)--(5,-8)--cycle;
		\draw [fill=black] (7,-4) circle (2.0pt);
		\draw [fill=black] (8,-4) circle (2.0pt);
		\draw [fill=black] (12,-8) circle (2.0pt);
		\draw [fill=black] (5,-8) circle (2.0pt);

		\fill[opacity=0.2]
		(-18,2)--(-12,2)--(-18,-4)--cycle;
		\draw[line width=2.0pt, color=black]
		(-18,2)--(-12,2)--(-18,-4)--cycle;
		\draw [fill=black] (-18,2) circle (2.0pt);
		\draw [fill=black] (-12,2) circle (2.0pt);
		\draw [fill=black] (-18,-4) circle (2.0pt);

	\end{tikzpicture}

\end{figure}
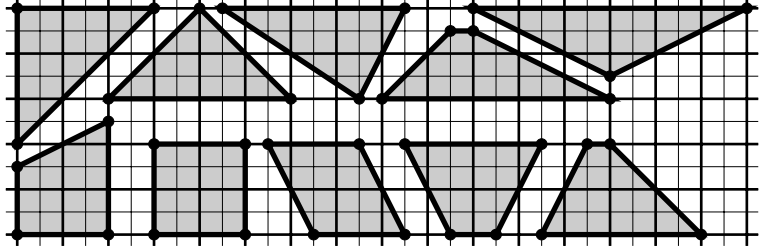

\begin{figure}
	\caption{All thirtynine \( 3 \)-maximal polygons with exactly one
		interior lattice points. The upper four rows are the 26 polygons
		in \( \R \times [-1,1] \). The lower two rows contain the 12 polygons in
		\( \R \times [-1,2] \) and the rightmost polygon in the last row is
		the only polygon in \( \R \times [-2,2] \).}
	\label{fig:three_maximal_polygons_one_interior_lattice_point}
	\begin{tikzpicture}[x=0.2cm,y=0.2cm]

		\draw[step=3,black,line width=1.0pt] (-15.5,-49.5) grid (45.5,6.5);
		\draw[help lines,step=1,black,line width=0.2pt] (-15.5,-49.5) grid (45.5,6.5);

		\fill[opacity=0.2]
		(-15,-22)--(0,-27)--(-15,-27)--cycle;
		\draw[line width=2.0pt, color=black]
		(-15,-22)--(0,-27)--(-15,-27)--cycle;
		\draw [fill=black] (-15,-22) circle (2.0pt);
		\draw [fill=black] (0,-27) circle (2.0pt);
		\draw [fill=black] (-15,-27) circle (2.0pt);

		\fill[opacity=0.2]
		(26,-21)--(28,-21)--(32,-27)--(22,-27)--cycle;
		\draw[line width=2.0pt, color=black]
		(26,-21)--(28,-21)--(32,-27)--(22,-27)--cycle;
		\draw [fill=black] (26,-21) circle (2.0pt);
		\draw [fill=black] (28,-21) circle (2.0pt);
		\draw [fill=black] (32,-27) circle (2.0pt);
		\draw [fill=black] (22,-27) circle (2.0pt);

		\fill[opacity=0.2]
		(11,-21)--(12,-22)--(18,-26)--(20,-26)--(22,-22)--(22,-21)--cycle;
		\draw[line width=2.0pt, color=black]
		(11,-21)--(12,-22)--(18,-26)--(20,-26)--(22,-22)--(22,-21)--cycle;
		\draw [fill=black] (11,-21) circle (2.0pt);
		\draw [fill=black] (12,-22) circle (2.0pt);
		\draw [fill=black] (18,-26) circle (2.0pt);
		\draw [fill=black] (20,-26) circle (2.0pt);
		\draw [fill=black] (22,-22) circle (2.0pt);
		\draw [fill=black] (22,-21) circle (2.0pt);

		\fill[opacity=0.2]
		(1,-25)--(3,-25)--(3,-21)--(-15,-21)--cycle;
		\draw[line width=2.0pt, color=black]
		(1,-25)--(3,-25)--(3,-21)--(-15,-21)--cycle;
		\draw [fill=black] (1,-25) circle (2.0pt);
		\draw [fill=black] (3,-25) circle (2.0pt);
		\draw [fill=black] (3,-21) circle (2.0pt);
		\draw [fill=black] (-15,-21) circle (2.0pt);

		\fill[opacity=0.2]
		(-5,6)--(-1,6)--(1,0)--(-7,0)--cycle;
		\draw[line width=2.0pt, color=black]
		(-5,6)--(-1,6)--(1,0)--(-7,0)--cycle;
		\draw [fill=black] (-5,6) circle (2.0pt);
		\draw [fill=black] (-1,6) circle (2.0pt);
		\draw [fill=black] (1,0) circle (2.0pt);
		\draw [fill=black] (-7,0) circle (2.0pt);

		\fill[opacity=0.2]
		(17,-18)--(18,-17)--(24,-13)--(26,-12)--(28,-18)--cycle;
		\draw[line width=2.0pt, color=black]
		(17,-18)--(18,-17)--(24,-13)--(26,-12)--(28,-18)--cycle;
		\draw [fill=black] (17,-18) circle (2.0pt);
		\draw [fill=black] (18,-17) circle (2.0pt);
		\draw [fill=black] (24,-13) circle (2.0pt);
		\draw [fill=black] (26,-12) circle (2.0pt);
		\draw [fill=black] (28,-18) circle (2.0pt);

		\fill[opacity=0.2]
		(7,6)--(8,5)--(10,1)--(10,0)--(2,0)--(4,6)--cycle;
		\draw[line width=2.0pt, color=black]
		(7,6)--(8,5)--(10,1)--(10,0)--(2,0)--(4,6)--cycle;
		\draw [fill=black] (7,6) circle (2.0pt);
		\draw [fill=black] (8,5) circle (2.0pt);
		\draw [fill=black] (10,1) circle (2.0pt);
		\draw [fill=black] (10,0) circle (2.0pt);
		\draw [fill=black] (2,0) circle (2.0pt);
		\draw [fill=black] (4,6) circle (2.0pt);

		\fill[opacity=0.2]
		(5,-12)--(6,-12)--(12,-18)--(1,-18)--cycle;
		\draw[line width=2.0pt, color=black]
		(5,-12)--(6,-12)--(12,-18)--(1,-18)--cycle;
		\draw [fill=black] (5,-12) circle (2.0pt);
		\draw [fill=black] (6,-12) circle (2.0pt);
		\draw [fill=black] (12,-18) circle (2.0pt);
		\draw [fill=black] (1,-18) circle (2.0pt);

		\fill[opacity=0.2]
		(-2,-3)--(-1,-4)--(1,-8)--(1,-9)--(-6,-9)--(-6,-3)--cycle;
		\draw[line width=2.0pt, color=black]
		(-2,-3)--(-1,-4)--(1,-8)--(1,-9)--(-6,-9)--(-6,-3)--cycle;
		\draw [fill=black] (-2,-3) circle (2.0pt);
		\draw [fill=black] (-1,-4) circle (2.0pt);
		\draw [fill=black] (1,-8) circle (2.0pt);
		\draw [fill=black] (1,-9) circle (2.0pt);
		\draw [fill=black] (-6,-9) circle (2.0pt);
		\draw [fill=black] (-6,-3) circle (2.0pt);

		\fill[opacity=0.2]
		(32,-21)--(34,-27)--(44,-21)--cycle;
		\draw[line width=2.0pt, color=black]
		(32,-21)--(34,-27)--(44,-21)--cycle;
		\draw [fill=black] (32,-21) circle (2.0pt);
		\draw [fill=black] (34,-27) circle (2.0pt);
		\draw [fill=black] (44,-21) circle (2.0pt);

		\fill[opacity=0.2]
		(16,6)--(20,0)--(14,0)--(10,6)--cycle;
		\draw[line width=2.0pt, color=black]
		(16,6)--(20,0)--(14,0)--(10,6)--cycle;
		\draw [fill=black] (16,6) circle (2.0pt);
		\draw [fill=black] (20,0) circle (2.0pt);
		\draw [fill=black] (14,0) circle (2.0pt);
		\draw [fill=black] (10,6) circle (2.0pt);

		\fill[opacity=0.2]
		(-9,6)--(-15,6)--(-15,0)--(-9,0)--cycle;
		\draw[line width=2.0pt, color=black]
		(-9,6)--(-15,6)--(-15,0)--(-9,0)--cycle;
		\draw [fill=black] (-9,6) circle (2.0pt);
		\draw [fill=black] (-15,6) circle (2.0pt);
		\draw [fill=black] (-15,0) circle (2.0pt);
		\draw [fill=black] (-9,0) circle (2.0pt);

		\fill[opacity=0.2]
		(25,6)--(26,5)--(28,1)--(28,0)--(23,0)--(22,1)--(20,5)--(20,6)--cycle;
		\draw[line width=2.0pt, color=black]
		(25,6)--(26,5)--(28,1)--(28,0)--(23,0)--(22,1)--(20,5)--(20,6)--cycle;
		\draw [fill=black] (25,6) circle (2.0pt);
		\draw [fill=black] (26,5) circle (2.0pt);
		\draw [fill=black] (28,1) circle (2.0pt);
		\draw [fill=black] (28,0) circle (2.0pt);
		\draw [fill=black] (23,0) circle (2.0pt);
		\draw [fill=black] (22,1) circle (2.0pt);
		\draw [fill=black] (20,5) circle (2.0pt);
		\draw [fill=black] (20,6) circle (2.0pt);

		\fill[opacity=0.2]
		(32,-3)--(31,-4)--(29,-8)--(29,-9)--(38,-9)--(34,-3)--cycle;
		\draw[line width=2.0pt, color=black]
		(32,-3)--(31,-4)--(29,-8)--(29,-9)--(38,-9)--(34,-3)--cycle;
		\draw [fill=black] (32,-3) circle (2.0pt);
		\draw [fill=black] (31,-4) circle (2.0pt);
		\draw [fill=black] (29,-8) circle (2.0pt);
		\draw [fill=black] (29,-9) circle (2.0pt);
		\draw [fill=black] (38,-9) circle (2.0pt);
		\draw [fill=black] (34,-3) circle (2.0pt);

		\fill[opacity=0.2]
		(-2,-17)--(-1,-17)--(1,-13)--(1,-12)--(-12,-12)--cycle;
		\draw[line width=2.0pt, color=black]
		(-2,-17)--(-1,-17)--(1,-13)--(1,-12)--(-12,-12)--cycle;
		\draw [fill=black] (-2,-17) circle (2.0pt);
		\draw [fill=black] (-1,-17) circle (2.0pt);
		\draw [fill=black] (1,-13) circle (2.0pt);
		\draw [fill=black] (1,-12) circle (2.0pt);
		\draw [fill=black] (-12,-12) circle (2.0pt);

		\fill[opacity=0.2]
		(22,-12)--(21,-13)--(15,-17)--(9,-13)--(8,-12)--cycle;
		\draw[line width=2.0pt, color=black]
		(22,-12)--(21,-13)--(15,-17)--(9,-13)--(8,-12)--cycle;
		\draw [fill=black] (22,-12) circle (2.0pt);
		\draw [fill=black] (21,-13) circle (2.0pt);
		\draw [fill=black] (15,-17) circle (2.0pt);
		\draw [fill=black] (9,-13) circle (2.0pt);
		\draw [fill=black] (8,-12) circle (2.0pt);

		\fill[opacity=0.2]
		(-3,-18)--(-15,-12)--(-15,-18)--cycle;
		\draw[line width=2.0pt, color=black]
		(-3,-18)--(-15,-12)--(-15,-18)--cycle;
		\draw [fill=black] (-3,-18) circle (2.0pt);
		\draw [fill=black] (-15,-12) circle (2.0pt);
		\draw [fill=black] (-15,-18) circle (2.0pt);

		\fill[opacity=0.2]
		(43,-18)--(43,-17)--(39,3)--(39,-18)--cycle;
		\draw[line width=2.0pt, color=black]
		(43,-18)--(43,-17)--(39,3)--(39,-18)--cycle;
		\draw [fill=black] (43,-18) circle (2.0pt);
		\draw [fill=black] (43,-17) circle (2.0pt);
		\draw [fill=black] (39,3) circle (2.0pt);
		\draw [fill=black] (39,-18) circle (2.0pt);

		\fill[opacity=0.2]
		(34,6)--(35,5)--(37,1)--(37,0)--(31,0)--(29,6)--cycle;
		\draw[line width=2.0pt, color=black]
		(34,6)--(35,5)--(37,1)--(37,0)--(31,0)--(29,6)--cycle;
		\draw [fill=black] (34,6) circle (2.0pt);
		\draw [fill=black] (35,5) circle (2.0pt);
		\draw [fill=black] (37,1) circle (2.0pt);
		\draw [fill=black] (37,0) circle (2.0pt);
		\draw [fill=black] (31,0) circle (2.0pt);
		\draw [fill=black] (29,6) circle (2.0pt);

		\fill[opacity=0.2]
		(7,-3)--(11,-9)--(3,-9)--(3,-3)--cycle;
		\draw[line width=2.0pt, color=black]
		(7,-3)--(11,-9)--(3,-9)--(3,-3)--cycle;
		\draw [fill=black] (7,-3) circle (2.0pt);
		\draw [fill=black] (11,-9) circle (2.0pt);
		\draw [fill=black] (3,-9) circle (2.0pt);
		\draw [fill=black] (3,-3) circle (2.0pt);

		\fill[opacity=0.2]
		(-8,-9)--(-10,-3)--(-15,-3)--(-15,-9)--cycle;
		\draw[line width=2.0pt, color=black]
		(-8,-9)--(-10,-3)--(-15,-3)--(-15,-9)--cycle;
		\draw [fill=black] (-8,-9) circle (2.0pt);
		\draw [fill=black] (-10,-3) circle (2.0pt);
		\draw [fill=black] (-15,-3) circle (2.0pt);
		\draw [fill=black] (-15,-9) circle (2.0pt);

		\fill[opacity=0.2]
		(41,6)--(45,6)--(45,-18)--cycle;
		\draw[line width=2.0pt, color=black]
		(41,6)--(45,6)--(45,-18)--cycle;
		\draw [fill=black] (41,6) circle (2.0pt);
		\draw [fill=black] (45,6) circle (2.0pt);
		\draw [fill=black] (45,-18) circle (2.0pt);

		\fill[opacity=0.2]
		(14,-9)--(10,-3)--(17,-3)--(19,-9)--cycle;
		\draw[line width=2.0pt, color=black]
		(14,-9)--(10,-3)--(17,-3)--(19,-9)--cycle;
		\draw [fill=black] (14,-9) circle (2.0pt);
		\draw [fill=black] (10,-3) circle (2.0pt);
		\draw [fill=black] (17,-3) circle (2.0pt);
		\draw [fill=black] (19,-9) circle (2.0pt);

		\fill[opacity=0.2]
		(25,-3)--(26,-4)--(28,-8)--(28,-9)--(20,-9)--(20,-8)--(22,-4)--(23,-3)--cycle;
		\draw[line width=2.0pt, color=black]
		(25,-3)--(26,-4)--(28,-8)--(28,-9)--(20,-9)--(20,-8)--(22,-4)--(23,-3)--cycle;
		\draw [fill=black] (25,-3) circle (2.0pt);
		\draw [fill=black] (26,-4) circle (2.0pt);
		\draw [fill=black] (28,-8) circle (2.0pt);
		\draw [fill=black] (28,-9) circle (2.0pt);
		\draw [fill=black] (20,-9) circle (2.0pt);
		\draw [fill=black] (20,-8) circle (2.0pt);
		\draw [fill=black] (22,-4) circle (2.0pt);
		\draw [fill=black] (23,-3) circle (2.0pt);

		\fill[opacity=0.2]
		(31,-18)--(34,-18)--(38,-12)--(29,-12)--cycle;
		\draw[line width=2.0pt, color=black]
		(31,-18)--(34,-18)--(38,-12)--(29,-12)--cycle;
		\draw [fill=black] (31,-18) circle (2.0pt);
		\draw [fill=black] (34,-18) circle (2.0pt);
		\draw [fill=black] (38,-12) circle (2.0pt);
		\draw [fill=black] (29,-12) circle (2.0pt);

		\fill[opacity=0.2]
		(8,-21)--(7,-21)--(5,-27)--(16,-27)--cycle;
		\draw[line width=2.0pt, color=black]
		(8,-21)--(7,-21)--(5,-27)--(16,-27)--cycle;
		\draw [fill=black] (8,-21) circle (2.0pt);
		\draw [fill=black] (7,-21) circle (2.0pt);
		\draw [fill=black] (5,-27) circle (2.0pt);
		\draw [fill=black] (16,-27) circle (2.0pt);

		\fill[opacity=0.2]
		(40,-37)--(41,-37)--(43,-35)--(40,-29)--(36,-33)--cycle;
		\draw[line width=2.0pt, color=black]
		(40,-37)--(41,-37)--(43,-35)--(40,-29)--(36,-33)--cycle;
		\draw [fill=black] (40,-37) circle (2.0pt);
		\draw [fill=black] (41,-37) circle (2.0pt);
		\draw [fill=black] (43,-35) circle (2.0pt);
		\draw [fill=black] (40,-29) circle (2.0pt);
		\draw [fill=black] (36,-33) circle (2.0pt);

		\fill[opacity=0.2]
		(20,-39)--(17,-30)--(14,-39)--cycle;
		\draw[line width=2.0pt, color=black]
		(20,-39)--(17,-30)--(14,-39)--cycle;
		\draw [fill=black] (20,-39) circle (2.0pt);
		\draw [fill=black] (17,-30) circle (2.0pt);
		\draw [fill=black] (14,-39) circle (2.0pt);

		\fill[opacity=0.2]
		(30,-32)--(31,-32)--(35,-34)--(37,-38)--(37,-39)--(30,-39)--cycle;
		\draw[line width=2.0pt, color=black]
		(30,-32)--(31,-32)--(35,-34)--(37,-38)--(37,-39)--(30,-39)--cycle;
		\draw [fill=black] (30,-32) circle (2.0pt);
		\draw [fill=black] (31,-32) circle (2.0pt);
		\draw [fill=black] (35,-34) circle (2.0pt);
		\draw [fill=black] (37,-38) circle (2.0pt);
		\draw [fill=black] (37,-39) circle (2.0pt);
		\draw [fill=black] (30,-39) circle (2.0pt);

		\fill[opacity=0.2]
		(6,-44)--(7,-43)--(11,-41)--(13,-43)--(11,-47)--(10,-48)--(6,-48)--cycle;
		\draw[line width=2.0pt, color=black]
		(6,-44)--(7,-43)--(11,-41)--(13,-43)--(11,-47)--(10,-48)--(6,-48)--cycle;
		\draw [fill=black] (6,-44) circle (2.0pt);
		\draw [fill=black] (7,-43) circle (2.0pt);
		\draw [fill=black] (11,-41) circle (2.0pt);
		\draw [fill=black] (13,-43) circle (2.0pt);
		\draw [fill=black] (11,-47) circle (2.0pt);
		\draw [fill=black] (10,-48) circle (2.0pt);
		\draw [fill=black] (6,-48) circle (2.0pt);

		\fill[opacity=0.2]
		(19,-48)--(20,-47)--(22,-43)--(20,-41)--(16,-43)--(14,-47)--(14,-48)--cycle;
		\draw[line width=2.0pt, color=black]
		(19,-48)--(20,-47)--(22,-43)--(20,-41)--(16,-43)--(14,-47)--(14,-48)--cycle;
		\draw [fill=black] (19,-48) circle (2.0pt);
		\draw [fill=black] (20,-47) circle (2.0pt);
		\draw [fill=black] (22,-43) circle (2.0pt);
		\draw [fill=black] (20,-41) circle (2.0pt);
		\draw [fill=black] (16,-43) circle (2.0pt);
		\draw [fill=black] (14,-47) circle (2.0pt);
		\draw [fill=black] (14,-48) circle (2.0pt);

		\fill[opacity=0.2]
		(-4,-44)--(4,-40)--(0,-48)--cycle;
		\draw[line width=2.0pt, color=black]
		(-4,-44)--(4,-40)--(0,-48)--cycle;
		\draw [fill=black] (-4,-44) circle (2.0pt);
		\draw [fill=black] (4,-40) circle (2.0pt);
		\draw [fill=black] (0,-48) circle (2.0pt);

		\fill[opacity=0.2]
		(0,-29)--(1,-30)--(7,-39)--(0,-39)--cycle;
		\draw[line width=2.0pt, color=black]
		(0,-29)--(1,-30)--(7,-39)--(0,-39)--cycle;
		\draw [fill=black] (0,-29) circle (2.0pt);
		\draw [fill=black] (1,-30) circle (2.0pt);
		\draw [fill=black] (7,-39) circle (2.0pt);
		\draw [fill=black] (0,-39) circle (2.0pt);

		\fill[opacity=0.2]
		(-3,-32)--(-5,-33)--(-13,-39)--(-3,-39)--cycle;
		\draw[line width=2.0pt, color=black]
		(-3,-32)--(-5,-33)--(-13,-39)--(-3,-39)--cycle;
		\draw [fill=black] (-3,-32) circle (2.0pt);
		\draw [fill=black] (-5,-33) circle (2.0pt);
		\draw [fill=black] (-13,-39) circle (2.0pt);
		\draw [fill=black] (-3,-39) circle (2.0pt);

		\fill[opacity=0.2]
		(-15,-30)--(-6,-30)--(-15,-39)--cycle;
		\draw[line width=2.0pt, color=black]
		(-15,-30)--(-6,-30)--(-15,-39)--cycle;
		\draw [fill=black] (-15,-30) circle (2.0pt);
		\draw [fill=black] (-6,-30) circle (2.0pt);
		\draw [fill=black] (-15,-39) circle (2.0pt);

		\fill[opacity=0.2]
		(13,-32)--(7,-29)--(5,-31)--(8,-37)--cycle;
		\draw[line width=2.0pt, color=black]
		(13,-32)--(7,-29)--(5,-31)--(8,-37)--cycle;
		\draw [fill=black] (13,-32) circle (2.0pt);
		\draw [fill=black] (7,-29) circle (2.0pt);
		\draw [fill=black] (5,-31) circle (2.0pt);
		\draw [fill=black] (8,-37) circle (2.0pt);

		\fill[opacity=0.2]
		(-8,-43)--(-5,-46)--(-11,-49)--(-14,-46)--cycle;
		\draw[line width=2.0pt, color=black]
		(-8,-43)--(-5,-46)--(-11,-49)--(-14,-46)--cycle;
		\draw [fill=black] (-8,-43) circle (2.0pt);
		\draw [fill=black] (-5,-46) circle (2.0pt);
		\draw [fill=black] (-11,-49) circle (2.0pt);
		\draw [fill=black] (-14,-46) circle (2.0pt);

		\fill[opacity=0.2]
		(25,-38)--(26,-36)--(28,-30)--(21,-30)--cycle;
		\draw[line width=2.0pt, color=black]
		(25,-38)--(26,-36)--(28,-30)--(21,-30)--cycle;
		\draw [fill=black] (25,-38) circle (2.0pt);
		\draw [fill=black] (26,-36) circle (2.0pt);
		\draw [fill=black] (28,-30) circle (2.0pt);
		\draw [fill=black] (21,-30) circle (2.0pt);

		\fill[opacity=0.2]
		(25,-43)--(29,-41)--(31,-43)--(29,-47)--(25,-49)--(23,-47)--cycle;
		\draw[line width=2.0pt, color=black]
		(25,-43)--(29,-41)--(31,-43)--(29,-47)--(25,-49)--(23,-47)--cycle;
		\draw [fill=black] (25,-43) circle (2.0pt);
		\draw [fill=black] (29,-41) circle (2.0pt);
		\draw [fill=black] (31,-43) circle (2.0pt);
		\draw [fill=black] (29,-47) circle (2.0pt);
		\draw [fill=black] (25,-49) circle (2.0pt);
		\draw [fill=black] (23,-47) circle (2.0pt);

	\end{tikzpicture}
\end{figure}
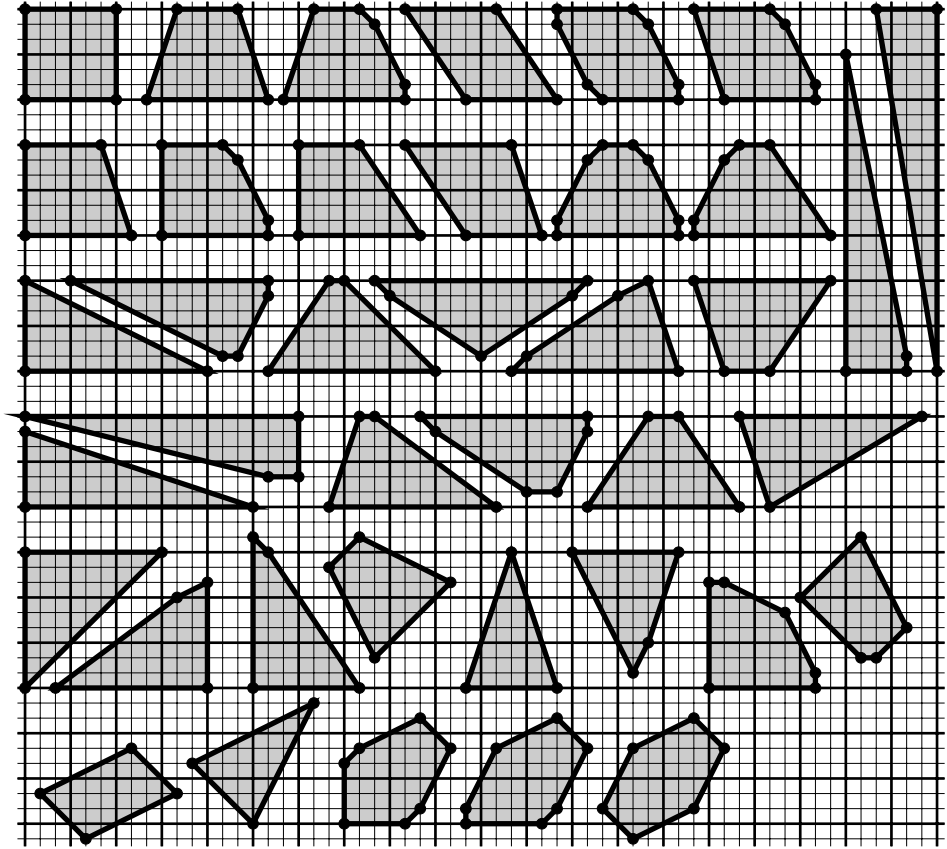

We turn to LDP polygons, which are of special interest to toric
geometry. Recall that a lattice polygon \( P \) is called \emph{LDP},
if \( 0 \in P^{\circ} \) and its vertices are primitive vectors of \(
\Z^2 \). By taking the fan generated by its faces, an LDP polygon \(
P \) defines a toric log del Pezzo surface \( X_P \) and this
assignment is a bijection on isomorphy classes. We refer
to~\cite{KKN10} for more information on LDP polygons as well as their
classification up to index~17.

In~\cite{HHS23}, the authors describe a classification algorithm for
\emph{almost \( k \)-hollow} LDP polygons, i.e. LDP polygons with \(
P \cap k\Z^2 = \{0\} \). These correspond to toric \( \frac{1}{k}
\)-log canonical del Pezzo surfaces, which means the discrepancies of
the exceptional divisors in the canonical resolution are all greater
or equal to \( \frac{1}{k} - 1 \). Dividing an almost \( k \)-hollow
LDP polygon by \( k \) gives a \( k \)-rational polygon having only
the origin as an interior lattice point. Hence we can also classify
\( k \)-hollow LDP polygons by taking our \( k \)-maximal polygons
with one interior lattice point and computing all subpolygons having
primitive vertices. We have done this up to \( k = 6 \), in
particular reproducing the classification of the almost \( 3
\)-hollow LDP polygons from~\cite{HHS23}:

\begin{class}
	\label{class:ldp_polygons}

	Below are the numbers \( \#_{\ldp}(k) \) of almost \( k \)-hollow LDP
	polygons, their maximimally attained number of vertices \(
	N_{\ldp}(k) \) and the number of vertex maximizers \( M_{\ldp}(k) \).
	Equivalently, \( \#_{\ldp}(k) \) is the number of toric \(
	\frac{1}{k} \)-log canonical del Pezzo surfaces and \( N_{\ldp}(k) -
	2 \) is their maximally attained Picard number. For comparison, we
	include the analogous numbers for all \( k \)-rational polygons with
	exactly one interior lattice points.
	\begin{center}
		{\small
			\begin{tabular}{c|cccccc}
				\( k \)                 & 1                & 2      & 3        & 4           & 5  & 6  \\\hline
				\( \#_{\mathrm{ldp}} \) & 16               & 505    & 48\,032  & 1\,741\,603 &
				154\,233\,886           & 2\,444\,400\,116                                             \\
				\( N_{\mathrm{ldp}} \)  & 6                & 8      & 12       & 12          & 16 & 18 \\
				\( M_{\mathrm{ldp}} \)  & 1                & 1      & 1        & 20          & 1  & 1  \\\hline
				\( \#_{\mathrm{all}} \) & 16               & 5\,145 & 924\,042 &
				101\,267\,212           & 8\,544\,548\,186 &                                           \\
				\( N_{\mathrm{all}} \)  & 6                & 9      & 12       & 14          & 18 &    \\
				\( M_{\mathrm{all}} \)  & 1                & 1      & 2        & 23          & 1  &    \\
			\end{tabular}
		}
	\end{center}

\end{class}

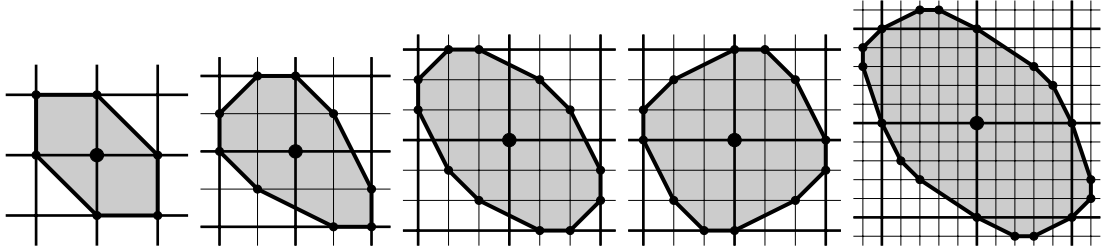
\begin{figure}
	\caption{The \( k \)-rational polygons with exactly one interior
		lattice point and maximal number of vertices for \( k \in
		\{1,2,3,5\} \).}
	\label{fig:maximal_number_of_vertices}
	\begin{tikzpicture}[x=0.8cm,y=0.8cm]

		\draw[step=1,black,line width=1pt] (-1.5,-1.5) grid (1.5,1.5);
		\draw[fill=black] (0,0) circle (2.5pt);

		\fill[opacity=0.2] (1,0) -- (0,1) -- (-1,1) -- (-1,0) --
		(0,-1) -- (1,-1) -- cycle;

		\draw[line width=1.5pt,color=black] (1,0) -- (0,1) -- (-1,1) -- (-1,0) -- (0,-1) -- (1,-1) -- cycle;

		\draw [fill=black] (1,0) circle (1.5pt);
		\draw [fill=black] (0,1) circle (1.5pt);
		\draw [fill=black] (-1,1) circle (1.5pt);
		\draw [fill=black] (-1,0) circle (1.5pt);
		\draw [fill=black] (0,-1) circle (1.5pt);
		\draw [fill=black] (1,-1) circle (1.5pt);

	\end{tikzpicture}
	\begin{tikzpicture}[x=0.5cm,y=0.5cm]
		\draw[step=2,black,line width=1pt] (-2.5,-2.5) grid (2.5,2.5);
		\draw[help lines,step=1,black,line width=0.2pt] (-2.5,-2.5) grid (2.5,2.5);
		\draw [fill=black] (0,0) circle (2.5pt);

		\fill[line width=2pt,color=black,fill=black,fill,opacity=0.2]
		(1,1)--(0,2)--(-1,2)--(-2,1)--(-2,0)--(-1,-1)--(1,-2)--(2,-2)--(2,-1)--cycle;

		\draw[line width=1.5pt, color=black]
		(1,1)--(0,2)--(-1,2)--(-2,1)--(-2,0)--(-1,-1)--(1,-2)--(2,-2)--(2,-1)--cycle;

		\draw [fill=black] (1,1) circle (1.5pt);
		\draw [fill=black] (0,2) circle (1.5pt);
		\draw [fill=black] (-1,2) circle (1.5pt);
		\draw [fill=black] (-2,1) circle (1.5pt);
		\draw [fill=black] (-2,0) circle (1.5pt);
		\draw [fill=black] (-1,-1) circle (1.5pt);
		\draw [fill=black] (1,-2) circle (1.5pt);
		\draw [fill=black] (2,-2) circle (1.5pt);
		\draw [fill=black] (2,-1) circle (1.5pt);

	\end{tikzpicture}
	\begin{tikzpicture}[x=0.4cm,y=0.4cm]
		\draw[step=3,black,line width=1pt] (-3.5,-3.5) grid (3.5,3.5);
		\draw[help lines,step=1,black,line width=0.2pt] (-3.5,-3.5) grid (3.5,3.5);
		\draw [fill=black] (0,0) circle (2.5pt);

		\fill[line width=2pt,color=black,fill=black,fill,opacity=0.2]
		(1,-3)--(2,-3)--(3,-2)--(3,-1)--(2,1)--(1,2)--(-1,3)--(-2,3)--(-3,2)--(-3,1)--(-2,-1)--(-1,-2)--cycle;
		\draw[line width=1.5pt, color=black]
		(1,-3)--(2,-3)--(3,-2)--(3,-1)--(2,1)--(1,2)--(-1,3)--(-2,3)--(-3,2)--(-3,1)--(-2,-1)--(-1,-2)--cycle;\draw [fill=black] (1,-3) circle (1.5pt);
		\draw [fill=black] (2,-3) circle (1.5pt);
		\draw [fill=black] (3,-2) circle (1.5pt);
		\draw [fill=black] (3,-1) circle (1.5pt);
		\draw [fill=black] (2,1) circle (1.5pt);
		\draw [fill=black] (1,2) circle (1.5pt);
		\draw [fill=black] (-1,3) circle (1.5pt);
		\draw [fill=black] (-2,3) circle (1.5pt);
		\draw [fill=black] (-3,2) circle (1.5pt);
		\draw [fill=black] (-3,1) circle (1.5pt);
		\draw [fill=black] (-2,-1) circle (1.5pt);
		\draw [fill=black] (-1,-2) circle (1.5pt);

	\end{tikzpicture}
	\begin{tikzpicture}[x=0.4cm,y=0.4cm]
		\draw[step=3,black,line width=1pt] (-3.5,-3.5) grid (3.5,3.5);
		\draw[help lines,step=1,black,line width=0.2pt] (-3.5,-3.5) grid (3.5,3.5);
		\draw [fill=black] (0,0) circle (2.5pt);
		\fill[line width=2pt,color=black,fill=black,fill,opacity=0.2]
		(1,3)--(0,3)--(-2,2)--(-3,1)--(-3,0)--(-2,-2)--(-1,-3)--(0,-3)--(2,-2)--(3,-1)--(3,0)--(2,2)--cycle;
		\draw[line width=1.5pt, color=black]
		(1,3)--(0,3)--(-2,2)--(-3,1)--(-3,0)--(-2,-2)--(-1,-3)--(0,-3)--(2,-2)--(3,-1)--(3,0)--(2,2)--cycle;\draw [fill=black] (1,3) circle (1.5pt);
		\draw [fill=black] (0,3) circle (1.5pt);
		\draw [fill=black] (-2,2) circle (1.5pt);
		\draw [fill=black] (-3,1) circle (1.5pt);
		\draw [fill=black] (-3,0) circle (1.5pt);
		\draw [fill=black] (-2,-2) circle (1.5pt);
		\draw [fill=black] (-1,-3) circle (1.5pt);
		\draw [fill=black] (0,-3) circle (1.5pt);
		\draw [fill=black] (2,-2) circle (1.5pt);
		\draw [fill=black] (3,-1) circle (1.5pt);
		\draw [fill=black] (3,0) circle (1.5pt);
		\draw [fill=black] (2,2) circle (1.5pt);
	\end{tikzpicture}
	\begin{tikzpicture}[x=0.25cm,y=0.25cm]
		\draw[step=5,black,line width=1pt] (-6.5,-6.5) grid (6.5,6.5);
		\draw[help lines,step=1,black,line width=0.2pt] (-6.5,-6.5) grid (6.5,6.5);
		\draw [fill=black] (0,0) circle (2.5pt);

		\fill[line width=2pt,color=black,fill=black,fill,opacity=0.2]
		(4,2)--(3,3)--(0,5)--(-2,6)--(-3,6)--(-5,5)--(-6,4)--(-6,3)--(-5,0)--(-4,-2)--(-3,-3)--(0,-5)--(2,-6)--(3,-6)--(5,-5)--(6,-4)--(6,-3)--(5,0)--cycle;
		\draw[line width=1.5pt, color=black]
		(4,2)--(3,3)--(0,5)--(-2,6)--(-3,6)--(-5,5)--(-6,4)--(-6,3)--(-5,0)--(-4,-2)--(-3,-3)--(0,-5)--(2,-6)--(3,-6)--(5,-5)--(6,-4)--(6,-3)--(5,0)--cycle;\draw [fill=black] (4,2) circle (1.5pt);
		\draw [fill=black] (3,3) circle (1.5pt);
		\draw [fill=black] (0,5) circle (1.5pt);
		\draw [fill=black] (-2,6) circle (1.5pt);
		\draw [fill=black] (-3,6) circle (1.5pt);
		\draw [fill=black] (-5,5) circle (1.5pt);
		\draw [fill=black] (-6,4) circle (1.5pt);
		\draw [fill=black] (-6,3) circle (1.5pt);
		\draw [fill=black] (-5,0) circle (1.5pt);
		\draw [fill=black] (-4,-2) circle (1.5pt);
		\draw [fill=black] (-3,-3) circle (1.5pt);
		\draw [fill=black] (0,-5) circle (1.5pt);
		\draw [fill=black] (2,-6) circle (1.5pt);
		\draw [fill=black] (3,-6) circle (1.5pt);
		\draw [fill=black] (5,-5) circle (1.5pt);
		\draw [fill=black] (6,-4) circle (1.5pt);
		\draw [fill=black] (6,-3) circle (1.5pt);
		\draw [fill=black] (5,0) circle (1.5pt);
	\end{tikzpicture}
\end{figure}

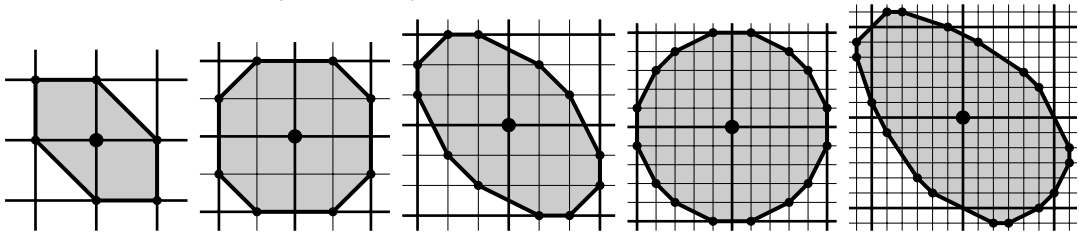
\begin{figure}
	\caption{The almost \( k \)-hollow LDP polygons with maximal number of vertices for \( k \in
		\{1,2,3,5,6\} \).}
	\label{fig:maximal_number_of_vertices_ldp}
	\begin{tikzpicture}[x=0.8cm,y=0.8cm]

		\draw[step=1,black,line width=1pt] (-1.5,-1.5) grid (1.5,1.5);
		\draw [fill=black] (0,0) circle (2.5pt);

		\fill[line width=2pt,color=black,fill=black,fill
			opacity=0.2] (1,0) -- (0,1) -- (-1,1) -- (-1,0) --
		(0,-1) -- (1,-1) -- cycle;

		\draw[line width=1.5pt,color=black] (1,0) -- (0,1) -- (-1,1) -- (-1,0) -- (0,-1) -- (1,-1) -- cycle;

		\draw [fill=black] (1,0) circle (1.5pt);
		\draw [fill=black] (0,1) circle (1.5pt);
		\draw [fill=black] (-1,1) circle (1.5pt);
		\draw [fill=black] (-1,0) circle (1.5pt);
		\draw [fill=black] (0,-1) circle (1.5pt);
		\draw [fill=black] (1,-1) circle (1.5pt);

	\end{tikzpicture}
	\begin{tikzpicture}[x=0.5cm,y=0.5cm]
		\draw[step=2,black,line width=1pt] (-2.5,-2.5) grid (2.5,2.5);
		\draw[help lines,step=1,black,line width=0.2pt] (-2.5,-2.5) grid (2.5,2.5);
		\draw [fill=black] (0,0) circle (2.5pt);

		\fill[line width=2pt,color=black,fill=black,fill,opacity=0.2]
		(1,-2)--(2,-1)--(2,1)--(1,2)--(-1,2)--(-2,1)--(-2,-1)--(-1,-2)--cycle;
		\draw[line width=1.5pt, color=black]
		(1,-2)--(2,-1)--(2,1)--(1,2)--(-1,2)--(-2,1)--(-2,-1)--(-1,-2)--cycle;\draw [fill=black] (1,-2) circle (1.5pt);
		\draw [fill=black] (2,-1) circle (1.5pt);
		\draw [fill=black] (2,1) circle (1.5pt);
		\draw [fill=black] (1,2) circle (1.5pt);
		\draw [fill=black] (-1,2) circle (1.5pt);
		\draw [fill=black] (-2,1) circle (1.5pt);
		\draw [fill=black] (-2,-1) circle (1.5pt);
		\draw [fill=black] (-1,-2) circle (1.5pt);

	\end{tikzpicture}
	\begin{tikzpicture}[x=0.4cm,y=0.4cm]
		\draw[step=3,black,line width=1pt] (-3.5,-3.5) grid (3.5,3.5);
		\draw[help lines,step=1,black,line width=0.2pt] (-3.5,-3.5) grid (3.5,3.5);
		\draw [fill=black] (0,0) circle (2.5pt);

		\fill[line width=2pt,color=black,fill=black,fill,opacity=0.2]
		(1,-3)--(2,-3)--(3,-2)--(3,-1)--(2,1)--(1,2)--(-1,3)--(-2,3)--(-3,2)--(-3,1)--(-2,-1)--(-1,-2)--cycle;
		\draw[line width=1.5pt, color=black]
		(1,-3)--(2,-3)--(3,-2)--(3,-1)--(2,1)--(1,2)--(-1,3)--(-2,3)--(-3,2)--(-3,1)--(-2,-1)--(-1,-2)--cycle;\draw [fill=black] (1,-3) circle (1.5pt);
		\draw [fill=black] (2,-3) circle (1.5pt);
		\draw [fill=black] (3,-2) circle (1.5pt);
		\draw [fill=black] (3,-1) circle (1.5pt);
		\draw [fill=black] (2,1) circle (1.5pt);
		\draw [fill=black] (1,2) circle (1.5pt);
		\draw [fill=black] (-1,3) circle (1.5pt);
		\draw [fill=black] (-2,3) circle (1.5pt);
		\draw [fill=black] (-3,2) circle (1.5pt);
		\draw [fill=black] (-3,1) circle (1.5pt);
		\draw [fill=black] (-2,-1) circle (1.5pt);
		\draw [fill=black] (-1,-2) circle (1.5pt);

	\end{tikzpicture}
	\begin{tikzpicture}[x=0.25cm,y=0.25cm]
		\draw[step=5,black,line width=1pt] (-5.5,-5.5) grid (5.5,5.5);
		\draw[help lines,step=1,black,line width=0.2pt] (-5.5,-5.5) grid (5.5,5.5);
		\draw [fill=black] (0,0) circle (2.5pt);

		\fill[line width=2pt,color=black,fill=black,fill,opacity=0.2]
		(1,-5)--(3,-4)--(4,-3)--(5,-1)--(5,1)--(4,3)--(3,4)--(1,5)--(-1,5)--(-3,4)--(-4,3)--(-5,1)--(-5,-1)--(-4,-3)--(-3,-4)--(-1,-5)--cycle;
		\draw[line width=1.5pt, color=black]
		(1,-5)--(3,-4)--(4,-3)--(5,-1)--(5,1)--(4,3)--(3,4)--(1,5)--(-1,5)--(-3,4)--(-4,3)--(-5,1)--(-5,-1)--(-4,-3)--(-3,-4)--(-1,-5)--cycle;\draw [fill=black] (1,-5) circle (1.5pt);
		\draw [fill=black] (3,-4) circle (1.5pt);
		\draw [fill=black] (4,-3) circle (1.5pt);
		\draw [fill=black] (5,-1) circle (1.5pt);
		\draw [fill=black] (5,1) circle (1.5pt);
		\draw [fill=black] (4,3) circle (1.5pt);
		\draw [fill=black] (3,4) circle (1.5pt);
		\draw [fill=black] (1,5) circle (1.5pt);
		\draw [fill=black] (-1,5) circle (1.5pt);
		\draw [fill=black] (-3,4) circle (1.5pt);
		\draw [fill=black] (-4,3) circle (1.5pt);
		\draw [fill=black] (-5,1) circle (1.5pt);
		\draw [fill=black] (-5,-1) circle (1.5pt);
		\draw [fill=black] (-4,-3) circle (1.5pt);
		\draw [fill=black] (-3,-4) circle (1.5pt);
		\draw [fill=black] (-1,-5) circle (1.5pt);
	\end{tikzpicture}
	\begin{tikzpicture}[x=0.2cm,y=0.2cm]
		\draw[step=6,black,line width=1pt] (-7.5,-7.5) grid (7.5,7.5);
		\draw[help lines,step=1,black,line width=0.2pt] (-7.5,-7.5) grid (7.5,7.5);
		\draw [fill=black] (0,0) circle (2.5pt);

		\fill[line width=2pt,color=black,fill=black,fill,opacity=0.2]
		(7,-3)--(7,-2)--(5,2)--(4,3)--(1,5)--(-1,6)--(-4,7)--(-5,7)--(-7,5)--(-7,4)--(-6,1)--(-5,-1)--(-3,-4)--(-2,-5)--(2,-7)--(3,-7)--(5,-6)--(6,-5)--cycle;
		\draw[line width=1.5pt, color=black]
		(7,-3)--(7,-2)--(5,2)--(4,3)--(1,5)--(-1,6)--(-4,7)--(-5,7)--(-7,5)--(-7,4)--(-6,1)--(-5,-1)--(-3,-4)--(-2,-5)--(2,-7)--(3,-7)--(5,-6)--(6,-5)--cycle;\draw [fill=black] (7,-3) circle (1.5pt);
		\draw [fill=black] (7,-2) circle (1.5pt);
		\draw [fill=black] (5,2) circle (1.5pt);
		\draw [fill=black] (4,3) circle (1.5pt);
		\draw [fill=black] (1,5) circle (1.5pt);
		\draw [fill=black] (-1,6) circle (1.5pt);
		\draw [fill=black] (-4,7) circle (1.5pt);
		\draw [fill=black] (-5,7) circle (1.5pt);
		\draw [fill=black] (-7,5) circle (1.5pt);
		\draw [fill=black] (-7,4) circle (1.5pt);
		\draw [fill=black] (-6,1) circle (1.5pt);
		\draw [fill=black] (-5,-1) circle (1.5pt);
		\draw [fill=black] (-3,-4) circle (1.5pt);
		\draw [fill=black] (-2,-5) circle (1.5pt);
		\draw [fill=black] (2,-7) circle (1.5pt);
		\draw [fill=black] (3,-7) circle (1.5pt);
		\draw [fill=black] (5,-6) circle (1.5pt);
		\draw [fill=black] (6,-5) circle (1.5pt);
	\end{tikzpicture}
\end{figure}

\section{Half-integral polygons}
\label{sec:half_integral_polygons}

In this section, we classify half-integral polygons with few interior
lattice points. Our method is inspired by Castryck's algorithm for
classifying of lattice polygons \cite{Cas12}, where the notion of an
\emph{internal polygon} (i.e. a polygon of the form \( P^{(1)} \)) is
replaced by the \emph{Fine interior} as introduced in
\cite[§4.2.]{Fin83} (see also \cite[II, 4 Appendix]{Rei85}). For more
information on the Fine interior and its recent use in combinatorial
mirror symmetry and combinatorial birational geometry, see
\cite{Bat17,Bat23}.

Recall that for a lattice polytope \( P \subseteq \Z^d \), we set \(
P^{(1)} := \convhull(P^{\circ} \cap \Z^d) \). On the other hand \(
P^{(-1)} \) is the rational polytope obtained by "moving out" all
facets of \( P \) by one, see section~\ref{subsec:maximal_polygons}.

\begin{defi}

	The \emph{Fine interior} of a rational polytope $P\subseteq \R^d$ is
	\[
		F(P)\ :=\ \{x\in \R^d \mid \left<x,n\right>\geq \min \{ \left<y,n\right> \mid y\in P\} +1 \  \forall n\in \Z^d\setminus \{0\}\}.
	\]
\end{defi}

\begin{rem}
	\( F(P) \) is a rational polytope. Moreover, if \( P \) is a
	lattice polytope, we have \( P^{(1)} \subseteq
	F(P) \) by \cite[2.6 f.]{Bat17}. For lattice polygons, we even have
	equality by \cite[2.9]{Bat17}. In this sense, a rational polytope
	that is a Fine interior of another one can be seen as a
	generalization of an internal polygon.
\end{rem}

We have seen a criterion for the maximality of lattice polygons using
internal polygons in \ref{prp:lattice_polygon_maximal_criterion}.
Here we focus first on the polytopes that are maximal with respect to
inclusion among those sharing same Fine interior. The following is a
direct consequence of the definition of the Fine interior.

\begin{prop}
	A full-dimensional rational polytope $P$ is the Fine interior of a rational polytope if and only if $F(P^{(-1)})=P$. In this case $P^{(-1)}$ contains every rational polytope, which has $P$ as its Fine interior.
\end{prop}

This result implies that we can classify all lattice polytopes with
given number of interior lattice points by first classifying their
Fine interiors, then getting all maximal ones by moving out the
facets and finally computing all subpolytopes. Recently \cite{Boh24}
gave a classification of two-dimensional Fine interiors of lattice
$3$-polytopes, which allows us to follow this path with a small
detour in dimension $3$. We give the essential results from
\cite{Boh24} for our classification in three steps. The first one
gives us a connection between dimension $2$ and $3$.

\begin{prop}

	The Fine interior of a lattice 3-polytope $P\subseteq \R^2\times
		[-1,1]$ with at least two interior lattice points and $\lw(P)=2$ has
	vertices in $\frac{1}{2}\Z^2 \times \{0\}$ and depends only on the
	half-integral polygon $P\cap (\R^2\times \{0\})$. Moreover, the Fine
	interior of such a lattice 3-polytope is of dimension $1$ if and only
	if $P\cap (\R^2\times \{0\})$ is affine unimodular equivalent to a
	polygon in $\R\times [-1,1] \times \{0\}$.

\end{prop}

The second result gives a suitable notion of maximality in this
context.

\begin{prop}
	\label{prp:martin_third_step}
	A half-integral polygon $Q$ is the Fine interior of a lattice 3-tope if and only if there is a half-integral polygon $P_0\subseteq \R^2$ such that $Q$ is the Fine interior of the lattice 3-tope $P:=\conv{2P_0\times \{-1\}, (0,0,1)}$. Furthermore, there is a unique half-integral polygon $Q^-$ with $F(\conv{2Q^-\times \{-1\}, (0,0,1)})=Q$, which contains every half-integral polygon $P_0'$ for which
	\begin{align*}
		F(\conv{2P_0'\times \{-1\}, (0,0,1)})=Q.
	\end{align*}
\end{prop}

And the third gives us the classification which was done there.

\begin{class}
	\label{class:fine_interiors}
	There are exactly 24\,324\,158 half-integral polygons with at most $40$ lattice points, which are Fine interiors of a lattice 3-polytope.
\end{class}

For each polygon from Classification~\ref{class:fine_interiors}, we
can compute the polygon \( Q^{-} \) from
Proposition~\ref{prp:martin_third_step} and check for \( 2
\)-maximality. This gives us all \( 2 \)-maximal polygons with up to
\( 40 \) interior lattice points that cannot be realized in \( \R
\times [-1,1] \). The latter ones can be obtained with
Algorithm~\ref{algo:classify_maximal_polygons_m1p1}, hence we arrive
at the following classification.

\begin{class}[Data available at~\cite{springer_2024_13928298}]
	\label{class:half_integral_polygons}
	Below are the numbers \( \#_{\max}(i) \) of \( 2 \)-maximal polygons with \( i
	\) interior lattice points.
	\begin{center}
		{\scriptsize
			\begin{tabular}{c|cccccccccc}
				\( i \)            & 1           & 2           & 3           & 4        & 5        & 6      & 7   & 8   & 9  & 10 \\
				\( \#_{\max}(i) \) & 10          & 25          & 33          & 63       & 106      & 178    & 274 & 476 &
				693                & 1\,058
				\\\hline
				\( i \)            & 11          & 12          & 13          & 14       & 15       & 16     & 17  & 18  & 19 & 20 \\
				\( \#_{\max}(i) \) & 1\,561      & 2\,345      & 3\,431      & 4\,970   & 6\,712   & 9\,345 &
				12\,883            & 17\,667     & 23\,809
				                   & 31\,806
				\\\hline
				\( i \)            & 21          & 22          & 23          & 24       & 25       & 26     & 27  & 28  & 29 & 30 \\
				\( \#_{\max}(i) \) & 41\,394     & 54\,071     & 70\,088     & 90\,805  & 116\,613 &
				148\,877           & 187\,818    &
				236\,793           & 296\,708    & 370\,444
				\\\hline
				\( i \)            & 31          & 32          & 33          & 34       & 35       & 36     & 37  & 38  & 39 & 40 \\
				\( \#_{\max}(i) \) & 459\,667    & 569\,098    & 698\,986    & 857\,540 &
				1\,047\,830        & 1\,277\,349 & 1\,549\,049 & 1\,875\,058 &
				2\,257\,565        & 2\,714\,004                                                                                  \\
			\end{tabular}
		}
	\end{center}
    Moreover, we obtain the numbers \( \#_{\all}(i) \) of all \( 2
    \)-rational polygons as well as the numbers \( \#_{\ehr}(i) \) of
    distinct Ehrhart quasipolynomials:
	\begin{center}
		{\scriptsize
			\begin{tabular}{c|cccccc}
				\( i \)            & 1                & 2                & 3                & 4             & 5           & 6  \\
				\( \#_{\all}(i) \) & 5\,145           & 48\,639          & 249\,540         & 893\,402      & 2\,798\,638 &
				7\,299\,589                                                                                                    \\
                \( \#_{\ehr}(i) \) & 270 & 586 & 1\,060 & 1\,701 & 2\,525 &
                3\,577 \\\hline

				\( i \)            & 7                & 8                & 9                & 10            & 11          & 12 \\
				\( \#_{\all}(i) \) & 17\,556\,059     & 41\,005\,529     & 83\,848\,960     & 170\,081\,198 &
				339\,219\,561      & 617\,457\,338                                                                             \\
                \( \#_{\ehr}(i) \) & 4\,875 & 6\,459 & 8\,343 & 10\,565 &
                13\,138 & 16\,111 \\\hline
				\( i \)            & 13               & 14               & 15               & 16                               \\
				\( \#_{\all}(i) \) & 1\,140\,740\,735 & 2\,040\,032\,893 & 3\,454\,142\,981 &
				5\,957\,874\,274                                                                                               \\
                \( \#_{\ehr}(i) \) & 19\,488 & 23\,307 & 27\,596 &
                32\,381
			\end{tabular}
		}
	\end{center}
\end{class}

We end this section with a conjecture on the number of Ehrhart
quasipolynomials of half-integral polygons. Recall that for a lattice
polygon \( P \), the Ehrhart polynomial is of the form \(
At^2+\frac{b}{2}t+1 \), where \( A \) is the euclidian area and \( b
\) the number of boundary lattice points. By Pick's formula
\cite{Pic99}, fixing \( i \) and \( b \) uniquely determines \( A \).
Moreover, Scott's inequality \cite{Sco76} states that for lattice
polygons, \( b \leq 2i+6 \) holds for \( i \geq 2 \). Lattice
polygons clearly satisfy \( b \geq 3 \) and one can show that every
intermediate value of \( b \) is also attained, see e.g. \cite[Section
1]{HS09}.
This implies that the number of distinct Ehrhart polynomials of
lattice polygons is \( 2i+4 \) for \( i \geq 2 \). Looking at our
data for \( k=2 \), we found that there is also is a polynomial
expression for the number of distinct Ehrhart quasipolynomials,
provided \( b \geq 2 \). Hence we raise the following conjecture,
which we have verified up to \( i = 16 \).

\begin{conj}
	There are exactly
	\begin{align*}
		\frac{9}{2}i^3+36i^2+\frac{175}{2}i+53
	\end{align*}
	Ehrhart quasipolynomials of half-integral polygons with $i\in
		\Z_{\geq 2}$ interior lattice points and at least $2$ boundary lattice points.
\end{conj}

\section{The general case}
\label{sec:general_case}

In this section, we show how a classification of \( k \)-maximal
polygons with \( i \in \Z_{\geq 0} \) interior lattice points can be
derived from a sufficiently large classification of lattice polygons.
First, we need a suitable upper volume bound of rational polygons in
terms of their number of interior lattice points.

\begin{lemma}
	\label{lem:volume_bound_not_m1p1}
	Let \( P \) be a rational polygon with \( k = \denom(P) \) and \( i
	:= i(P) \) which cannot be realized in \( \R \times [-1,1] \). Then we have
	\[
		\Vol_{k}(P) \leq
		\max\left(k^2(4i+5),\ \frac{1}{2}k(k+2)^2(i+1)\right).
	\]
\end{lemma}

\begin{proof}
	After an affine unimodular tranformation, we can assume that \( (0,1)
	\) is a lattice width direction of \( P \) and \( P \subseteq \R
	\times [0,\infty) \). Let \( m \in \Z \) be minimal such that \( P
	\subseteq \R \times [0,m+1] \). By assumption, we have \( m \geq 2 \). If
	\( m \geq 3 \), \cite[Proposition 3.5]{Boh23} implies
	\[
		\Vol_k(P)
		\leq 2k^2\left(2i+2+\frac{1}{2}\right) = k^2(4i+5)
	\]
	and we are done. Consider \( m = 2 \) and let \( 0 < a,b \leq 1 \) be
	minimal such that \( P \subseteq \R \times [1-a,2+b] \). If \( a+b
	\geq 1 \), \cite[Proposition 3.3]{Boh23} implies again \( \Vol_k(P)
	\leq 4k^2(i+1) \). Otherwise, it gives
	\[
		\Vol_k(P) \leq \frac{(a+b+1)^2}{a+b}(i+1)k^2,
	\]
	which attains its maximum for \( a = b = \frac{1}{k} \), hence we get
	\( \Vol_k(P) \leq \frac{1}{2}k(k+2)^2(i+1) \).
\end{proof}

\begin{prop}
	\label{prp:k_maximal_polygons_from_koelman}
	Let \( P \) be a \( k \)-maximal polygon with \( \denom(P) = k
	\geq 2 \) and \( i := i(P) \) that cannot be realized in \( \R \times [-1,1] \). Then
	there exists a lattice polygon \( Q \) such that \( P =
	\frac{1}{k}Q^{(-1)} \), where
	\[
		l(Q)\ \leq\ \frac{1}{2}\Vol_k(P) - \frac{1}{2}\ \leq\
		\max\left(k^2(2i+\frac{5}{2}),\
		\frac{1}{4}k(k+2)^2(i+1)\right) - \frac{1}{2}.
	\]
\end{prop}
\begin{proof}
	Since \( P \) is \( k \)-maximal, the lattice polygon \( kP \)
	is \( 1 \)-maximal by
	Proposition~\ref{prp:k_maximal_criterion}. Set \( Q :=
	(kP)^{(1)} \). Since \( P\) cannot be realized in \( \R \times [-1,1] \) and \( k \geq
	2 \), it follows that \( Q \) is two-dimensional. Thus
	Proposition~\ref{prp:lattice_polygon_maximal_criterion}
	implies \( P = \frac{1}{k}Q^{(-1)} \). Since \( l(Q) = i(kP)
	\), the bound for \( l(Q) \) follows with Pick's formula and
	then Lemma~\ref{lem:volume_bound_not_m1p1}.
\end{proof}

This proposition tells us that if we know all lattice polygons with
sufficiently many lattice points, we can use them to classify \( k
\)-maximal polygons with \( i \in \Z_{\geq 0} \) interior lattice
points (those in \( \R \times [-1,1] \) can be classified directly
using Algorithm~\ref{algo:classify_maximal_polygons_m1p1}). A
classification algorithm for lattice polygons by number of lattice
points was given by Koelman \cite[Section 4.4]{Koe91}. How far we
would have to know this classification for is given by the bound in
Proposition~\ref{prp:k_maximal_polygons_from_koelman}, i.e.
\[
	l_{\max}(k,i)\ :=\
	\left\lfloor
	\max\left(k^2(2i+\frac{5}{2}),\
	\frac{1}{4}k(k+2)^2(i+1)\right) - \frac{1}{2}
	\right\rfloor.
\]
Below are some values for \( l_{\max}(k,i) \), where \( k \) runs
vertically and \( i \) horizontally. Note however that this bound is
not sharp and can likely be improved.
\begin{center}
	\begin{tabular}{c|ccccccccc}
		\( l_{\max}(k,i) \) & 0   & 1   & 2   & 3   & 4   & 5   & 6   & 7    & 8
		\\\hline
		2                   & 9   & 17  & 25  & 33  & 41  & 49  & 57  & 65   & 73   \\
		3                   & 22  & 40  & 58  & 76  & 94  & 112 & 130 & 149  & 168  \\
		4                   & 39  & 71  & 107 & 143 & 179 & 215 & 251 & 287  & 323  \\
		5                   & 62  & 122 & 183 & 244 & 305 & 367 & 428 & 489  & 550  \\
		6                   & 95  & 191 & 287 & 383 & 479 & 575 & 671 & 767  & 863  \\
		7                   & 141 & 283 & 424 & 566 & 708 & 850 & 991 & 1133 & 1275
	\end{tabular}
\end{center}

We have implemented Koelman's classification algorithm in
\texttt{RationalPolygons.jl} \cite{RationalPolygons_jl}, in particular reproducing his results
up to \( l = 42 \) and extending it up to \( l = 112 \). In total, we
have found 115\,449\,011\,813 lattice polygons with at most \( 112 \)
lattice points, 79\,140\,459 of which are internal, i.e. of the form
\( P^{(1)} \). We made the internal ones as well as all polygons up to
70 lattice points available at~\cite{springer_2024_13959996}, see
also~\cite[A371917]{oeis} for their numbers. Using
Proposition~\ref{prp:k_maximal_polygons_from_koelman}, we then
obtained the classification of \( 3 \)-maximal polygons with up to
five and \( 4 \)-maximal polygons with up to two interior lattice
points, see Theorem~\ref{thm:main_classification}. While this
approach works in principle for all pairs \( (k,i) \), it is very
inefficient compared to our direct classifications for \( i \in
\{0,1\} \) in
Sections~\ref{sec:rational_polygons_without_interior_lattice_points}
and~\ref{sec:rational_polygons_with_exactly_one_interior_lattice_points}
as well as \( k = 2 \) in Section~\ref{sec:half_integral_polygons}.
However, it did prove useful as an independent verification of our
results in these sections for low values of \( k \) and \( i \)
respectively. 

\printbibliography

@article {AKW17,
    AUTHOR = {Averkov, Gennadiy and Kr\"umpelmann, Jan and Weltge, Stefan},
     TITLE = {Notions of maximality for integral lattice-free polyhedra: the
              case of dimension three},
   JOURNAL = {Math. Oper. Res.},
  FJOURNAL = {Mathematics of Operations Research},
    VOLUME = {42},
      YEAR = {2017},
    NUMBER = {4},
     PAGES = {1035--1062},
      ISSN = {0364-765X,1526-5471},
   MRCLASS = {52B20 (52B10 52C07 90C11)},
  MRNUMBER = {3722425},
MRREVIEWER = {Uwe\ Schnell},
       DOI = {10.1287/moor.2016.0836},
       URL = {https://doi.org/10.1287/moor.2016.0836},
}

@article {Bal21,
    AUTHOR = {Balletti, Gabriele},
     TITLE = {Enumeration of lattice polytopes by their volume},
   JOURNAL = {Discrete Comput. Geom.},
  FJOURNAL = {Discrete \& Computational Geometry. An International Journal
              of Mathematics and Computer Science},
    VOLUME = {65},
      YEAR = {2021},
    NUMBER = {4},
     PAGES = {1087--1122},
      ISSN = {0179-5376,1432-0444},
   MRCLASS = {52B20 (52B11)},
  MRNUMBER = {4249896},
       DOI = {10.1007/s00454-020-00187-y},
       URL = {https://doi.org/10.1007/s00454-020-00187-y},
}

@misc{BK16,
      title={Three-dimensional lattice polytopes with two interior lattice points}, 
      author={Gabriele Balletti and Alexander M. Kasprzyk},
      year={2016},
      eprint={1612.08918},
      archivePrefix={arXiv},
      primaryClass={math.CO},
      url={https://arxiv.org/abs/1612.08918}, 
}

@phdthesis {Bat85,
    AUTHOR = {Victor V. Batyrev},
    TITLE  = {Higher-dimensional toric varieties with ample anticanonical class},
    YEAR   = {1985},
    SCHOOL = {Moscow State University},
    LANGUAGE = {Russian},
}

@article {Bat17,
    AUTHOR = {Batyrev, Victor},
     TITLE = {The stringy {E}uler number of {C}alabi-{Y}au hypersurfaces in
              toric varieties and the {M}avlyutov duality},
   JOURNAL = {Pure Appl. Math. Q.},
  FJOURNAL = {Pure and Applied Mathematics Quarterly},
    VOLUME = {13},
      YEAR = {2017},
    NUMBER = {1},
     PAGES = {1--47},
      ISSN = {1558-8599,1558-8602},
   MRCLASS = {14J33 (52B20)},
  MRNUMBER = {3858013},
MRREVIEWER = {Amin\ Gholampour},
       DOI = {10.4310/PAMQ.2017.v13.n1.a1},
       URL = {https://doi.org/10.4310/PAMQ.2017.v13.n1.a1},
}

@article {Bat23,
    AUTHOR = {Batyrev, Victor V.},
     TITLE = {Canonical models of toric hypersurfaces},
   JOURNAL = {Algebr. Geom.},
  FJOURNAL = {Algebraic Geometry},
    VOLUME = {10},
      YEAR = {2023},
    NUMBER = {4},
     PAGES = {394--431},
      ISSN = {2313-1691,2214-2584},
   MRCLASS = {14M25 (14B05 14E30 14J70)},
  MRNUMBER = {4606408},
}

@book {BS15,
    AUTHOR = {Beck, Matthias and Robins, Sinai},
     TITLE = {Computing the continuous discretely},
    SERIES = {Undergraduate Texts in Mathematics},
   EDITION = {Second},
      NOTE = {Integer-point enumeration in polyhedra,
              With illustrations by David Austin},
 PUBLISHER = {Springer, New York},
      YEAR = {2015},
     PAGES = {xx+285},
      ISBN = {978-1-4939-2968-9; 978-1-4939-2969-6},
   MRCLASS = {11P21 (05A15 05B15 11-02 11H06 52B05 52B20)},
  MRNUMBER = {3410115},
       DOI = {10.1007/978-1-4939-2969-6},
       URL = {https://doi.org/10.1007/978-1-4939-2969-6},
}

@misc{Boh23,
      title={Area bounds for planar convex bodies containing a fixed number of interior integral points}, 
      author={Martin Bohnert},
      year={2023},
      eprint={2305.11485},
      archivePrefix={arXiv},
      primaryClass={math.MG},
      url={https://arxiv.org/abs/2305.11485}, 
}

@misc{Boh24,
      title={Lattice 3-polytopes of lattice width 2 and corresponding toric hypersurfaces}, 
      author={Martin Bohnert},
      year={2024}
}

@article {BK13,
    AUTHOR = {Brown, Gavin and Kasprzyk, Alexander M.},
     TITLE = {Small polygons and toric codes},
   JOURNAL = {J. Symbolic Comput.},
  FJOURNAL = {Journal of Symbolic Computation},
    VOLUME = {51},
      YEAR = {2013},
     PAGES = {55--62},
      ISSN = {0747-7171,1095-855X},
   MRCLASS = {52B20 (14G50 94B27)},
  MRNUMBER = {3005781},
MRREVIEWER = {Hendrik\ S\"u\ss},
       DOI = {10.1016/j.jsc.2012.07.001},
       URL = {https://doi.org/10.1016/j.jsc.2012.07.001},
}

@article {Cas12,
    AUTHOR = {Castryck, Wouter},
     TITLE = {Moving out the edges of a lattice polygon},
   JOURNAL = {Discrete Comput. Geom.},
  FJOURNAL = {Discrete \& Computational Geometry. An International Journal
              of Mathematics and Computer Science},
    VOLUME = {47},
      YEAR = {2012},
    NUMBER = {3},
     PAGES = {496--518},
      ISSN = {0179-5376,1432-0444},
   MRCLASS = {52C05 (52B20)},
  MRNUMBER = {2891244},
MRREVIEWER = {Matthias\ Schymura},
       DOI = {10.1007/s00454-011-9376-2},
       URL = {https://doi.org/10.1007/s00454-011-9376-2},
}

@phdthesis {Fin83,
    AUTHOR = {J. Fine},
    TITLE  = {Resolution and Completion of Algebraic Varieties},
    YEAR   = {1983},
    SCHOOL = {University of Warwick},
}

@misc{GK13,
      title={Normal forms of convex lattice polytopes}, 
      author={Roland Grinis and Alexander Kasprzyk},
      year={2013},
      eprint={1301.6641},
      archivePrefix={arXiv},
      primaryClass={math.CO},
      url={https://arxiv.org/abs/1301.6641}, 
}

@article {HS09,
    AUTHOR = {Haase, Christian and Schicho, Josef},
     TITLE = {Lattice polygons and the number {$2i+7$}},
   JOURNAL = {Amer. Math. Monthly},
  FJOURNAL = {American Mathematical Monthly},
    VOLUME = {116},
      YEAR = {2009},
    NUMBER = {2},
     PAGES = {151--165},
      ISSN = {0002-9890,1930-0972},
   MRCLASS = {52B20 (05A15)},
  MRNUMBER = {2478059},
MRREVIEWER = {David\ Haws},
       DOI = {10.4169/193009709X469913},
       URL = {https://doi.org/10.4169/193009709X469913},
}

@misc{HHS23,
      title={Classifying log del Pezzo surfaces with torus action}, 
      author={Daniel Haettig and Juergen Hausen and Justus Springer},
      year={2023},
      eprint={2302.03095},
      archivePrefix={arXiv},
      primaryClass={math.AG},
      url={https://arxiv.org/abs/2302.03095}, 
}

@article {Kas10,
    AUTHOR = {Kasprzyk, Alexander M.},
     TITLE = {Canonical toric {F}ano threefolds},
   JOURNAL = {Canad. J. Math.},
  FJOURNAL = {Canadian Journal of Mathematics. Journal Canadien de
              Math\'ematiques},
    VOLUME = {62},
      YEAR = {2010},
    NUMBER = {6},
     PAGES = {1293--1309},
      ISSN = {0008-414X,1496-4279},
   MRCLASS = {14J45 (14J30 14M25 52B20)},
  MRNUMBER = {2760660},
MRREVIEWER = {Helena\ Ferreira\ Soares},
       DOI = {10.4153/CJM-2010-070-3},
       URL = {https://doi.org/10.4153/CJM-2010-070-3},
}

@article {KS00,
    AUTHOR = {Kreuzer, Maximilian and Skarke, Harald},
     TITLE = {Complete classification of reflexive polyhedra in four
              dimensions},
   JOURNAL = {Adv. Theor. Math. Phys.},
  FJOURNAL = {Advances in Theoretical and Mathematical Physics},
    VOLUME = {4},
      YEAR = {2000},
    NUMBER = {6},
     PAGES = {1209--1230},
      ISSN = {1095-0761,1095-0753},
   MRCLASS = {52B20 (14J32 14J45 14M25)},
  MRNUMBER = {1894855},
MRREVIEWER = {Alexander\ A.\ Borisov},
       DOI = {10.4310/ATMP.2000.v4.n6.a2},
       URL = {https://doi.org/10.4310/ATMP.2000.v4.n6.a2},
}

@article {KS04,
    AUTHOR = {Kreuzer, Maximilian and Skarke, Harald},
     TITLE = {P{ALP}: a package for analysing lattice polytopes with
              applications to toric geometry},
   JOURNAL = {Comput. Phys. Comm.},
  FJOURNAL = {Computer Physics Communications. An International Journal and
              Program Library for Computational Physics and Physical
              Chemistry},
    VOLUME = {157},
      YEAR = {2004},
    NUMBER = {1},
     PAGES = {87--106},
      ISSN = {0010-4655},
   MRCLASS = {52B20 (14M25)},
  MRNUMBER = {2033673},
       DOI = {10.1016/S0010-4655(03)00491-0},
       URL = {https://doi.org/10.1016/S0010-4655(03)00491-0},
}

@phdthesis {Koe91,
    AUTHOR = {R. J. Koelman},
    TITLE  = {The number of moduli of families of curves on toric surfaces},
    YEAR   = {1991},
    SCHOOL = {University of Nijmegen (1991)},
}

@article {Kho97,
    AUTHOR = {Khovanski\u i, A. G.},
     TITLE = {Newton polytopes, curves on toric surfaces, and inversion of
              {W}eil's theorem},
   JOURNAL = {Uspekhi Mat. Nauk},
  FJOURNAL = {Uspekhi Matematicheskikh Nauk},
    VOLUME = {52},
      YEAR = {1997},
    NUMBER = {6(318)},
     PAGES = {113--142},
      ISSN = {0042-1316,2305-2872},
   MRCLASS = {14M25 (12D10)},
  MRNUMBER = {1611333},
MRREVIEWER = {I.\ Dolgachev},
       DOI = {10.1070/RM1997v052n06ABEH002156},
       URL = {https://doi.org/10.1070/RM1997v052n06ABEH002156},
}

@article {LZ91,
    AUTHOR = {Lagarias, Jeffrey C. and Ziegler, G\"unter M.},
     TITLE = {Bounds for lattice polytopes containing a fixed number of
              interior points in a sublattice},
   JOURNAL = {Canad. J. Math.},
  FJOURNAL = {Canadian Journal of Mathematics. Journal Canadien de
              Math\'ematiques},
    VOLUME = {43},
      YEAR = {1991},
    NUMBER = {5},
     PAGES = {1022--1035},
      ISSN = {0008-414X,1496-4279},
   MRCLASS = {52C07 (11H06)},
  MRNUMBER = {1138580},
MRREVIEWER = {J.\ M.\ Wills},
       DOI = {10.4153/CJM-1991-058-4},
       URL = {https://doi.org/10.4153/CJM-1991-058-4},
}

@article {Pic99,
    AUTHOR = {G. Pick},
     TITLE = {Geometrisches zur Zahlenlehre},
   JOURNAL = {Lotos - Zeitschrift für Naturwissenschaften},
    VOLUME = {47},
      YEAR = {1899},
     PAGES = {311--319},
}

@article {Rab89,
    AUTHOR = {Rabinowitz, Stanley},
     TITLE = {A census of convex lattice polygons with at most one interior
              lattice point},
   JOURNAL = {Ars Combin.},
  FJOURNAL = {Ars Combinatoria. A Canadian Journal of Combinatorics},
    VOLUME = {28},
      YEAR = {1989},
     PAGES = {83--96},
      ISSN = {0381-7032,2817-5204},
   MRCLASS = {52B20},
  MRNUMBER = {1039134},
MRREVIEWER = {L.\ M.\ Kelly},
}

@article {Sco76,
    AUTHOR = {Scott, P. R.},
     TITLE = {On convex lattice polygons},
   JOURNAL = {Bull. Austral. Math. Soc.},
  FJOURNAL = {Bulletin of the Australian Mathematical Society},
    VOLUME = {15},
      YEAR = {1976},
    NUMBER = {3},
     PAGES = {395--399},
      ISSN = {0004-9727},
   MRCLASS = {52A10 (10E05)},
  MRNUMBER = {430960},
MRREVIEWER = {J.\ M.\ Wills},
       DOI = {10.1017/S0004972700022826},
       URL = {https://doi.org/10.1017/S0004972700022826},
}

@book {CLS11,
    AUTHOR = {Cox, David A. and Little, John B. and Schenck, Henry K.},
     TITLE = {Toric varieties},
    SERIES = {Graduate Studies in Mathematics},
    VOLUME = {124},
 PUBLISHER = {American Mathematical Society, Providence, RI},
      YEAR = {2011},
     PAGES = {xxiv+841},
      ISBN = {978-0-8218-4819-7},
   MRCLASS = {14M25 (05A15 05E45 52B12)},
  MRNUMBER = {2810322},
MRREVIEWER = {Ivan\ Arzhantsev},
       DOI = {10.1090/gsm/124},
       URL = {https://doi.org/10.1090/gsm/124},
}

@article{Gra72,
title = {An efficient algorith for determining the convex hull of a finite planar set},
journal = {Information Processing Letters},
volume = {1},
number = {4},
pages = {132-133},
year = {1972},
issn = {0020-0190},
doi = {https://doi.org/10.1016/0020-0190(72)90045-2},
url = {https://www.sciencedirect.com/science/article/pii/0020019072900452},
author = {R.L. Graham}
}

@article {KKN10,
    AUTHOR = {Kasprzyk, Alexander M. and Kreuzer, Maximilian and Nill,
              Benjamin},
     TITLE = {On the combinatorial classification of toric log del {P}ezzo
              surfaces},
   JOURNAL = {LMS J. Comput. Math.},
  FJOURNAL = {LMS Journal of Computation and Mathematics},
    VOLUME = {13},
      YEAR = {2010},
     PAGES = {33--46},
      ISSN = {1461-1570},
   MRCLASS = {14M25 (14J26 52B20)},
  MRNUMBER = {2593911},
MRREVIEWER = {Stefan\ Schr\"oer},
       DOI = {10.1112/S1461157008000387},
       URL = {https://doi.org/10.1112/S1461157008000387},
}

@misc{oeis,
    Author = {{OEIS Foundation Inc.}},
    Note = {Published electronically at \url{http://oeis.org}},
    Title = {The {O}n-{L}ine {E}ncyclopedia of {I}nteger {S}equences},
    Year = 2024}

@incollection {Rei85,
    AUTHOR = {Reid, Miles},
     TITLE = {Young person's guide to canonical singularities},
 BOOKTITLE = {Algebraic geometry, {B}owdoin, 1985 ({B}runswick, {M}aine,
              1985)},
    SERIES = {Proc. Sympos. Pure Math.},
    VOLUME = {46, Part 1},
     PAGES = {345--414},
 PUBLISHER = {Amer. Math. Soc., Providence, RI},
      YEAR = {1987},
      ISBN = {0-8218-1476-1},
   MRCLASS = {14E30 (14B05 14E05 14J10)},
  MRNUMBER = {927963},
MRREVIEWER = {Eckart\ Viehweg},
       DOI = {10.1090/pspum/046.1/927963},
       URL = {https://doi.org/10.1090/pspum/046.1/927963},
}

@incollection {Sch03,
    AUTHOR = {Schicho, Josef},
     TITLE = {Simplification of surface parametrizations---a lattice polygon
              approach},
      NOTE = {International Symposium on Symbolic and Algebraic Computation
              (ISSAC'2002) (Lille)},
   JOURNAL = {J. Symbolic Comput.},
  FJOURNAL = {Journal of Symbolic Computation},
    VOLUME = {36},
      YEAR = {2003},
    NUMBER = {3-4},
     PAGES = {535--554},
      ISSN = {0747-7171,1095-855X},
   MRCLASS = {65D17},
  MRNUMBER = {2004041},
       DOI = {10.1016/S0747-7171(03)00094-4},
       URL = {https://doi.org/10.1016/S0747-7171(03)00094-4},
}

@dataset{springer_2024_13838476,
  author       = {Springer, Justus and
                  Bohnert, Martin},
  title        = {{Lattice subpolygons of a square with given 
                   sidelength}},
  month        = sep,
  year         = 2024,
  publisher    = {Zenodo},
  version      = {1.0.0},
  doi          = {10.5281/zenodo.13838476},
  url          = {https://doi.org/10.5281/zenodo.13838476}
}

@dataset{springer_2024_13838991,
  author       = {Springer, Justus and
                  Bohnert, Martin},
  title        = {Rational polygons with no interior lattice points},
  month        = sep,
  year         = 2024,
  publisher    = {Zenodo},
  version      = {1.0.0},
  doi          = {10.5281/zenodo.13838991},
  url          = {https://doi.org/10.5281/zenodo.13838991}
}

@dataset{springer_2024_13839216,
  author       = {Springer, Justus and
                  Bohnert, Martin},
  title        = {{Rational polygons with exactly one interior 
                   lattice point}},
  month        = sep,
  year         = 2024,
  publisher    = {Zenodo},
  version      = {1.0.0},
  doi          = {10.5281/zenodo.13839216},
  url          = {https://doi.org/10.5281/zenodo.13839216}
}

@dataset{springer_2024_13928298,
  author       = {Springer, Justus and
                  Bohnert, Martin},
  title        = {{Half-integral polygons with few interior lattice 
                   points}},
  month        = oct,
  year         = 2024,
  publisher    = {Zenodo},
  version      = {1.0.0},
  doi          = {10.5281/zenodo.13928298},
  url          = {https://doi.org/10.5281/zenodo.13928298}
}

@Article{AWW11,
 Author = {Averkov, Gennadiy and Wagner, Christian and Weismantel, Robert},
 Title = {Maximal lattice-free polyhedra: finiteness and an explicit description in dimension three},
 FJournal = {Mathematics of Operations Research},
 Journal = {Math. Oper. Res.},
 ISSN = {0364-765X},
 Volume = {36},
 Number = {4},
 Pages = {721--742},
 Year = {2011},
 Language = {English},
 DOI = {10.1287/moor.1110.0510},
 zbMATH = {6038612},
 Zbl = {1246.90107}
}

@article{Julia,
author = {Bezanson, Jeff and Edelman, Alan and Karpinski, Stefan and Shah, Viral B.},
doi = {10.1137/141000671},
journal = {SIAM Review},
month = sep,
number = {1},
pages = {65--98},
title = {{Julia: A fresh approach to numerical computing}},
volume = {59},
year = {2017}
}

@software{RationalPolygons_jl,
author = {Springer, Justus},
month = oct,
title = {{RationalPolygons.jl}},
url = {https://github.com/justus-springer/RationalPolygons.jl},
version = {1.0.0},
year = {2024}
}

@dataset{springer_2024_13959996,
  author       = {Springer, Justus and
                  Bohnert, Martin},
  title        = {Lattice polygons with at most 70 lattice points},
  month        = oct,
  year         = 2024,
  publisher    = {Zenodo},
  version      = {1.1.0},
  doi          = {10.5281/zenodo.13959996},
  url          = {https://doi.org/10.5281/zenodo.13959996}
}

@dataset{springer_2024_13960791,
  author       = {Springer, Justus and
                  Bohnert, Martin},
  title        = {{3-rational polygons with few interior lattice 
                   points}},
  month        = oct,
  year         = 2024,
  publisher    = {Zenodo},
  version      = {1.0.0},
  doi          = {10.5281/zenodo.13960791},
  url          = {https://doi.org/10.5281/zenodo.13960791}
}

\end{document}